\RequirePackage{fix-cm}
\documentclass[smallextended]{svjour3}
\smartqed  % flush right qed marks, e.g. at end of proof
\usepackage{mathptmx}      % use Times fonts if available on your TeX system
\usepackage{graphicx}
\journalname{Computational Optimization and Applications}

\def\datavname{Data availability statement}%
\def\datav{\par\addvspace{17pt}\small\rmfamily
	\trivlist\if!\datavname!\item[]\else
	\item[\hskip\labelsep
	{\bfseries\datavname}]\fi}

\newenvironment{dataavailability}{\begin{datav}}
	{\end{datav}}
\def\coiname{Conflict of interest}%
\def\coi{\par\addvspace{17pt}\small\rmfamily
	\trivlist\if!\coiname!\item[]\else
	\item[\hskip\labelsep
	{\bfseries\coiname}]\fi}

\newenvironment{conflict}{\begin{coi}}
	{\end{coi}}

\usepackage{amsmath,enumitem,datetime,xfrac,fix-cm,color,multirow,siunitx}
\usepackage{amssymb,mathtools,stmaryrd,dsfont,tikz}
\usepackage{fullpage}
\usepackage[colorlinks, citecolor=blue, bookmarks=True]{hyperref}
\usetikzlibrary{positioning,trees}
\usepackage{algorithm,algpseudocode}
\graphicspath{{Figures/}}
\DeclareGraphicsExtensions{.pdf,.png}

\let\F\mathds\let\C\mathcal\newcommand{\R}{\F{R}}\newcommand{\A}{\tens{A}}
\newcommand{\norm}[1]{{\left\lVert #1 \right\rVert}}
\newcommand{\Norm}[1]{{\left\vert\kern-0.25ex\left\vert\kern-0.25ex\left\vert #1 \right\vert\kern-0.25ex\right\vert\kern-0.25ex\right\vert}}
\newcommand{\IP}[2]{\left\langle #1,#2 \right\rangle}
\newcommand{\ip}[2]{#1 \vcenter{\hbox{\resizebox{6pt}{!}{\ensuremath\cdot}}} #2}
\newcommand{\op}[1]{\operatorname{#1}}
\newcommand{\splitln}[4]{\begin{cases} #1 & #2 \\ #3 & #4\end{cases}}
\newcommand{\1}{\F{1}}
\DeclareMathOperator{\st}{\;s.t.\;}
\renewcommand{\epsilon}{\varepsilon}\renewcommand{\bar}{\overline}
\renewcommand{\hat}{\widehat}\renewcommand{\tilde}{\widetilde}
\DeclareMathOperator*{\argmin}{argmin}
\DeclareMathOperator{\supp}{supp}

\DeclareMathOperator{\sign}{sign}
\newcommand{\diff}{\mathop{}\!\mathrm{d}}
\newcommand{\Emin}[1][\var0]{\inf_{#1\in\F H}\op{E}(#1)}
\newcommand{\aU}{a_{\op{U}}}\newcommand{\aE}{a_{\op{E}}}
\newcommand{\aUaE}{$(\aU,\aE)$-minimising sequence of $\op{E}$}

\newcounter{adaptStepCounter}
\newenvironment{adaptiveStep}{\refstepcounter{adaptStepCounter}}{}

\usepackage{clipboard}\newclipboard{aFEMclipboard}
\newcommand{\data}{\eta}
\newcommand{\Domain}{\Omega}\newcommand{\domain}{\omega}

\newcommand{\meshsize}{h}

\usepackage{xstring}
\newcommand*{\var}[1]{%
	\IfEqCase{#1}{%
		{0}{u}%
		{1}{v}%
		{2}{w}%
	}[u#1]%
}
\newcommand*{\vars}[1]{%
	\IfEqCase{#1}{%
		{0}{\alpha}%
		{1}{\beta}%
		{2}{\gamma}%
		{3}{a}%
		{4}{b}%
		{5}{c}%
	}[k#1]%
}
\newcommand*{\vvars}[1]{\vec{\vars{#1}}}

\usepackage[normalem]{ulem}
\definecolor{darkgreen}{rgb}{0.0, 0.5, 0.0}
\newcommand{\edit}[3][1]{% 
	\IfEq{#1}{2}{\def\mysecondvar{}}{\def\mysecondvar{#2}}% \edit[2] hides the deleted text
	\ifx\mysecondvar\empty{}\else{%
		\IfEq{#1}{1}{% \edit[1] strikes out the deleted text
			\ifmmode%
			\text{\color{red}\sout{\ensuremath{#2}}}%
			\else%
			{\color{red} \sout{#2}}% \edit[0] colours red the deleted text
			\fi%
		}{\color{red}#2}\ %
	}\fi{\color{darkgreen}#3}}
\renewcommand{\edit}[3][1]{#3}

\author{Antonin Chambolle \and Robert Tovey }
\institute{A. Chambolle \at
	CEREMADE, CNRS \& Université Paris Dauphine, PSL Research University, Paris \\
	\email{chambolle@ceremade.dauphine.fr}, ORCID: 0000-0002-9465-4659          %  \\
	\and
	R. Tovey \at
	MOKAPLAN, INRIA Paris, Paris \\
	\email{robert.tovey@inria.fr}, ORCID: 0000-0001-5411-2268
}
\title{``FISTA'' in Banach spaces with adaptive discretisations}
\begin{document}
	\maketitle
	
	\begin{abstract}
		FISTA is a popular convex optimisation algorithm which is known to converge at an optimal rate whenever a minimiser is contained in a suitable Hilbert space. We propose a modified algorithm where each iteration is performed in a subset which is allowed to change at every iteration. Sufficient conditions are provided for guaranteed convergence, although at a reduced rate depending on the conditioning of the specific problem. These conditions have a natural interpretation when a minimiser exists in an underlying Banach space. Typical examples are L1-penalised reconstructions where we provide detailed theoretical and numerical analysis.
	\end{abstract}
	
	\section{Introduction}
	The Fast Iterative Shrinkage-Thresholding Algorithm (FISTA) was proposed by Beck and Teboulle \cite{Beck2009} as an extension of Nesterov's fast gradient method \cite{Nesterov2004} and is now a very popular algorithm for minimising the sum of two convex functions.
	We write this as the problem of computing
	\begin{equation}\label{eq: min F = fD+fC}
		\Emin \qquad\text{such that}\qquad \op{E}(\var0)\coloneqq \op{f}(\var0) + \op{g}(\var0),
	\end{equation}
	for a Hilbert space $\F H$ where $\op{f}\colon\F H\to\R$ is a convex differentiable function with $L$-Lipschitz gradient and $\op{g}\colon\F H\to\bar\R$ is a ``simple'' convex function, whose ``proximity operator'' is easy to compute. Throughout this work we assume that $\op{E}$ is bounded below so that the infimum is finite. The iterates of the FISTA algorithm will be denoted $\var0_n\in\F H$. If, moreover the infimum is achieved, it has been shown that $\op{E}(\var0_n) - \Emin$ converges at the optimum rate of $n^{-2}$ \cite{Beck2009}, and later (after a small modification) the convergence of the iterates was also shown in a general Hilbert space setting \cite{Chambolle2015}. Many further works have gone on to demonstrate faster practical convergence rates for slightly modified variants of FISTA \cite{Tao2016,Liang2017,Alamo2019}.
	
	In this work we address the case where the minimiser possibly fails to exist or lies in a larger space where $\F H$ is dense. There is much overlap between the techniques used in this work and those used in the literature of inexact optimisation, however, our interpretation is relatively novel. In particular, we emphasise the infinite-dimensional setting where errors come from ``discretisation'', rather than random or decaying errors in $\F H$, which enables two new perspectives:
	\begin{itemize}
		\item Analytically, we prove new rates of convergence for FISTA when the minimum energy is not achieved (at least not in $\F H$). The exact rate can be computed by quantifying coercivity and regularity properties of $\op{E}$. If there isn't a minimiser in $\F H$, then this rate is strictly slower than $n^{-2}$.
		
		\item Numerically, we allow the optimisation domain to change on every iteration. This enables us to understand how FISTA behaves with adaptive discretisations. Adaptive finite-element methods are known to improve the efficiency of, for example, approximating the solutions of PDEs. Our analytical results show how to combine such tools with FISTA without reducing the guaranteed rate of convergence, and our numerical results confirm much improved time and computer memory efficiency in the Lasso example (Section~\ref{sec: Lasso definition}).
	\end{itemize}
	All the examples in this work, discussed from Section~\ref{sec: general examples} onward, consider $\{\var0\in\F H\st \op{E}(\var0)<\infty\}$ to be contained in some ambient Banach space $\F{U}$. The idea is that FISTA provides a minimising sequence in $\F H\cap\F U$, but further properties like rate of convergence (of $\op{E}$ or the iterates) must come from the topology of $\F{U}$. It will not be necessary for $\F H\hookrightarrow\F U$ to be a continuous embedding, nor in fact the full inclusion $\F H\subset\F U$.
	
	Some other works for FISTA-like algorithms include \cite{Jiang2012,Villa2013}. Of particular note, our stability estimate for FISTA  in Theorem~\ref{thm: mini FISTA convergence} is very similar to \cite[Prop 2]{Schmidt2011} and \cite[Prop 3.3]{Aujol2015}. This is then used to analyse the convergence properties in our more general Banach space setting, but where all sources of inexactness come from subspace approximations. The ideas in \cite{Parpas2017} are similar although in application to the proximal gradient method with an additional smoothing on the functional $\op{g}$. The permitted refinement steps are also more broad in our work. Very recent work in \cite{Yu2021} proposes a ``Multilevel FISTA'' algorithm which allows similar coarse-to-fine refinement strategies, although only a finite number. We also allow for non-uniform refinement with a posteriori strategies.

	\subsection{Outline}
	This work is organised as follows. Section~\ref{sec: prelims} defines notation and the generic form of our proposed refining FISTA algorithm, Algorithm~\ref{alg: refining FISTA}. The main theoretical contribution of this work is the convergence analysis of Algorithm~\ref{alg: refining FISTA} which is split into two parts: first we outline the proof structure in Section~\ref{sec: recipe}, then we state the specific results in the case of FISTA in Section~\ref{sec: FISTA convergence}. The main results are Theorems~\ref{thm: exponential FISTA convergence}/\ref{thm: stronger exponential FISTA convergence} which extend the convergence of FISTA to cases with un-attained minima with uniform/adaptively chosen subspaces $\F{U}^n$ respectively. 
	
	Section~\ref{sec: general examples} presents some general results for the application of Algorithm~\ref{alg: refining FISTA} in Banach spaces and Section~\ref{sec: Lasso definition} gives a much more detailed discussion of adaptive refinement for Lasso minimisation. In particular, we describe how to choose efficient refining discretisations to approximate $\Emin$, estimate the convergence of $\op{E}$, and identify the support of the minimiser. The numerical results in Section~\ref{sec: numerics} demonstrate these techniques in four different models demonstrating the apparent sharpness of our convergence rates and the computational efficiency of adaptive discretisations.

	\section{Definitions and notation}\label{sec: prelims}
	We consider optimisation of \eqref{eq: min F = fD+fC} over a Hilbert space $(\F H,\IP\cdot\cdot,\norm\cdot)$. In the more analytical section (Sections~\ref{sec: recipe} and \ref{sec: FISTA convergence}) it will be more convenient to use the translated energy
	\begin{equation}
		\op{E}_0\colon\F H\to \R,\qquad \op{E}_0(\var0) \coloneqq \op{E}(\var0)-\Emin[\tilde{\var0}]
	\end{equation}
	so that $\inf_{\var0\in\F{H}}\op{E}_0(\var0) = 0$, although access to this function is not assumed for numerical examples.

	The proposed generalised FISTA algorithm is stated in Algorithm~\ref{alg: refining FISTA} for an arbitrary choice of closed convex subsets $\F{U}^n\subset \F H$ for $n\in\F N$. The only difference from standard FISTA is that on iteration $n$, all computations are performed in the subset $\F{U}^n$. If $\F{U}^n=\F H$, then we recover the original algorithm. More generally, the idea is that $\F{U}^n$ are ``growing'', for example $\F{U}^n\subset\F{U}^{n+1}$, but this assumption is not necessary in most of the results.
	
	Without loss of generality we will assume $L=1$, i.e. $\nabla\op{f}$ is 1-Lipschitz. To get the general statement of any of the results which follow, replace $\op{E}$ with $\frac{\op{E}}{L}$. In particular,
	\begin{equation}
		\norm{\nabla\op{f}(\var0)- \nabla\op{f}(\var1)} \leq \norm{\var0-\var1}
	\end{equation}
	for all $\var0,\var1\in\F H$ and $\op{g}$ is called ``simple'' if it is proper, convex, weakly lower-semicontinuous, and
	\begin{equation}
		\argmin_{\var0\in\tilde{\F{U}}}\tfrac12\norm{\var0-\var1}^2 + \op{g}(\var0)
	\end{equation}
	is exactly computable for all $\var1\in\F H$ and all $\tilde{\F{U}}\in\{\F{U}^n\}_{n=0}^\infty$. Closed subsets of $\F H$ are locally weakly compact, therefore this argmin is always non-empty.
	
	One defining property of the FISTA algorithm is an appropriate choice of inertia, dictated by $t_n$. In particular, we will say that $(t_n)_{n=0}^\infty$ is a \emph{FISTA stepsize} if
	\begin{equation}\label{def: t_n and rho_n}
		t_0 = 1, \qquad t_n \geq 1, \qquad \text{and}\qquad \rho_n \coloneqq t_n^2 - t_{n+1}^2 + t_{n+1} \geq 0\qquad \text{ for all }n=0,1,\ldots.
	\end{equation}
	
	The precise constants associated to a given rate are given in the statements of the theorems but, for convenience, are otherwise omitted from the text. For sequences $(a_n)_{n=0}^\infty$,$(b_n)_{n=0}^\infty$ we will use the notation:
	\begin{align*}
		a_n \lesssim b_n \qquad &\iff \qquad \exists C,N>0 \st a_n \leq C b_n \text{ for all } n>N,
		\\ a_n \simeq b_n \qquad &\iff \qquad a_n\lesssim b_n \lesssim a_n.
	\end{align*}
	For $n\in\F N$ we use the abbreviation $[n] = \{1,2,\ldots,n\}$. When the subdifferential of $\op{E}$ is set-valued, we will use the short-hand
	\begin{equation}
		\Norm{\partial\op{E}(\var0)} \coloneqq \inf_{\var1\in\partial\op{E}(\var0)}\Norm{\var1}
	\end{equation}
	for any specified norm $\Norm\cdot$.
	
	\begin{algorithm}\caption{Refining subset FISTA}\label{alg: refining FISTA}
		\centering
		\begin{algorithmic}[1]
			\State Choose $(\F{U}^n)_{n\in\F N}$, $\var0_0\in \F{U}^0$ and some FISTA stepsize choice $(t_n)_{n\in\F N}$
			\State $\var1_0\gets \var0_0, n\gets0$
			\Repeat
			\State $\bar{\var0}_n \gets (1-\tfrac{1}{t_n})\var0_n + \tfrac{1}{t_n}\var1_n$
			\State $\displaystyle \var0_{n+1} \gets \argmin_{\var0\in \F{U}^{n+1}} \tfrac12\norm{\var0-\bar{\var0}_n+\nabla \op{f}(\bar{\var0}_n)}^2+\op{g}(\var0)$ \Comment{Only modification, $\F{U}^{n+1}\subset \F{U}$}
			\State $\var1_{n+1} \gets (1-t_n)\var0_n + t_n\var0_{n+1}$
			\State $n\gets n+1$
			\Until{some stopping criterion is met}
		\end{algorithmic}
	\end{algorithm}

	\section{General proof recipe}\label{sec: recipe}
	In this section we give an intuitive outline of the full proof for convergence of Algorithm~\ref{alg: refining FISTA} before giving formal theorems and proofs in the next section. First we recall the classical FISTA convergence guarantee given by \cite[Thm 3.1]{Chambolle2015}; if there exists $\var0^*\in \argmin_{\var0\in\F H}\op{E}_0(\var0)$, then
	\begin{equation}\label{eq: classical FISTA}
		t_N^2\op{E}_0(\var0_N) + \sum^{N-1}_{n=1}\rho_n\op{E}_0(\var0_n) + \tfrac12\norm{\var1_N-\var0^*}^2 \leq \tfrac12\norm{\var0_0-\var0^*}^2
	\end{equation}
	for any FISTA stepsize choice $t_N\simeq N$ such that $\rho_n\geq0$.
	
	\paragraph{Step 1: Quantifying the stability}
	\begin{adaptiveStep}\label{step: step 1 stability}
		The first step is to generalise \eqref{eq: classical FISTA} to account for the adapting subsets $\F{U}^n$. In the notation of Algorithm~\ref{alg: refining FISTA}, Theorem~\ref{thm: mini FISTA convergence} shows that 
		\begin{equation}\label{eq: mini FISTA convergence}
			t_N^2\op{E}_0(\var0_N) + \sum_{n=1}^{N-1}\rho_n \op{E}_0(\var0_n) + \tfrac12\norm{\var1_N-\var2_N}^2 \leq \tfrac12\norm{\var0_0-\var2_0}^2 +\tfrac{\norm{\var2_N}^2-\norm{\var2_0}^2}2 
			+ \sum^N_{n=1} t_n\op{E}_0(\var2_n) + \IP{\var1_{n-1}}{\var2_{n-1}-\var2_n}
		\end{equation}
		for any $\var2_n \in \F{U}^n$. The similarities to \eqref{eq: classical FISTA} are clear. If $\F{U}^n=\F H$, then we can choose $\var2_n=\var0^*$ and the two estimates agree. These extra terms in \eqref{eq: mini FISTA convergence} quantify the robustness to changing of discretisation.
	\end{adaptiveStep}

	\paragraph{Step 2: Quantifying the scaling properties}
	\begin{adaptiveStep}\label{step: step 2 smoothness}
		To show that the extra terms in \eqref{eq: mini FISTA convergence} are small, we need to quantify the approximation properties of $\F{U}^n$. The idea is that there is a sequence $\var2_n\in\F{U}^n$, $n\in\F N$ such that $\norm{\var2_n}$ grows slowly and $\op{E}_0(\var2_n)$ decreases quickly. To quantify this balance, we introduce a secondary sequence $n_0 < n_1 < \ldots$ and constants $\aU,\aE\geq1$ such that for each $k\in\F N$
		\begin{equation}\label{eq: first regularity}
			n\leq n_k \implies \norm{\var2_n} \lesssim \aU^k, \qquad n\geq n_k \implies \op{E}_0(\var2_n) \lesssim \aE^{-k}.
		\end{equation}
		A canonical example would be $\F{U}^n=\{\var0\in\F H\st \norm{\var0}\leq \aU^k\}$ for $n\in[n_k,n_{k+1})$, then $\aE$ reflects the smoothness of $\op{E}_0$. The choice of exponential scaling is introduced to improve stability of Algorithm~\ref{alg: refining FISTA}. It is natural if we consider the $\F{U}^n$ to be the subspace of functions discretised on a uniform mesh. If that mesh is sequentially refined, then the resolution of the mesh will be of order $h^k$ after $k$ refinements and for some $h<1$. The integer $n_k$ is then the time at which the mesh has refined $k$ times.  The trade-off between $\aE$ and $\aU$ dictates the final convergence rate of the algorithm. If $\aU>1$, then we cannot guarantee the original $n^{-2}$ rate of convergence.
	\end{adaptiveStep}
	
	\paragraph{Step 3: Generalising the convergence bound}
	\begin{adaptiveStep}\label{step: step 3 convergence bound}
		In this step we combine the FISTA stability estimate with the subset approximation guarantees to provide a sharper estimate of stability with respect to the parameters $\aE$ and $\aU$. For example, if for each $k\in\F N$
		$$\var2_{n}= \var2_{n_k} \text{ for each } n=n_k,n_k+1,\ldots,n_{k+1}-1,$$
		then many terms on the right-hand side of \eqref{eq: mini FISTA convergence} telescope to 0. The result of this is presented in Lemma~\ref{thm: mini exponential FISTA convergence}. The key idea is that the stability error in \eqref{eq: mini FISTA convergence} has $K\ll N$ terms, rather than $N$.
	\end{adaptiveStep}
	
	\paragraph{Step 4: Sufficiently fast growth}
	\begin{adaptiveStep}\label{step: step 4 sufficiently fast}
		In Step~\ref{step: step 3 convergence bound} we develop a convergence bound, now we wish to show that it is only worse than the classical \eqref{eq: classical FISTA} by a constant factor. In particular, it is equivalent to either run Algorithm~\ref{alg: refining FISTA} for $N$ iterations, or the classical FISTA algorithm for $N$ iterations on the fixed subset $\F{U}^N$. The estimate from \eqref{eq: classical FISTA} provides the estimate $N^2\op{E}_0(\var0_N)\lesssim \norm{\var0_0-\var2_N}^2 = O(\aU^{2K})$ for $N\leq n_K$. Lemma~\ref{thm: sufficiently fast} shows that Algorithm~\ref{alg: refining FISTA} can achieve the same order of approximation, so long as $\F{U}^n$ grow sufficiently quickly (in particular $n_k^2\lesssim \aE^k\aU^{2k}$).
	\end{adaptiveStep}
	
	\paragraph{Step 5: Sufficiently slow growth}
	\begin{adaptiveStep}\label{step: step 5 sufficiently slow}
		The result of Step~\ref{step: step 4 sufficiently fast} is sufficient to prove convergence, but not yet a rate. If the subsets grow too quickly, then the influence of $\norm{\var0_n}\to\infty$ will slow the rate of convergence. If $n_k$ is too large, then we overfit to the discrete problem, but if $n_k$ is too small, then FISTA converges slowly. Lemma~\ref{thm: sufficiently slow} balances these two factors in an optimal way ($n_k^2\simeq \aE^k\aU^{2k}$) for Algorithm~\ref{alg: refining FISTA} resulting in a convergence rate of 
		$$ \op{E}_0(\var0_N) \lesssim \frac{\aU^{2K}}{N^2} \lesssim \frac{N^{2\kappa}}{N^2}$$
		for all $N\in\F N$ and $\kappa=\frac{2\log\aU}{\log\aE+2\log\aU}\in[0,1)$. In particular, if the minimum is attained in $\F H$, then we recover the classical rate with $\kappa=0$.
	\end{adaptiveStep}
	
	\paragraph{Step 6: Adaptivity}
	\begin{adaptiveStep}\label{step: step 6 adaptivity}
		Up to this point we have implicitly focused on the case where $\F{U}^n$ (and $n_k$) are chosen a priori. The main challenge for adaptive choice of $\F{U}^n$ is to guarantee \eqref{eq: first regularity} from Step~\ref{step: step 2 smoothness} using a posteriori estimates. Combined with the partial telescoping requirement in Step~\ref{step: step 3 convergence bound}, a natural choice is $\var2_n=\var0_{n_k-1}$ for $n\in[n_k,n_{k+1})$, i.e. the value of $n_k$ is chosen to be $n+1$ once the iterate $\var0_n$ is observed.
		Theorem~\ref{thm: stronger exponential FISTA convergence} shows that a sufficient condition is
		$$\var0_{n_k-1}\in\F{U}^{n_k}\cap\F{U}^{n_k+1}\cap\ldots\cap\F{U}^{n_{k+1}-1}, \qquad \norm{\var0_{n_k-1}} \lesssim \aU^k, \quad\text{and}\quad \op{E}_0(\var0_{n_k-1}) \lesssim \aE^{-k}.$$
		Convergence is most stable if the approximation spaces $\F{U}^n$ satisfy a monotone inclusion, breaking the monotonicity requires more care.
		The only non-trivial property to verify is the energy gap  $\op{E}_0(\var0_n) = \op{E}(\var0_n)-\Emin$. Lemma~\ref{thm: practical refinement criteria} proposes some sufficient conditions to guarantee the same overall rate of convergence as in Step~\ref{step: step 5 sufficiently slow},
		$$ \min_{n\leq N}\op{E}_0(\var0_n) \lesssim \frac{N^{2\kappa}}{N^2}$$
		for all $N\in\F N$, with the same $\kappa\in[0,1)$ from Step~\ref{step: step 5 sufficiently slow}. The penalty for accelerating the change of discretisation is a potential loss of stability or monotonicity in $\op{E}_0(\var0_n)$, although this behaviour has not been seen in numerical experiments.
	\end{adaptiveStep}
	
	\section{Proof of convergence}\label{sec: FISTA convergence}
	In this section we follow the recipe motivated in Section~\ref{sec: recipe} to prove convergence of two variants of Algorithm~\ref{alg: refining FISTA}. Each of the main theorems and lemmas will be stated with a sketch proof in this section. The details of the proofs are either trivial or very technical and are therefore placed in Section~\ref{app: FISTA convergence} to preserve the flow of the argument.
	
	\subsection{Computing the convergence bound}
	For Step~\ref{step: step 1 stability} of Section~\ref{sec: recipe} we look to replicate the classical bound of the form in \eqref{eq: classical FISTA} for Algorithm~\ref{alg: refining FISTA}. The proofs in this step follow the classical arguments \cite{Beck2009,Chambolle2015} very closely. 
	Throughout this section we consider a sequence $(\F{U}^n)_{n\in\F N}$ which generate the iterates $(\var0_n)_{n\in\F N}$ in Algorithm~\ref{alg: refining FISTA} such that
	\begin{equation}\label{eq: refining subspace definition}
		\var0^n\in\F{U}^{n+1}\subset\F H \quad\text{where }\F{U}^n\text{ is a closed, convex subset for all } n\in\F N.
	\end{equation}
	
	\subsubsection{Single iterations}
	We first wish to understand a single iteration of Algorithm~\ref{alg: refining FISTA}. This is done through the following two lemmas.
	\begin{lemma}[equivalent to {\cite[Lemma 3.1]{Chambolle2015}}]\label{thm: descent lemma}
		\Copy{thm: descent lemma}{
			Suppose $\nabla \op{f}$ is 1-Lipschitz, for any $\bar{\var0}\in \F H$ define
			$$ \var0 \coloneqq \argmin_{\var0\in \F{U}^{n}} \tfrac12\norm{\var0-\bar{\var0}+\nabla \op{f}(\bar{\var0})}^2+\op{g}(\var0).$$
			Then, for all $\var2\in\F{U}^n$, we have
			$$\op{E}_0(\var0) + \tfrac12\norm{\var0-\var2}^2 \leq \op{E}_0(\var2) + \tfrac12\norm{\bar{\var0}-\var2}^2.$$
			
		}
	\end{lemma}
	\noindent The proof is exactly the same as in \cite{Chambolle2015} on the subset $\F{U}^n$. Applying Lemma~\ref{thm: descent lemma} to the iterates from Algorithm~\ref{alg: refining FISTA} gives a more explicit inequality.
	
	\begin{lemma}[{\cite[(17)]{Chambolle2015}, \cite[Lemma 4.1]{Beck2009}}]\label{thm: one step FISTA}
		\Copy{thm: one step FISTA}{
			Let $\var2_n\in\F{U}^n$ be chosen arbitrarily and $\var0_n$/$\var1_n$ be generated by Algorithm~\ref{alg: refining FISTA} for all $n\in\F N$.
			For all $n>0$, it holds that}
		\begin{equation}\Copy{thm:eq: one step FISTA}{
				t_{n}^2(\op{E}_0(\var0_{n}) - \op{E}_0(\var2_n)) - (t_{n}^2-t_{n})(\op{E}_0(\var0_{n-1})-\op{E}_0(\var2_n)) \leq \tfrac1{2}\left[\norm{\var1_{n-1}}^2-\norm{\var1_{n}}^2\right] + \IP{\var1_{n}-\var1_{n-1}}{\var2_n}.}\label{eq: one-step FISTA}
		\end{equation}
	\end{lemma}
	\noindent The proof is given in Theorem~\ref{app:thm: one step FISTA} and is a result of the convexity of $\op{E}_0$ and $\F{U}_n$ for a well chosen $\var2$ in Lemma~\ref{thm: descent lemma}. 
	
	\subsubsection{Generic convergence bound}
	Lemma~\ref{thm: one step FISTA} gives us an understanding of a single iteration of Algorithm~\ref{alg: refining FISTA}, summing over $n$ then gives our generic convergence bound for any variant of Algorithm~\ref{alg: refining FISTA}.
	
	\begin{theorem}[analogous to {\cite[Thm 3.2]{Chambolle2015}, \cite[Thm 4.1]{Beck2009}}]\label{thm: mini FISTA convergence}
		\Copy{thm: mini FISTA convergence}{
			Fix a sequence of subsets $(\F{U}^n)_{n\in\F N}$ satisfying \eqref{eq: refining subspace definition}, arbitrary $\var0_0\in \F{U}^0$, and FISTA stepsize choice $(t_n)_{n\in\F N}$. Let $\var0_n$ and $\var1_n$ be generated by Algorithm~\ref{alg: refining FISTA}, then, for any choice of $\var2_n\in\F{U}^n$ and $N\in\F N$ we have
		}
		\begin{equation}\label{eq: FISTA inequality}\Copy{thm:eq: mini FISTA convergence}{
				t_N^2\op{E}_0(\var0_N) + \sum_{n=1}^{N-1}\rho_n \op{E}_0(\var0_n) + \frac{\norm{\var1_N-\var2_N}^2}{2} \leq \frac{\norm{\var0_0-\var2_0}^2-\norm{\var2_0}^2+\norm{\var2_N}^2}{2}
				+ \sum^N_{n=1} t_n\op{E}_0(\var2_n) + \IP{\var1_{n-1}}{\var2_{n-1}-\var2_n}.
		}\end{equation}
	\end{theorem}
	\noindent The proof is given in Theorem~\ref{app:thm: mini FISTA convergence}. This result is the key approximation for showing convergence of FISTA with changing subsets. In the classical setting, we have $\F{U}^n=\F H$, $\var2_n=\var2_0\in\argmin_{\var0\in\F H}\op{E}_0(\var0)$ and the extra terms on the right-hand side collapse to 0. 
	
	If there exists a minimiser $\var0^*\in\argmin_{\var0\in\F H}\op{E}_0(\var0)$, then the natural choice in \eqref{eq: FISTA inequality} is $\var2_n=\tens{\Pi}_n\var0^*$ for some projection $\tens{\Pi}_n\colon\F H\to\F{U}^n$, however, there are simple counter-examples which give $\op{E}_0(\tens{\Pi}_n\var0^*)=\infty$ and so this inequality becomes useless. For example, if $\op{f}(\var0) = \norm{\var0}_{L^2([0,1])}^2$, $\op{g}$ is the indicator on the set $\F D=\{\var0\in L^1([0,1])\st \var0(x) \geq x\}$, and $\tens{\Pi}_n$ is the $L^2$ projection onto a set of piecewise constant functions, then $\var0^*=x\mapsto x$. On the other hand, suppose one of the pixels of the discretisation is $[x_0-\meshsize,x_0+\meshsize]$, then 
	$$\tens{\Pi}_n\var0^*\left(x_0+\tfrac\meshsize2\right) = \argmin_{\vars5\in\R}\int_{x_0-\meshsize}^{x_0+\meshsize}(\var0^*(x)-\vars5)^2\diff x = \argmin_{\vars5\in\R}\int_{x_0-\meshsize}^{x_0+\meshsize}(x-\vars5)^2\diff x = x_0 < x_0+\tfrac\meshsize2.$$
	In particular $\tens{\Pi}_n\var0^* \notin\F D$ therefore $\op{E}_0(\tens{\Pi}_n\var0^*) = \infty$. The choice $\var2_n=\argmin_{\var0\in\F{U}^n}\op{E}_0(\var0)$ is much more robust and allows us to apply Algorithm~\ref{alg: refining FISTA} more broadly. The penalty for this flexibility is a more complicated analysis; each time the subset changes, because $\var1_n\in\F{U}_n$, the system receives a ``shock'' proportional to $\norm{\var1_n}\norm{\var2_{n}-\tens{\Pi}_n\var2_{n+1}}$.

	\subsection{Convergence bound with milestones}
	In standard FISTA, the right-hand side of \eqref{eq: FISTA inequality} is a constant. The following lemma minimises the growth of the ``constant'' as a function of $N$ by partially telescoping the sum on the right-hand side. Before progressing to the content of Step~\ref{step: step 3 convergence bound}, we will first formalise the definition of the constants $\aU$ and $\aE$ introduced in Step~\ref{step: step 2 smoothness}.
	\begin{definition}\label{def: exp subspaces}
		Fix $\aU,\aE\geq 1$ and a sequence $\tilde{\var2}_k\in\F H$. We say that $(\tilde{\var2}_k)_{k\in\F N}$ is an \emph{\aUaE\ } if 
		$$ \norm{\tilde{\var2}_k}\lesssim \aU^k \qquad\text{and}\qquad \op{E}_0(\tilde{\var2}_k) \lesssim \aE^{-k} $$
		for all $k\in\F N$.
	\end{definition}
	In this section we will simply assume that such sequences exist and in Section~\ref{sec: general examples} we will give some more general examples. 
	
	\begin{lemma}\label{thm: mini exponential FISTA convergence}
		\Copy{thm: mini exponential FISTA convergence}{
			Let $\var0_n$, $\var1_n$ be generated by Algorithm~\ref{alg: refining FISTA} with $(\F{U}^n)_{n\in\F N}$ satisfying \eqref{eq: refining subspace definition}, $(n_k\in\F N)_{k\in\F N}$ be a monotone increasing sequence, and choose
			$$\tilde{\var2}_k \in \F{U}^{n_k} \cap \F{U}^{n_k+1}\cap \ldots \cap \F{U}^{n_{k+1}-1}$$
			for each $k\in\F N$. If such a sequence exists, then for all $K\in\F N$, $n_K\leq N<n_{K+1}$ we have
		}
		\begin{multline*}\Copy{thm:eq: mini exponential FISTA convergence}{
				t_N^2\op{E}_0(\var0_N) + \sum_{n=1}^{N-1}\rho_n \op{E}_0(\var0_n) + \frac{\norm{\var1_N-\tilde{\var2}_K}^2}{2} \leq C + \frac{\norm{\tilde{\var2}_K}^2}{2} + \frac{(N+1)^2-n_K^2}{2}\op{E}_0(\tilde{\var2}_K) \\+ \sum_{k=1}^K \frac{n_k^2-n_{k-1}^2}{2}\op{E}_0(\tilde{\var2}_{k-1}) + \IP{\var1_{n_k-1}}{\tilde{\var2}_{k-1}-\tilde{\var2}_k}}
		\end{multline*}
		\Copy{thm:end: mini exponential FISTA convergence}{
			where $C = \frac{\norm{\var0_0-\tilde{\var2}_0}^2-\norm{\tilde{\var2}_0}^2}{2}$.
		}
	\end{lemma}
	
	\noindent The proof is given in Lemma~\ref{app:thm: mini exponential FISTA convergence}. The introduction of $n_k$ has greatly compressed the expression of Theorem~\ref{thm: mini FISTA convergence}. On the right-hand side, we now only consider $\op{E}_0$ evaluated on the sequence $\tilde{\var2}_k$ and there are $K$ elements to the sum rather than $N$.
	
	\subsection{Refinement without overfitting}
	The aim of Step~\ref{step: step 4 sufficiently fast} is to show that $n$ iterations of Algorithm~\ref{alg: refining FISTA} is no slower (up to a constant factor) than $n$ iterations of classical FISTA on the space $\F{U}^n$. In other words, we would like to ensure that
	\begin{equation}
		\op{E}_0(\var0_n) = \op{E}(\var0_n) - \min_{\var0\in\F{U}^n}\op{E}(\var0) \lesssim \frac{\norm{\var0_0-\tilde{\var2}_k}^2}{n^2}
	\end{equation}
	uniformly for $n\in[n_k,n_{k+1})$. If this condition is not satisfied, then it indicates that computational effort has been wasted by a poor choice of subsets. This can be interpreted as an overfitting to the discretisation of $\op{E}_0|_{\F{U}^n}$ rather than the desired function $\op{E}_0|_{\F H}$. Combining the assumptions given by Definition~\ref{def: exp subspaces} and the result of Lemma~\ref{thm: mini exponential FISTA convergence}, the following lemma proves the convergence of Algorithm~\ref{alg: refining FISTA} provided that the refinement times $n_k$ are sufficiently small (i.e. $\F{U}^n$ grows sufficiently quickly).
	
	\begin{lemma}\label{thm: sufficiently fast}
		\Copy{thm: sufficiently fast}{
			Suppose $\F{U}^n,\ \var0_n,\ \var1_n$ and $n_k$ satisfy the conditions of Lemma~\ref{thm: mini exponential FISTA convergence} and $(\tilde{\var2}_k)_{k\in\F N}$ forms an \\\aUaE\  with
			$$\tilde{\var2}_k\in \F{U}^{n_k}\cap\F{U}^{n_k+1}\cap\ldots\cap\F{U}^{n_{k+1}-1}.$$
			If either:
			\begin{itemize}
				\item $\aU>1$ and $n_k^2\lesssim \aE^k\aU^{2k}$,
				\item or $\aU=1$, $\sum_{k=1}^\infty n_k^2\aE^{-k}<\infty$, and $\sum_{k=1}^\infty\norm{\tilde{\var2}_k-\tilde{\var2}_{k+1}}<\infty,$
			\end{itemize}
			then
			$$ \op{E}_0(\var0_N) \lesssim \frac{\aU^{2K}}{N^2}\qquad \text{for all}\qquad n_K\leq N< n_{K+1}.$$
		}
	\end{lemma}
	
	\noindent The proof is given in Lemma~\ref{app:thm: sufficiently fast}. We make two observations of the optimality of Lemma~\ref{thm: sufficiently fast}:
	\begin{itemize}
		\item The convergence guarantee for $N\in[n_K,n_{K+1})$ iterations of classical FISTA in the space $\F{U}^N$ is 
		$$ \op{E}_0(\var0_N) \lesssim \frac{\norm{\var0_0-\tilde{\var2}_K}^2}{N^2} + \min_{\var0\in\F{U}^N}\op{E}_0(\var0) \lesssim \frac{\aU^{2K}}{N^2} + \aE^{-K}.$$
		This is equivalent to Lemma~\ref{thm: sufficiently fast} after the assumptions on $n_k$.
		\item If $\F H$ is finite dimensional, then the condition $\aU=1$ is almost trivially satisfied. Norms in finite dimensions are equivalent and any discretisation can be achieved with a finite number of refinements (i.e. the sums over $k$ are finite). 
	\end{itemize}
	
	\subsection{Convergence rate}
	In Lemma~\ref{thm: sufficiently fast} we show that $\op{E}_0(\var0_n)$ converges at a rate depending on $k$ and $n$, so long as $k$ grows sufficiently quickly. On the other hand, as $k$ grows, the rate becomes worse and so we need to also put a lower limit on the growth of $n_k$. The following lemma completes Step~\ref{step: step 5 sufficiently slow} by computing the global convergence rate of $\op{E}_0(\var0_n)$ when $k$ grows at the minimum rate which is consistent with Lemma~\ref{thm: sufficiently fast}.
	
	As a special case, note that if $\aU=1$ then Lemma~\ref{thm: sufficiently fast} already gives the optimal $O(N^{-2})$ convergence rate. This is in fact a special case of that shown in \cite[Prop 3.3]{Aujol2015}. If the minimum is achieved in $\F H$, then it is not possible to refine ``too quickly'' and the following lemma is not needed.
	
	\begin{lemma}\label{thm: sufficiently slow}
		\Copy{thm: sufficiently slow}{
			Suppose $\var0_n$ and $n_k$ are sequences satisfying 
			$$\forall N\in[n_K, n_{K+1}),\ \op{E}_0(\var0_N) \lesssim \frac{\aU^{2K}}{N^2} \qquad\text{where} \qquad  n_K^2\gtrsim \aE^K\aU^{2K},$$
			then 
			$$\op{E}_0(\var0_N) \lesssim \frac{1}{N^{2(1-\kappa)}}\qquad \text{ where } \qquad \kappa = \frac{\log \aU^2}{\log \aE + \log \aU^2}.$$}
	\end{lemma}
	\noindent The proof is given in Lemma~\ref{app:thm: sufficiently slow}.
	
	\subsubsection{FISTA convergence with a priori discretisation}
	We can summarise Lemmas~\ref{thm: mini exponential FISTA convergence} to~\ref{thm: sufficiently slow} into a single theorem stating the convergence guarantees when $\F{U}^n$ and $n_k$ are chosen a priori.
	\begin{theorem}\label{thm: exponential FISTA convergence}
		\Copy{thm: exponential FISTA convergence}{
			Let $(\tilde{\var2}_k)_{k\in\F N}$ be an \aUaE\  and choose any $\F{U}^n$ satisfying \eqref{eq: refining subspace definition} such that 
			$$ \tilde{\var2}_k \in \F{U}^{n_k}\cap\F{U}^{n_k+1} \cap\ldots\cap \F{U}^{n_{k+1}-1}$$
			for all $k\in\F N$. Compute $\var0_n$ and $\var1_n$ by Algorithm~\ref{alg: refining FISTA}.
			
			\bigbreak\noindent Suppose that either:
			\begin{itemize}
				\item $\aU>1$ and $n_k^2\simeq \aE^k\aU^{2k}$, or
				\item $\aU=1$, $\sum_{k=1}^\infty n_k^2\aE^{-k}<\infty$ and $\sum_{k=1}^\infty\norm{\tilde{\var2}_k-\tilde{\var2}_{k+1}}<\infty$,
			\end{itemize}
			then
			$$\op{E}_0(\var0_N) \lesssim \frac{1}{N^{2(1-\kappa)}}\qquad \text{ where } \qquad \kappa = \frac{\log \aU^2}{\log \aE + \log \aU^2} \qquad \text{uniformly for } N\in\F N.$$
		}
	\end{theorem}
	\noindent Analytically, this theorem gives new rates of convergence for FISTA when the minimiser is not achieved in $\F H$. Indeed for the original algorithm ($\F{U}^n=\F H$), if $\var0_0=0$ for simplicity and $(\tilde{\var2}_k)_{k\in\F N}$ is \emph{any} \aUaE\ exists, the result of Lemma~\ref{thm: mini exponential FISTA convergence} is
		\begin{equation}\label{eq: slow exact FISTA}
			\op{E}_0(\var0_N) \leq \inf_{\var2\in\F H} \frac{\norm{\var2}^2+N^2\op{E}_0(\var2)}{2t_N^2} \leq \min_{k\in\F N} \frac{\norm{\tilde{\var2}_k}^2+N^2\op{E}_0(\tilde{\var2}_k)}{2t_N^2} \lesssim \min_{k\in\F N} \frac{\aU^{2k}+N^2\aE^{-k}}{N^2}\lesssim N^{-2(1-\kappa)}.
		\end{equation}
		In this sense, we could say that $\op{E}_0$ converges at the rate $N^{-2(1-\kappa)}$ if and only if such a sequence exists. Nothing is lost (or gained) analytically by choosing $\F{U}_n\subsetneq\F H$.
	
	Numerically, it is easy to implement the strategy of Theorem~\ref{thm: exponential FISTA convergence} and requires very little knowledge of how to estimate $\op{E}_0(\var0_n)$. So long as $\aU$ and $\aE$ can be computed analytically, one can choose $\tilde{\var2}_k$ implicitly to be the discrete minimisers of some ``uniform'' discretisations (e.g. $\F{U}^n=\{\norm{\var0}\leq k\}$ or finite element spaces with uniform mesh) to achieve the stated convergence rate.

	\subsubsection{FISTA convergence with adaptivity}
	There are two properties of the sequence $(\F{U}^n)_{n\in\F N}$ which we may wish to decide adaptively: the refinement times $n_k$ and the discretising spaces $\{\F{U}^n \st n_k\leq n<n_{k+1}\}$. We will refer to these as temporal and spatial adaptivity respectively.
	
	Lemma~\ref{thm: sufficiently fast} gives a sufficient condition on $n_k$ for converging at the rate $O(N^{2(\kappa-1)})$, but it is not necessary. Indeed for $n\leq n_k$ we have
	$$\op{E}_0(\var0_n) \geq \min_{\var0\in\F{U}^n}\op{E}_0(\var0) = O(\aE^{-k}) = O(n^{2(\kappa-1)}),$$
	which suggests that to converge faster than $n^{2(\kappa-1)}$ requires choosing smaller $n_k$. As an example, in Section~\ref{sec: wavelet examples} we will see Algorithm~\ref{alg: refining FISTA} can converge at a near-linear rate, although this is not possible without adaptive refinement times. On the other hand, choice of spatial adaptivity has no impact on rate but can impact computational efficiency. It will be permitted to use greedy discretisation techniques so long as it is sufficient to estimate $\op{E}_0(\var0_n)$ accurately.
	
	Theorem~\ref{thm: exponential FISTA convergence} already allows for spatial adaptivity, so we focus on temporal adaptivity. Lemma~\ref{thm: sufficiently fast} suggests that a good refinement time strategy is to choose $n_k$ to be the minimal integer such that $\op{E}_0(\var0_{n_k-1})\lesssim \aE^{-k}$. However, the value of $\op{E}_0$ may be hard to estimate and so we retain a ``backstop'' condition which guarantees that convergence is no slower than the rate given by Theorem~\ref{thm: exponential FISTA convergence}. In the non-classical case of $\aU>1$, we provide the following theorem.
	\begin{theorem}\label{thm: stronger exponential FISTA convergence}
		\Copy{thm: stronger exponential FISTA convergence}{
			Let $(\F{U}^n\subset\F H)_{n\in\F N}$ be a sequence of subsets satisfying \eqref{eq: refining subspace definition}, compute $\var0_n$ and $\var1_n$ by Algorithm~\ref{alg: refining FISTA}. Suppose that there exists a monotone increasing sequence $n_k\in\F N$ such that
			$$ \tilde{\var2}_k\coloneqq \var0_{n_k-1}\in\F{U}^{n_k}\cap\F{U}^{n_k+1} \cap\ldots\cap \F{U}^{n_{k+1}-1}$$
			for all $k\in\F N$.
			
			If $(\tilde{\var2}_k)_{k\in\F N}$ is an \aUaE\ with $\aU>1$ and $n_k^2\lesssim \aE^k\aU^{2k}$, then
			$$\min_{n\leq N}\op{E}_0(\var0_n) = \min_{n\leq N} \op{E}(\var0_n) - \Emin \lesssim \frac{1}{N^{2(1-\kappa)}}\qquad \text{ where } \qquad \kappa = \frac{\log \aU^2}{\log \aE + \log \aU^2}$$
			uniformly for $N\in\F N$.
		}
	\end{theorem}
	\noindent The proof is given in Theorem~\ref{app:thm: stronger exponential FISTA convergence}. If we directly compare Theorems~\ref{thm: exponential FISTA convergence} and~\ref{thm: stronger exponential FISTA convergence}, both are a direct result of Lemma~\ref{thm: sufficiently fast} assuming a specific choice of $n_k$ or $\tilde{\var2}_k$ respectively. We note that the convergence rate is the same in both theorems but the price for better adaptivity (i.e. only an upper bound on $n_k$) is a slightly weaker stability guarantee (now convergence of $\min_{n\leq N}\op{E}_0(\var0_n)$). In Theorem~\ref{thm: exponential FISTA convergence}, as in the original FISTA algorithm, the sequence $\op{E}_0(\var0_n)$ is not monotone but the magnitude of oscillation is guaranteed to decay in time. This behaviour is lost in Theorem~\ref{thm: stronger exponential FISTA convergence}. Although we do not prove it here, it can be shown that the stronger condition 
	\begin{equation}
		\tilde{\var2}_k \in \F{U}^{n_k}\cap\F{U}^{n_k+1} \cap\ldots\cap \F{U}^{n_{k+1}-1}\cap\ldots\cap \F{U}^N
	\end{equation}
	is sufficient to restore the stronger last-iterate guarantee on $\op{E}_0(\var0_N)$. Again, monotonicity of $\F{U}^n$ corresponds with improved stability of Algorithm~\ref{alg: refining FISTA}.
	
	To enable a more practical implementation of Theorem~\ref{thm: stronger exponential FISTA convergence}, the following lemma describes several refinement strategies which provide sufficient condition for $\op{E}_0(\tilde{\var2}_k) \lesssim \aE^{-k}$.
	\begin{lemma}\label{thm: practical refinement criteria}
		\Copy{thm: practical refinement criteria}{
			Let $(\tilde{\var2}_k)_{k\in\F N}$ be a sequence in $\F H$ with $\norm{\tilde{\var2}_k}\lesssim \aU^k$. Suppose $\tilde{\var2}_k\in\tilde{\F U}^k\coloneqq\F{U}^{n_k}$ and denote $\op{E}_0(\tilde{\F{U}}^k) \coloneqq \inf_{\var0\in\tilde{\F U}^k}\op{E}_0(\var0)$. Any of the following conditions are sufficient to show that $\tilde{\var2}_k$ is an \aUaE:
			\begin{enumerate}
				\item Small continuous gap refinement: $ \op{E}_0(\tilde{\var2}_k)\leq \vars1\aE^{-k}$ for all $k\in\F N$, some $\vars1>0$.
				\item Small discrete gap refinement: $\op{E}_0(\tilde{\F{U}}^k)\leq \vars1\aE^{-k}$ and $ \op{E}_0(\tilde{\var2}_k)-\op{E}_0(\tilde{\F{U}}^{k-1})\leq \vars1\aE^{-k}$ for all $k>0$, some $\vars1>0$.
			\end{enumerate}
			Otherwise, suppose there exists a Banach space $(\F{U},\Norm\cdot)$ which contains each $\tilde{\F{U}}^k$, $\sup_{k\in\F N}\Norm{\tilde{\var2}_k}<\infty$, and the sublevel sets of $\op{E}$ are $\Norm\cdot$-bounded. With the subdifferential $\partial \op{E}\colon\F U\rightrightarrows\F U^*$,  it is also sufficient if either:
			\begin{enumerate}\setcounter{enumi}{2}
				\item Small continuous gradient refinement: $\sup_{\var0\in\F{U}} \inf_{\var1\in\partial\op{E}(\tilde{\var2}_k)} \frac{|\IP{\var1}{\var0}|}{\Norm{\var0}}\leq \vars1\aE^{-k}$ for all $k\in\F N$, some $\vars1>0$.
				\item Small discrete gradient refinement: $\op{E}_0(\tilde{\F{U}}^k)\leq \vars1\aE^{-k}$ and $\sup_{\var0,\tilde{\var2}\in \tilde{\F{U}}^k}\inf_{\var1\in\F V^k}\frac{|\IP{v}{\var0-\tilde{\var2}}|}{\Norm{\var0-\tilde{\var2}}}\leq \vars1\aE^{-k}$ for all $k\in\F N$, some $\vars1>0$, where $\F V^k\coloneqq \partial(\op{E}|_{\tilde{\F{U}}^k})(\tilde{\var2}_k)$.
			\end{enumerate}
		}
	\end{lemma}
	
	\noindent The proof is given in Lemma~\ref{app:thm: practical refinement criteria}. The refinement criteria described by Lemma~\ref{thm: practical refinement criteria} can be split into two groups. Cases (1) and (3) justify that any choice of $\F{U}^{n_k}$ satisfies the required conditions, so long as $\tilde{\var2}_k\in\F{U}^{n_k}$. In cases (2) and (4), $\tilde{\var2}_k$ is sufficient to choose the refinement time $n_k$, but an apriori bound is required on $\op{E}_0(\tilde{\F{U}}^k)$. In these cases one could, for example, choose $\tilde{\F{U}}^k$ to be a uniform discretisation with a priori estimates.
	
	Another splitting of the criteria is into gap and gradient computations. Typically, gradient norms (in (4) and (5)) should be easier to estimate than function gaps because they only require local knowledge rather than global, i.e. $\partial\op{E}(\var0_n)$ rather than an estimate of $\Emin$. Implicitly, the global information comes from an extra condition on $\op{E}$ to assert that sublevel sets are bounded.
	
	%%%%%%%%%%%%%%%%%%%%%%%%%%%%%%%%%%%%%%%%%%%%%%%%%%
	%%%%%%%%%%%%%%%%%%%%%%%%%%%%%%%%%%%%%%%%%%%%%%%%%%
	%%%%%%%%%%%%%%%%%%%%%%%%%%%%%%%%%%%%%%%%%%%%%%%%%%
	\section{General examples}\label{sec: general examples}
	We consider the main use of Algorithm~\ref{alg: refining FISTA} to be where there exists a Banach space $(\F{U},\Norm\cdot)$ such that $\F{U}\supset\{\var0\in\F{H}\st\op{E}(\var0)<\infty\}$ and
	$$\inf_{\var0\in\F H} \op{E}(\var0) = \min_{\var0\in\F{U}}\op{E}(\var0) = \op{E}(\var0^*)$$
	for some $\var0^*\in\F{U}$. The cases where $\F H$ has finite dimension or is separable are more straightforward; if the total number of refinements is finite (i.e. $\F{U}^n = \F{U}^N$ for all $n\geq N$, some $N\in\F N$), then $\aU = 1$. This holds for most finite dimensional problems as well as the countable example discussed in detail in Section~\ref{sec: Lasso definition}. In this section we give explicit computations of $\aU$ and $\aE$ in the setting where $\F H = L^2(\Domain)$ for some domain $\Domain\subset\R^d$ and the subsets $\F{U}^n$ will be finite dimensional finite-element--like spaces, as defined below.
	
	\begin{definition}\label{def: finite element space}
		Suppose $\norm\cdot_q\lesssim \Norm\cdot$ (i.e. $\F{U}\subset L^q(\Domain)$) for some $q\in[1,\infty]$ and connected, bounded, measurable domain $\Domain\subset\R^d$. We say that a collection $\F M$ is a \emph{mesh} if 
		$$\bigcup_{\domain\in\F M}\domain \supset  \Domain \qquad\text{and} \qquad |\domain\cap\domain'| = 0\qquad \text{for all $\domain,\domain'\in\F M,\ \domain\neq\domain'$.}$$
		Furthermore, we say a sequence of meshes $(\F M^k)_{k\in\F N}$ is \emph{consistent} if there exists $\domain_0\subset\Domain$ such that
		$$ \forall\domain\in\F{M}^k \quad \exists (\vars0_\domain,\vvars1_\domain)\in\R^{d\times d}\times\R^d \quad\text{such that}\quad \vec{x}\in\domain_0\iff \vars0_\domain\vec{x}+\vvars1_\domain\in\domain. $$
		Fix $\meshsize\in(0,1)$, linear subspaces $\tilde{\F{U}}^k\subset\F{H}$, and consistent meshes $\F M^k$. We say that the sequence $(\tilde{\F{U}}^k)_{k\in\F N}$ is an \emph{$\meshsize$-refining sequence of finite element spaces} if there exists $c_{\vars0}>0$ such that:
		$$ \forall (\tilde{\var0},\domain)\in \tilde{\F{U}}^k\times\F{M}^k,\quad \op{det}(\vars0_\domain)\geq c_{\vars0}\meshsize^{kd} \quad\text{and}\quad \exists \var0\in\tilde{\F U}^0 \quad\text{such that}\quad \forall\vec{x}\in\domain_0,\ \var0(\vec{x})=\tilde{\var0}(\vars0_\domain\vec{x}+\vvars1_\domain).$$
		We say that $(\tilde{\F{U}}^k)_{k\in\F N}$ is \emph{of order $p$} if for any $\var0^*\in\argmin_{\var0\in\F U}\op{E}(\var0)$ there exists a sequence $(\tilde{\var2}_k)_{k\in\F N}$ such that
		\begin{equation}\label{eq: projection of u*}
			\forall k\in\F N, \qquad\tilde{\var2}_k\in\tilde{\F U}^k  \quad\text{and}\quad \Norm{\tilde{\var2}_k-\var0^*}\lesssim_{\var0^*} \meshsize^{kp}.
		\end{equation}
		
		We allow the implicit constant to have any dependence on $\var0^*$ so long as it is finite. For example, in the case of Sobolev spaces we would expect an inequality of the form $\norm{\tilde{\var2}_k-\var0^*}_{W^{0,2}} \lesssim \meshsize^{kp}\norm{\var0^*}_{W^{p,2}}$ \cite{Strang1972}.
	\end{definition}
	\begin{remark}
		To clarify this definition with an example, suppose we wish to approximate $L^q(\Domain)$ with piecewise linear finite elements with a triangulated mesh. Then, $\domain_0\subset\Domain$ is a single triangle of diameter $O(\meshsize)$ and all meshes $\F M^k$ must be triangulations of $\Domain$ with cell volumes scaling no faster than $O(\meshsize^{kd})$. The function $\var0$ from the $\meshsize$-refining property is an arbitrary linear element, so that each $\var0\in\tilde{\F U}^k$ is linear on each $\domain\in\F M^k$, which leads to an order $p=2$ if $\var0^*\in W^{1,2}(\Domain)$.
	\end{remark}
	
	We note that any piecewise polynomial finite element (or spline) space can be used to form a $\meshsize$-refining sequence of subspaces. Wavelets with a compactly supported basis behave like a multi-resolution finite element space as there is always overlap in the supports of basis vectors. Similarly, a Fourier basis does satisfy the scaling properties, but each basis vector has global support. Both of these exceptions are important and could be accounted for with further analysis but we focus on the more standard finite element case. 
	In order to align these discretisation properties with the assumptions of Theorems~\ref{thm: exponential FISTA convergence} and \ref{thm: stronger exponential FISTA convergence}, we make the following observation.
	\begin{lemma}\label{thm: pq to p'q'}
		Fix $\var0^*\in\argmin_{\var0\in\F U}\op{E}(\var0)$ and $p',q'>0$. If a sequence $\tilde{\var2}_k\in\F H$ satisfies
		$$ \norm{\tilde{\var2}_k}\lesssim \meshsize^{-kq'} \quad\text{and}\quad\op{E}(\tilde{\var2}_k) - \op{E}(\var0^*)\lesssim \meshsize^{kp'},$$
		then $(\tilde{\var2}_k)_{k\in\F N}$ is an \aUaE\ for $\aU=\meshsize^{-q'}$ and $\aE=\meshsize^{-p'}$.
	\end{lemma}
	This is precisely rewriting the statement of Definition~\ref{def: exp subspaces} into terms of resolution $\meshsize$. The following theorem links $p$ and $q$ from Definition~\ref{def: finite element space} with $p'$ and $q'$ from Lemma~\ref{thm: pq to p'q'}.
	
	\begin{theorem}\label{thm: generic a_U and a_E}
		Suppose $\F H= L^2(\Domain)$ for some connected, bounded domain $\Domain\subset\R^d$ and $\norm\cdot_q\lesssim \Norm\cdot$ for some $q\in[1,\infty]$. For $p\geq0$ and $\meshsize\in(0,1)$, if $(\tilde{\F{U}}^k)_{k\in\F N}$ is an $\meshsize$-refining sequence of finite element spaces of order $p$, then $(\tilde{\var2}_k)_{k\in\F N}$ is an \aUaE\ for
		\begin{align*}
			\aU &\leq \splitln{1}{\text{if }q\geq 2,}{\sqrt{\meshsize^{-d}}}{q<2 \text{ and }\sup_{\var0\in\tilde{\F{U}}^0}\frac{\norm{\var0}_{L^\infty(\domain_0)}}{\norm{\var0}_{L^2(\domain_0)}}<\infty},
			& \aE &\geq \begin{cases}
				\meshsize^{-2p} & \text{if }\nabla\op{E} \text{ is $\Norm\cdot$-Lipschitz at }\var0^*,
				\\\meshsize^{-p}\qquad & \text{if }\op{E} \text{ is $\Norm\cdot$-Lipschitz at }\var0^*,
				\\ 1 &\text{otherwise.}
			\end{cases}
		\end{align*}
	\end{theorem}
	The proof of this theorem is in Appendix~\ref{app: generic a_U and a_E}. Note that $\sup_{\var0\in\tilde{\F{U}}^0}\norm{\var0}_{L^\infty(\domain_0)}\norm{\var0}_{L^2(\domain_0)}^{-1}$ is finite whenever $\tilde{\F{U}}^0\subset L^\infty(\Domain)$ is finite dimensional, so this is not a very strong assumption. The main take-home for this theorem is that the computation of $\aU$ and $\aE$ is typically very simple and clear given a particular choice of $\Norm\cdot$ and $\op{E}$. We also briefly remark that the Lipschitz constants in this lemma do not need to be valid globally, only on the sequence $\tilde{\var2}_k$. The same result holds under a local-Lipschitz assumption, for example on the ball of radius $\sup_{k\in\F N}\Norm{\tilde{\var2}_k}$ which is finite whenever $p\geq0$.
	
	%%%%%%%%%%%%%%%%%%%%%%%%%%%%%%%%%%%%%%%%%%%%%%%%%%
	%%%%%%%%%%%%%%%%%%%%%%%%%%%%%%%%%%%%%%%%%%%%%%%%%%
	%%%%%%%%%%%%%%%%%%%%%%%%%%%%%%%%%%%%%%%%%%%%%%%%%%
	
	\section{L1 penalised reconstruction}\label{sec: Lasso definition}
	The canonical example for FISTA is the LASSO problem with a quadratic data fidelity and L1 regularisation. In this section we develop the necessary analytical tools for the variant with general smooth fidelity term which will be used for numerical results in Section~\ref{sec: numerics}. We consider three forms which will be referred to as the continuous, countable, and discrete problem depending on whether the space $\F{U}$ is $\C M([0,1]^d)$, $\ell^1(\R)$, or $\R^M$ respectively. We choose $\F H$ to be $L^2([0,1]^d),\ \ell^2(\R),$ or $\R^M$ correspondingly. Let $\A\colon\F{U}\cap\F{H}\to\R^m$ be a linear operator represented by the kernels $\psi_j\in\F{H}$ such that 
	\begin{equation}\label{eq: kernels of A}
		\forall \var0\in\F{U}\cap\F{H},\ j=1,\ldots,m,\qquad (\A\var0)_j = \IP{\psi_j}{\var0}.
	\end{equation}
	In the continuous case we will assume the additional smoothness $\psi_j\in C^1([0,1]^d)$. In Section~\ref{sec: smoothing operators} we will formally define and estimate several operator semi-norms for $\A$ of this form, for example Lemma~\ref{thm: norm bound L2} confirms that $\A$ is continuous on $\F{H}$ (without loss of generality $\norm{\A}\leq 1$). In each case, the energy we consider is written as 
	\begin{equation}\label{eq: Lasso energy}
		\op{E}(\var0) = \op{f}(\A\var0-\data) + \mu\Norm{\var0}
	\end{equation}
	for some $\mu>0$ where $\Norm\cdot=\norm\cdot_1$. We assume $\op{f}\in C^1(\R^m)$ is convex, bounded from below, and $\nabla\op{f}$ is 1-Lipschitz.
	Let $\var0^*\in\argmin_{\var0\in\F{U}}\op{E}(\var0)$, which is non-empty so long as $\psi_j\in C([0,1]^d)$, see the proof of \cite[Prop. 3.1]{Bredies2013} when $\op{f}$ is quadratic.
	
	The aim of this section is to develop all of the necessary tools for implementing Algorithm~\ref{alg: refining FISTA} on the energy \eqref{eq: Lasso energy} using the convergence guarantees of either Theorem~\ref{thm: exponential FISTA convergence} or Theorem~\ref{thm: stronger exponential FISTA convergence}. This includes computing the rates $\aU$ and $\aE$, estimating the continuous gap $\op{E}_0(\var0_n)$, and developing an efficient refinement choice for $\F{U}^n$. Below we will just describe the form of $\tilde{\F U}^k$ under the assumption that $\F U^n\subset\tilde{\F U}^k$ is chosen adaptively for $n=n_{k-1}+1,\ldots,n_{k}$. The index $k$ refers to the scale or resolution and $n$ refers to the iteration number of the reconstruction algorithm.
	
	\subsection{Continuous case}\label{sec: continuous lasso rate}
	We start by estimating rates in the case $\F{U}=\C M(\Domain)$ where $\Domain =[0,1]^d$. In this case we choose $\tilde{\F{U}}^k$ to be the span of all piecewise constant functions on a mesh of squares with maximum side length $2^{-k}$ (i.e. $\meshsize=\tfrac12$) and $$\tilde{\var2}_k\coloneqq \sum_{\domain\in\F M^k}\frac{\var0^*(\domain)}{|\domain|}\1_\domain\quad\text{where}\quad\1_\domain(\vec x)=\splitln{1}{\vec{x}\in\domain}{0}{\text{else}}.$$
	By construction $\tilde{\var2}_k\in\tilde{\F{U}}^k$, however note that for any $\var0\in L^1(\Domain)$ and Dirac mass $\delta$ supported in $(0,1)^d$, 
	\begin{equation}
		\Norm{\var0-\delta} = \sup_{\varphi\in C(\Domain), \norm{\varphi}_{L^\infty}\leq 1} \IP{\varphi}{\var0-\delta} = \Norm{\var0}+\Norm{\delta} \geq 1 = \meshsize^0.
	\end{equation}
	Because of this, application of Theorem~\ref{thm: generic a_U and a_E} with $p=0$ gives $\aU=2^{\frac d2}$ but only $\aE\geq1$.
	To improve our estimate of $\aE$ requires additional assumptions on $\A$. Note that $\Norm{\tilde{\var2}_k}=\sum_{\domain\in\F M^k}|\var0^*(\domain)|\leq \Norm{\var0^*}$, therefore we have
	\begin{align}
		\op{E}(\tilde{\var2}_k)-\op{E}(\var0^*) &= \op{f}(\A\tilde{\var2}_k-\data)-\op{f}(\A\var0^*-\data) + \mu\left(\Norm{\tilde{\var2}_k}-\Norm{\var0^*}\right)
		\\&\leq \ip{\nabla \op{f}(\A\tilde{\var2}_k-\data)}{\A(\tilde{\var2}_k-\var0^*)}
		\\&\leq\left[\norm{\nabla\op{f}(\A\var0^*-\data)}_{\ell^2} + \norm{\A(\tilde{\var2}_k-\var0^*)}_{\ell^2}\right]\norm{\A(\tilde{\var2}_k-\var0^*)}_{\ell^2}.
	\end{align}
	as $\op{f}$ is convex with 1-Lipschitz gradient. Clearly $\norm{\nabla\op{f}(\A\var0^*-\data)}_{\ell^2}$ is a constant. For the other term, for all $\vec{r}\in\R^m$ denote $\varphi\coloneqq\A^*\vec{r}$, then note that
	\begin{equation}
		\ip{\vec r}{\A(\tilde{\var2}_k-\var0^*)} = \IP{\varphi}{\tilde{\var2}_k-\var0^*} = 
		\sum_{\domain\in\F M^k}\int_\domain \varphi(\vec{x}) \diff[\tilde{\var2}_k-\var0^*] 
		= \sum_{\domain\in\F M^k}|\domain|^{-1}\iint_{\domain^2} [\varphi(\vec{x}) - \varphi(\vec{y})] \diff\vec{x}\diff\var0^*(\vec{y}).
	\end{equation}
	With the pointwise bound $|\varphi(\vec{x}) - \varphi(\vec{y})|\leq \op{diam}(\domain)\norm{\nabla\varphi}_{L^\infty} = \sqrt{d}2^{-k}\norm{\nabla[\A^*\vec r]}_{L^\infty}$, we deduce the estimate
	\begin{equation}
		\norm{\A(\tilde{\var2}_k-\var0^*)}_{\ell^2} = \sup_{\vec{r}\in\R^m}\norm{\vec{r}}_{\ell^2}^{-1} \IP{\A^*\vec{r}}{\tilde{\var2}_k-\var0^*} \leq \sqrt{d}2^{-k}\Norm{\var0^*}\sup_{\vec{r}\in\R^m}\norm{\vec{r}}_{\ell^2}^{-1}\norm{\nabla[\A^*\vec{r}]}_{L^\infty}.
	\end{equation}
	In Lemma~\ref{thm: norm bound smoothness} we will show that this last term, which we denote the semi-norm $|\A^*|_{\ell^2\to C^1}$, is bounded by \\$\sqrt{m}\max_{j\in[m]}\norm{\nabla\Psi_j}_\infty$. We conclude that $\op{E}(\tilde{\var2}_k)-\op{E}(\var0^*)\lesssim 2^{-k}$.
	In particular, this computation confirms two things. Firstly that the scaling constant is $\aE = 2$, and secondly that the required smoothness to achieve a good rate with Algorithm~\ref{alg: refining FISTA} is that $\A^*\colon\R^m\to C^1(\Domain)$ is a bounded operator. This accounts for using the weaker topology of $\C M(\Domain)$ rather than $L^1(\Domain)$. 
	
	Inserting the computed rates into Theorem~\ref{thm: exponential FISTA convergence} or Theorem~\ref{thm: stronger exponential FISTA convergence} gives the guaranteed convergence rate
	\begin{equation}\label{eq: Lasso energy rate}
		\kappa = \frac{\log \aU^2}{\log \aE + \log \aU^2} = \frac{d}{1+d} \quad \implies \quad \op{E}(\var0_n)-\Emin \lesssim n^{-2(1-\kappa)} = n^{-\frac{2}{1+d}}.
	\end{equation}
	This rate can be used to infer the required resolution at each iteration, in particular on iteration $n$ with $n^2\simeq (\aE\aU^2)^k$ we expect the resolution to be
	\begin{equation}\label{eq: Lasso resolution rate}
		2^{-k} = \left(\aE\aU^2\right)^{\frac{k}{1+d}} \simeq n^{-\frac{2}{1+d}}.
	\end{equation}
	
	\subsection{Countable and discrete case}
	We now extend the rate computations to the case when $\F{U}=\ell^1(\R)$, or a finite dimensional subspace. The key fact here is that, even when $\F{U}$ is infinite dimensional, it is known (e.g. \cite[Thm 6]{Unser2016} and \cite[Cor 3.8]{Boyer2019}) that there exists $\var0^*\in\argmin_{\var0\in\F{U}}\op{E}(\var0)$ with at most $m$ non-zeros. If this is the case, then $\var0^*\in \ell^2(\R)$, indeed $\norm{\var0^*}_{\ell^2}\leq\sqrt{m}\norm{\var0^*}_{\ell^1}$. This makes the estimates of $\aE/\aU$ much simpler than in the continuous case as we can stay in the finite-dimensional Hilbert-space setting. 
	
	For countable dimensions we consider discretisation subspaces of the form 
	$$\tilde{\F{U}}^k = \{\var0\in \ell^1(\R) \st i\notin J_k\implies \var0_i=0\}$$ for some sets $J_k\subset\F N$, i.e. infinite vectors with finitely many non-zeros. The key change in analysis from the continuous case is $\norm{\var0^*}<\infty$, so $\aU=1$ and the expected rate of $n^{-2}$, independent of $\aE$ or any additional properties of $\A$. The number of refinements will also be finite, therefore $n_k=\infty$ for some $k$, the remaining conditions of Theorems~\ref{thm: exponential FISTA convergence} and~\ref{thm: stronger exponential FISTA convergence} hold trivially.
	
	\subsection{Refinement metrics}\label{sec: Lasso gap and gradient}
	Lemma~\ref{thm: practical refinement criteria} shows that adaptive refinement can be performed based on estimates of the function gap or the subdifferential. In this subsection we provide estimates for the forth case of Lemma~\ref{thm: practical refinement criteria} which can be easily computed. In this case we consider $\partial\op{E}\colon \F H\rightrightarrows\F H$ so that subdifferentials are well behaved, for example for explicit computation assuming validity of the chain/sum rules for differentiation.
	
	\subsubsection{Bounds for discretised functionals}\label{sec: bound discrete}
	We start by computing estimates for discretised energies. This covers the cases when either the continuous/countable energy is projected onto $\F{U}^n$, or $\F{U}$ is finite dimensional. For notation we will use the continuous case, to recover the other cases just replace continuous indexing with discrete (i.e. $\var0(\vec{x}) \leadsto \var0_i$).
	
	Let $\tens{\Pi}_n\colon\F H\to\F{U}^n$ denote the orthogonal projection. We consider the discretised function $\op{E}|_{\F{U}^n}\colon \F{U}^n\to\R$ and its subdifferential $\partial_n\op{E}(\cdot) = \tens{\Pi}_n\partial\op{E}(\cdot)$ on $\F{U}^n$. In our case, the behaviour of $\op{E}|_{\F{U}^n}$ is equivalent to replacing $\var0$ with $\tens{\Pi}_n\var0$, and $\A^*$ with $\tens{\Pi}_n\A^*$.
	
	\paragraph{Discrete gradient}
	We can use $\tens{\Pi}_n$ to compute the discrete subdifferential at $\var0_n\in\F{U}^n$:
	\begin{align}
		\partial_n \op{E}(\var0_n)(\vec{x}) &= [\tens{\Pi}_n\A^*\nabla\op{f}(\A\var0_n-\data)](\vec{x}) + \begin{cases}
			\{+\mu\} & \var0_n(\vec{x})>0 \\ [-\mu,\mu] & \var0_n(\vec{x})=0 
			\\\{-\mu\} & \var0_n(\vec{x})<0 
		\end{cases} \label{eq: sign long}
		\\&\eqqcolon [\tens{\Pi}_n\A^*\nabla\op{f}(\A\var0_n-\data)](\vec{x}) + \mu\tens{\Pi}_n\sign(\var0_n(\vec{x})) \label{eq: sign short}
	\end{align}
	where we define $s+\mu[-1,1] = [s-\mu,s+\mu]$ for all $s\in\R$, $\mu\geq0$. 
	
	As $\Norm\cdot = \norm\cdot_1$, the natural metric for $\partial_n\op{E}$ is $\Norm\cdot_* = \norm\cdot_\infty$ which we can estimate
	\begin{align}
		\Norm{\partial_n \op{E}(\var0_n)}_* &= \max_{\vec{x}\in\Domain}\min_{\var1}\left\{|\var1| \st \var1\in \tens{\Pi}_n\A^*\nabla\op{f}(\A\var0_n-\data)(\vec{x}) +  \mu\tens{\Pi}_n\sign(\var0_n(\vec{x}))\right\}
		\\&= \max_{\vec{x}\in\Domain}\begin{cases}
			|[\tens{\Pi}_n\A^*\nabla\op{f}(\A\var0_n-\data)(\vec{x}) + \mu| & \var0_n(\vec{x})>0
			\\|[\tens{\Pi}_n\A^*\nabla\op{f}(\A\var0_n-\data)(\vec{x}) - \mu| & \var0_n(\vec{x})<0
			\\\max\left(|\tens{\Pi}_n\A^*\nabla\op{f}(\A\var0_n-\data)(\vec{x})| - \mu, 0\right) & \var0_n(\vec{x}) = 0
		\end{cases}
	\end{align}
	which can be used directly in Lemma~\ref{thm: practical refinement criteria}.
	
	\paragraph{Discrete gap}
	We now move on to the discrete gap, $\op{E}(\var0_n)-\min_{\var0\in\F{U}^n}\op{E}(\var0)$. This can be computed with a dual representation (e.g. \cite{Duval2017a}),
	\begin{align}
		\min_{\var0\in \F{U}^n} \op{f}(\A\var0-\data) + \mu\Norm{\var0} &= \min_{\var0\in\F H}\max_{\vec{\varphi}\in\R^m} \ip{(\A\tens{\Pi}_n\var0-\data)}{ \vec{\varphi}} +\mu\Norm{\tens{\Pi}_n\var0} -\op{f}^*(\vec{\varphi})
		\\&= \max_{\vec{\varphi}\in\R^m}\min_{\var0\in\F H} \ip{(\A\tens{\Pi}_n\var0-\data)}{\vec{\varphi}} +\mu\Norm{\tens{\Pi}_n\var0} -\op{f}^*(\vec{\varphi})
		\\&= \max_{\vec{\varphi}\in\R^m} \splitln{-\ip{\data}{\vec{\varphi}} -\op{f}^*(\vec{\varphi})}{\qquad\Norm{\tens{\Pi}_n\A^*\vec{\varphi}}_*\leq \mu}{-\infty}{\qquad\text{else}}
		\\&= -\min_{\vec{\varphi}\in\R^m} \underbrace{\op{f}^*(\vec{\varphi})+\ip{\data}{\vec{\varphi}}}_{\eqqcolon \op{E}^\dag(\vec{\varphi})} + \chi(\Norm{\tens{\Pi}_n\A^*\vec{\varphi}}_*\leq \mu). \label{eq: E dagger}
	\end{align}
	In particular,
	\begin{equation}
		\op{E}(\var0) - \min_{\var0\in\F{U}^n}\op{E}(\var0) = \op{E}(\var0) + \min_{\vec{\varphi}\in\R^m\st \Norm{\tens{\Pi}_n\A^*\vec{\varphi}}_*\leq \mu} \op{E}^\dag(\vec{\varphi}) \leq \op{E}(\var0) + \op{E}^\dag(\vec{\varphi})
	\end{equation}
	for any feasible $\vec{\varphi}\in\R^m$. We further derive the criticality condition, if $(\var0^*,\vec{\varphi}^*)$ is a saddle point, then
	\begin{equation}
		\A\var0^*-\data\in\partial\op{f}^*(\vec{\varphi}^*),\qquad\text{or equivalently}\qquad
		\vec{\varphi}^* = \nabla\op{f}(\A\var0^*-\data).
	\end{equation}
	We remark briefly that $\op{E}^\dag$ should be thought of as the dual of $\op{E}$ but without the constraint. We choose to omit it here to highlight that it is only the constraint which changes between the discrete and continuous cases; the value of $\op{E}^\dag$ will remain the same.
	
	Given $\var0_n\in \F{U}^n$, the optimality condition motivates a simple rule for choosing $\vec{\varphi}$:
	\begin{equation}
		\vec{\varphi}_n\coloneqq \nabla\op{f}(\A\var0_n - \data), \qquad \op{E}(\var0)-\min_{\var0'\in\F{U}^n}\op{E}(\var0') \leq \op{E}(\var0) +\op{E}^\dag(\vars2\vec{\varphi}_n)
	\end{equation}
	for some $0\leq \vars2\leq\frac{\mu}{\Norm{\tens{\Pi}_n\A^*\vec{\varphi}_n}_*}$. In the case $\op{f}(\cdot)=\frac12\norm{\cdot}_{\ell^2}^2$, one can use the optimal choice
	\begin{equation}
		\vars2 =\max\left(0, \min\left(\frac{-\ip{\data}{\vec{\varphi}_n}}{\norm{\vec{\varphi}_n}_{\ell^2}^2}, \frac{\mu}{\Norm{\tens{\Pi}_n\A^*\vec{\varphi}_n}_*}\right)\right).
	\end{equation}
	To apply Algorithm~\ref{alg: refining FISTA}, we are assuming that both $\op{f}(\A\var0_n-\data)$ and $\tens{\Pi}_n\nabla\op{f}(\A\var0_n-\data)$ are easily computable, therefore $\vars2$ and $\op{E}(\var0_n) +\op{E}^\dag(\vars2\vec{\varphi}_n)$ are also easy to compute.
	
	\subsubsection{Bounds for countable functionals}
	Extending the results of Section~\ref{sec: bound discrete} to $\F{U}=\ell^1(\R)$ is analytically very simple but computationally relies heavily on the specific choice of $\A$. The computations of subdifferentials and gaps carry straight over replacing $\tens{\Pi}_n$ with the identity and adding the sets $J_n\subset\F N$ which define $\F{U}^n = \{\var0\in\ell^1\st i\notin J_n\implies \var0_i=0\}$. Recall that $\Norm{\partial\op{E}(\var0_n)}_* \coloneqq \inf_{s\in\op{sign}(\var0_n)} \Norm{\A^*\vec{\varphi}_n+\mu s}_*$ where the $\op{sign}$ function has the pointwise set-valued definition as indicated in \eqref{eq: sign long}-\eqref{eq: sign short}. Where $[u_n]_i=0$, the choice $s_i=\min(1,\max(-1,-\mu^{-1}[\A^*\vec{\varphi}_n]_i))$ achieves the minimal value
	\begin{align}
		\Norm{\partial\op{E}(\var0_n)}_* &= \max_{i\in\F N}\begin{cases}
			|[\A^*\vec{\varphi}_n]_i + \mu| & [\var0_n]_i>0
			\\|[\A^*\vec{\varphi}_n]_i - \mu| & [\var0_n]_i<0
			\\\max\left(|[\A^*\vec{\varphi}_n]_i| - \mu, 0\right) & [\var0_n]_i = 0
		\end{cases} \label{eq: pointwise gradient formula}
		\\ \op{E}(\var0_n)-\Emin &\leq \op{E}(\var0_n) + \op{E}^\dag(\vars2_0\vec{\varphi}_n), \qquad \vars2_0\in\left[0,\frac{\mu}{\Norm{\A^*\vec{\varphi}_n}_*}\right] \label{eq: discrete energy formula}
	\end{align}
	where $\vec{\varphi}_n = \nabla\op{f}(\A\var0_n-\data)\in\R^m$ is always exactly computable.
	
	In the countable case, the sets $J_n$ give a clear partition into known/unknown values in these definitions. For $i\in J_n$ the computation is the same as in Section~\ref{sec: bound discrete}, then for $i\notin J_n$ we know $[\var0_n]_i=0$ which simplifies the remaining computations. This leads to:
	\begin{align}
		\Norm{\partial\op{E}(\var0_n)}_* &= \max\left(\max_{i\in J_n}|[\partial\op{E}(\var0_n)]_i|,\ \sup_{i\notin J_n} |[\partial\op{E}(\var0_n)]_i|\right)
		= \max\left(\Norm{\partial_n\op{E}(\var0_n)}_*,\ \sup_{i\notin J_n} |[\A^*\vec{\varphi}_n]_i|-\mu\right)
		\\ \Norm{\A^*\vec{\varphi}_n}_* &= \max\left(\max_{i\in J_n}|[\A^*\vec{\varphi}_n]_i|,\ \sup_{i\notin J_n} |[\A^*\vec{\varphi}_n]_i|\right)
		\hspace{19pt}= \max\left(\Norm{\tens{\Pi}_n\A^*\vec{\varphi}_n}_*,\ \sup_{i\notin J_n} |[\A^*\vec{\varphi}_n]_i|\right).
	\end{align}
	Both estimates only rely on an upper bound of $\max_{i\notin J_n} |[\A^*\vec{\varphi}_n]_i|$. One example computing this value is seen in Section~\ref{sec: wavelet examples}.
	
	\subsubsection{Bounds for continuous functionals}\label{sec: bound continuous}
	Finally we extend the results of Section~\ref{sec: bound discrete} to continuous problems. Similar to the countable case \eqref{eq: pointwise gradient formula}-\eqref{eq: discrete energy formula}, the exact formulae can be written down immediately:
	\begin{align}
		\Norm{\partial\op{E}(\var0_n)}_* &= \max_{\vec{x}\in\Domain}\begin{cases}
			|[\A^*\vec{\varphi}_n](\vec{x}) + \mu| & \var0_n(\vec{x})>0
			\\|[\A^*\vec{\varphi}_n](\vec{x}) - \mu| & \var0_n(\vec{x})<0
			\\\max\left(|[\A^*\vec{\varphi}_n](\vec{x})| - \mu, 0\right) & \var0_n(\vec{x}) = 0
		\end{cases}
		\\ \op{E}(\var0_n)-\Emin &\leq \op{E}(\var0_n) + \op{E}^\dag(\vars2_0\vec{\varphi}_n), \qquad \vars2_0\in\left[0,\frac{\mu}{\Norm{\A^*\vec{\varphi}_n}_*}\right]
	\end{align}
	with $\op{E}^\dag$ as defined in \eqref{eq: E dagger}.
	Recall that there is a mesh $\F M^n$ corresponding to $\F{U}^n$ such that $\var0_n$ is constant on each $\domain\in\F M^n$, so we can rewrite these bounds:
	\begin{align}
		\Norm{\partial \op{E}(\var0_n)}_* &= \max_{\domain\in\F M^n} \begin{cases}
			\norm{\A^*\vec{\varphi}_n + \mu}_{L^\infty(\domain)} & \var0_n|_{\domain} > 0
			\\\norm{\A^*\vec{\varphi}_n - \mu}_{L^\infty(\domain)} & \var0_n|_{\domain} < 0
			\\\max(0, \norm{\A^*\vec{\varphi}_n}_{L^\infty(\domain)}-\mu) & \var0_n|_{\domain} = 0
		\end{cases}
		\\ \Norm{\A^*\vec{\varphi}_n}_* &= \max_{\domain\in\F M^n} \norm{\A^*\vec{\varphi}_n}_{L^\infty(\domain)}.
	\end{align}
	Now, both values can be estimated relying on pixel-wise supremum norms of $\A^*\vec{\varphi}_n$ which we have assumed is sufficiently smooth. We will therefore use a pixel-wise Taylor expansion to provide a simple and accurate estimate. For instance, let $\vec{x}_i$ be the midpoint of the pixel $\domain$, then 
	\begin{equation}\label{eq: Taylor approximation}
		\norm{\A^*\vec{\varphi}_n}_{L^\infty(\domain)} \leq |[\A^*\vec{\varphi}_n](\vec{x}_i)| + \frac{\op{diam}(\domain)}{2}|[\nabla\A^*\vec{\varphi}_n](\vec{x}_i)| + \frac{\op{diam}(\domain)^2}{8}|\A^*\vec{\varphi}_n|_{C^2}.
	\end{equation}
	In this work we chose a first order expansion because we are looking for extrema of $\A^*\vec{\varphi}_n$, i.e. we are most interested in the squares $\domain$ such that 
	\begin{equation}
		|[\A^*\vec{\varphi}_n](\vec{x}_i)| \approx \mu, \qquad |[\nabla\A^*\vec{\varphi}_n](\vec{x}_i)| \approx 0, \qquad [\nabla^2\A^*\vec{\varphi}_n](\vec{x}_i) \preceq 0. 
	\end{equation}
	A zeroth order expansion would be optimally inefficient (approximating $|[\nabla\A^*\vec{\varphi}_n](\vec{x}_i)|$ with $|\A^*\vec{\varphi}_n|_{C^1}$) and a second order expansion would possibly be more elegant but harder to implement. We found that a first order expansion was simple and efficient.
	
	The bounds presented here for continuous problems emphasise the twinned properties required for adaptive mesh optimisation. The mesh should be refined greedily to the structures of $\var0^*$, but also must be sufficiently uniform to provide a good estimate for $\op{E}(\var0^*)$. This is a classical exploitation/exploration trade-off; exploiting visible structure whilst searching for other structures which are not yet visible.

	\subsection{Support detection}\label{sec: support detection}
	The main motivation for using L1 penalties in applications is because it recovers sparse signals, in the case of compressed sensing the support of $\var0^*$ is also provably close to the ``true'' support \cite{Duval2017a,Poon2018}. If $\var0_n\approx \var0^*$ in the appropriate sense, then we should also be able to quantify the statement $\op{supp}(\var0_n)\approx\op{supp}(\var0^*)$. Such methods are referred to as \emph{safe screening} rules \cite{ElGhaoui2010} which gradually identify the support and allow the optimisation algorithm to constrain parts of the reconstruction to 0. In this subsection we propose a new simple screening rule which is capable of generalising to our continuous subspace approximation setting. It is likely that more advanced methods \cite{Bonnefoy2015,Ndiaye2017} can also be adapted, although that is beyond the scope of this work. The key difference is the allowance of inexact computations resulting from estimates such as \eqref{eq: Taylor approximation}.
	
	The support of $\var0^*$ has already been characterised very precisely \cite{Duval2017a,Poon2018}. In particular, the support is at most $m$ distinct points and are a subset of $\{ \vec{x}\in\Domain \st |\A^*\vec{\varphi}^*|(\vec{x}) = \mu\} $ (an equivalent statement holds for the countable case). Less formally, this can also be seen from the the subdifferential computations in Section~\ref{sec: Lasso gap and gradient}, for all $\vec{x}\in\supp(\var0^*)$ we have
	\begin{equation}
		0\in \partial \op{E}(\var0^*)(\vec{x}) = [\A^*\vec{\varphi}^*](\vec{x}) + \mu\sign(\var0^*(\vec{x})).
	\end{equation}
	
	Heuristically, we will use strong convexity of $\op{E}^\dag$ from \eqref{eq: E dagger} and smoothness of $\A^*$ to quantify the statement:
	$$ \text{if}\quad \op{E}(\var0_n) + \op{E}^\dag(\vars2_0 \vec{\varphi}_n) \approx 0 \quad \text{then}\quad \left\{\vec{x}\st |[\A^*\vec{\varphi}_n](\vec{x})|\ll \mu\right\} \subset\{\vec{x} \st \var0^*(\vec{x})=0\}.$$
	Recall that $\nabla\op{f}$ is 1-Lipschitz if and only if $\op{f}^*$ is 1-strongly convex \cite[Chapter 10, Thm. 4.2.2]{Hiriart2013}. Therefore, if $\vars2_0\vec{\varphi}_n$ and $\vec{\varphi}^*$ are both dual-feasible, then
	\begin{equation}
		\tfrac12\norm{\vars2_0\vec{\varphi}_n-\vec{\varphi}^*}_{\ell^2}^2 \leq \op{E}^\dag(\vars2_0\vec{\varphi}_n) - \op{E}^\dag(\vec{\varphi}^*) = \op{E}^\dag(\vars2_0\vec{\varphi}_n) + \op{E}(\var0^*) \leq \op{E}^\dag(\vars2_0\vec{\varphi}_n) + \op{E}(\var0_n),
	\end{equation}
	which gives an easily computable bound on $\norm{\vars2_0\vec{\varphi}_n-\vec{\varphi}^*}_{\ell^2}$. Now we estimate $\A^*\vec{\varphi}_n$ on the support of $\var0^*$:
	\begin{align}
		\min_{\vec{x}\in\supp(\var0^*)} |[\tens{\Pi}_n\A^*\vec{\varphi}_n](\vec{x})| &\geq \min_{\vec{x}\in\supp(\var0^*)} |[\A^*\vec{\varphi}_n](\vec{x})|
		\\&= \vars2_0^{-1}\min_{\vec{x}\in\supp(\var0^*)} |[\A^*\vars2_0 \vec{\varphi}_n](\vec{x})|
		\\&\geq \vars2_0^{-1}\min_{\vec{x}\in\supp(\var0^*)} |[\A^*\vec{\varphi}^*](\vec{x})| - |[\A^*\vars2_0 \vec{\varphi}_n-\A^*\vec{\varphi}^*](\vec{x})|
		\\&= \vars2_0^{-1}\min_{\vec{x}\in\supp(\var0^*)} \mu - |[\A^*\vars2_0 \vec{\varphi}_n-\A^*\vec{\varphi}^*](\vec{x})|
		\\&\geq \vars2_0^{-1}\left(\mu - |\A^*|_{\ell^2\to L^\infty}\norm{\vars2_0 \vec{\varphi}_n-\vec{\varphi}^*}_{\ell^2}\right).
	\end{align}
	Therefore,
	\begin{equation}
		|[\tens{\Pi}_n\A^* \vec{\varphi}_n](\vec{x})|< \vars2_0^{-1}\left(\mu - \sqrt{2(\op{E}(\var0_n)+\op{E}^\dag(\vars2_0\vec{\varphi}_n))}|\A^*|_{\ell^2\to L^\infty}\right) \qquad \implies \qquad \var0^*(\vec{x})= 0.\label{eq: support equation}
	\end{equation}
	This equation is valid when $\vec{x}$ is either a continuous or countable index, the only distinction is to switch to $\ell^\infty$ in the norm of $\A^*$. To make the equivalent statement on the discretised problem, simply replace $\vars2_0$ with $\vars2$ and $\A^*$ with $\tens{\Pi}_n\A^*$. There are two short observations on this formula:
	\begin{itemize}
		\item The convergence guarantee from Theorem~\ref{thm: exponential FISTA convergence} is for the primal gap $\op{E}(\var0_n)-\op{E}(\var0^*)$, rather than the primal-dual gap $\op{E}(\var0_n)+\op{E}^\dag(\vars2_0\vec{\varphi}_n)$ used here. Although there is no guaranteed rate for the primal-dual gap, it is much more easily computable than the primal gap.
		\item In Section~\ref{sec: continuous lasso rate}, $|\A^*|_{\ell^2\to C^1}<\infty$ was required to compute a rate of convergence for $\op{E}(\var0_n)$, but only $|\A^*|_{\ell^2\to L^\infty}<\infty$ is needed to estimate the support.
	\end{itemize}

	\subsection{Operator norms}\label{sec: smoothing operators}
	For numerical implementation of \eqref{eq: Lasso energy}, we are required to accurately estimate several operator norms of $\A$ of the form in \eqref{eq: kernels of A}. In particular, there are kernels $\psi_j\in\F H$ such that $(\A \var0)_j=\IP{\psi_j}{\var0}$ for each $j\in[m]$.
	Verifying that $\norm{\A}\leq1$ can be performed by computing $|\A\A^*|_{\ell^2\to\ell^2}$, and the adaptivity described in Sections~\ref{sec: continuous lasso rate}, \ref{sec: bound continuous}, and~\ref{sec: support detection} requires the values of $|\A^*|_{\ell^2\to L^\infty}$, $|\A^*|_{\ell^2\to C^1}$, and $|\A^*|_{\ell^2\to C^2}$. The aim for this section is to provide estimates of these norms and seminorms for the numerical examples presented in Section~\ref{sec: numerics}.
	
	The following lemma allows for exact computation of the operator norm of $\A$.
	\begin{lemma}\label{thm: norm bound L2}
		If $\A\colon \F H \to \R^m$ has kernels $\psi_j\in \F H$ for $j\in[m]$, then
		$\A\A^*\in\R^{m\times m}$ has entries $(\A\A^*)_{i,j} = \IP{\psi_i}{\psi_j}$, so the spectral norm $\norm{\A^*\A} = \norm{\A\A^*}$ can be computed efficiently.
	\end{lemma}
	\begin{proof}
		To compute the entries of $\A\A^*\colon\R^m\to\R^m$, observe that for any $\vec{r}\in\R^m$
		\begin{equation}
			(\A\A^*\vec{r})_i = \IP{\psi_i}{\A^*\vec{r}} = \IP{\psi_i}{\sum_{j=1}^mr_j\psi_j} = \sum_{j=1}^m \IP{\psi_i}{\psi_j}r_j
		\end{equation}
		as required.
		\qed\end{proof}
	If $\norm{\A^*\A}$ is not analytically tractable, then Lemma~\ref{thm: norm bound L2} enables it to be computed using standard finite dimensional methods. The operator $\A\A^*$ is always finite dimensional, and can be computed without discretisation error.

	In the continuous case, when $\F H=L^2(\Domain)$ we also need to estimate the smoothness properties of $\A^*$. A generic result for this is given in the following lemma.
	\begin{lemma}\label{thm: norm bound smoothness}
		If $\A\colon L^2([0,1]^d) \to \R^m$ has kernels $\psi_j\in L^2(\Domain)\cap C^k(\Domain)$ for $j\in[m]$, then for all $\frac1q+\frac1{q^*}=1$, $q\in[1,\infty]$, we have
		\begin{align}
			|\A^*\vec{r}|_{C^k}&\coloneqq \sup_{\vec{x}\in\Domain} |\nabla^k[\A^*\vec{r}]|(\vec{x}) \leq \sup_{\vec{x}\in\Domain}\norm{(\nabla^k\psi_j(\vec{x}))_{j=1}^m}_{\ell^{q^*}}\norm{\vec{r}}_{\ell^q},
			\\ |\A^*|_{\ell^2\to C^k}&\coloneqq \sup_{\norm{\vec{r}}_{\ell^2}\leq 1} |\A^*\vec{r}|_{C^k} \leq \sup_{\vec{x}\in\Domain}\norm{(\nabla^k\psi_j(\vec{x}))_{j=1}^m}_{\ell^{q^*}}\times \splitln{1}{q\geq 2}{\sqrt{m^{2-q}}}{q<2}.
		\end{align}
	\end{lemma}
	\begin{proof}
		For the first inequality, we apply the H\"older inequality on $\R^m$:
		$$ |\nabla^k[\A^*\vec{r}]|(\vec{x}) = \left|\sum_{j=1}^m\nabla^k\psi_j(\vec{x})r_j\right|
		\leq \left(\sum_{j=1}^m|\nabla^k\psi_j(\vec{x})|^{q^*}\right)^{\frac{1}{q^*}}\norm{\vec{r}}_{\ell^q} = \norm{(\nabla^k\psi_j(\vec{x}))_j}_{\ell^{q^*}}\norm{\vec{r}}_{\ell^q}\;.$$
		For the second inequality, if $q\geq2$ and $\sum_{j=1}^m r_j^2\leq 1$, then $|r_j|\leq 1$ for all $j$ and $\norm{\vec{r}}_{\ell^q}^q\leq \norm{\vec{r}}_{\ell^2}^2\leq 1$. If $q<2$ and $\norm{\vec{r}}_{\ell^2}\leq 1$, then we again use H\"older's inequality:
		$$\sum_{j=1}^m r_j^q \leq \Big(\sum_{j=1}^m 1^{Q^*}\Big)^{\frac1{Q^*}} \Big(\sum_{j=1}^m r_j^{qQ}\Big)^{\frac1{Q}} \leq m^{\frac{2-q}{2}}$$
		for $Q = \frac2q$.
		\qed\end{proof}
	
	The examples in Section~\ref{sec: numerics} require explicit computations of the expressions in Lemmas~\ref{thm: norm bound L2} and~\ref{thm: norm bound smoothness}. These computations are provided\ in the appendix, Theorem~\ref{thm: norm bound examples}.
	
	\section{Numerical examples}\label{sec: numerics}
	We present four numerical examples. The first two are in 1D to demonstrate the performance of different variants of Algorithm~\ref{alg: refining FISTA}, both with and without adaptivity. In particular, we explore sparse Gaussian deconvolution and sparse signal recovery from Fourier data. We compare with the \emph{continuous basis pursuit} (CBP) discretisation \cite{Ekanadham2011,Duval2017b} which is also designed to achieve super-resolution accuracy within a convex framework. More details of this method will be provided in Section~\ref{sec: 1D Lasso examples}.
	
	The next example is 2D reconstruction from Radon or X-ray data with wavelet-sparsity and a robust data fidelity. As the forward operator is not sufficiently smooth, we must optimise in $\ell^1(\R)$, which naturally leads to the choice of a wavelet basis. 
	
	Finally, we process a dataset which represents a realistic application in biological microscopy, referred to as STORM microscopy. In essence, the task is to perform 2D Gaussian de-blurring/super-resolution and denoising to find the location of sparse spikes of signal.
	
	In this section, the main aim is to minimise $\op{E}_0(\var0_n) = \op{E}(\var0_n)-\op{E}(\var0^*)$, and so this will be our main metric for the success of an algorithm, referred to as the ``continuous gap''. Lemma~\ref{thm: practical refinement criteria} only provides guarantees on the values of $\min_{n\leq N}\op{E}_0(\var0_n)$ so it is this monotone estimate which is plotted. As $\op{E}(\var0^*)$ is not known exactly, we always use the estimate $\min_{n\leq N}\op{E}_0(\var0_n) \approx \min_{n\leq N}\op{E}(\var0_n) + \min_{n'\leq n}\op{E}^\dag(\vars2_0\vec{\varphi}_{n'})$. Another quantity of interest is minimisation of the discrete energy $\min_{n\leq N}\op{E}(\var0_n) + \min_{n'\leq n}\op{E}^\dag(\vars2\vec{\varphi}_{n'})$ which will be referred to as the ``discrete gap''. Note that for the adaptive schemes the discrete gap may not be monotonic as the discrete dual problem changes with $N$.
	
	The code to reproduce these examples can be found online\footnote{\href{https://github.com/robtovey/2020SpatiallyAdaptiveFISTA}{https://github.com/robtovey/2020SpatiallyAdaptiveFISTA}}.

	\subsection{1D continuous LASSO}\label{sec: 1D Lasso examples}
	In this example we choose $\F{U}=\C M([0,1])$, $\F H=L^2([0,1])$,  $\op{f}(\cdot)=\frac12\norm{\cdot}_{\ell^2}^2$ and $\A\colon \F{U}\to \R^{30}$ with either random Fourier kernels:
	\begin{equation}
		(\A\var0)_j = \int_0^1\cos(\vars3_jx)\diff\var0(x), \qquad \vars3_j\sim \op{Uniform}[-100,100],\ j=1,2,\ldots,30,\ \mu =0.02, 
	\end{equation}
	or Gaussian kernels on a regular grid:
	\begin{equation}
		(\A\var0)_j = (2\pi\sigma^2)^{-\frac12}\int_0^1\exp\left(-\frac{(x-(j-1)\Delta)^2}{2\sigma^2}\right)\diff\var0(x), \quad \sigma=0.12,\ \Delta = \tfrac1{29},\ j=1,2,\ldots,30,\ \mu=0.06.
	\end{equation}
	
	Several variants of FISTA are compared for these examples but the key alternative shown here is the CBP discretisation. For this choice of $\op{f}$, we call \eqref{eq: Lasso energy} the continuous LASSO problem, for which there are many numerical methods (c.f. \cite{Bredies2013,Castro2016,Boyd2017,Catala2019}) however, most require the solution of a non-convex problem. We have focused on CBP because it approximates $\var0^*$ through a convex discrete optimisation problem which is asymptotically exact in the limit $\meshsize\to0$. It can also be optimised with FISTA which allows for direct comparison with the uniform and adaptive mesh approaches. The idea is that for a fixed mesh, the kernels of $\A$ are expanded to first order on each pixel and a particular first order basis is also chosen \cite{Ekanadham2011,Duval2017b}. If $\var0^*$ has only one Dirac spike in each pixel, then the zeroth order information should correspond to the mass of the spike, and additional first order information should determine the location.
	
	As shown in Section~\ref{sec: Lasso definition}, in 1D we have $\aU = \aE = 2$. The estimates given in \eqref{eq: Lasso energy rate} and \eqref{eq: Lasso resolution rate} in dimension $d=1$ predict that the adaptive energy will decay at a rate of $\op{E}(\var0_n)-\op{E}(\var0^*)\lesssim\frac1n$ so long as the pixel size also decreases at a rate of $\meshsize\sim\frac1n$. To achieve these rates, we implement a refinement criterion from Lemma~\ref{thm: practical refinement criteria} with guarantee of $\op{E}(\var0_{n_k-1})-\op{E}(\var0^*)\lesssim 2^{-k}$ using the estimates made in Section~\ref{sec: Lasso gap and gradient}. We choose subspaces $\F{U}^n$ to approximately enforce
	\begin{equation}\label{eq: less than double gap bound}
		\op{E}(\var0_n) + \op{E}^\dag(\vars2_0\vec{\varphi}_n) \leq 2(\op{E}(\var0_n) + \op{E}^\dag(\vars2\vec{\varphi}_n)),
	\end{equation}
	i.e. the continuous gap is bounded by twice the discrete gap. In particular, note that for $\vars2_0\approx \vars2$,
	\begin{equation}
		\op{E}^\dag(\vars2_0\vec{\varphi}_n) = \tfrac12\norm{\vars2_0\vec{\varphi}_n}^2 + \vars2_0\ip{\data}{\vec{\varphi}_n} = \frac{\vars2_0}{\vars2}\left(\frac{\vars2_0}{\vars2}\tfrac12\norm{\vars2\vec{\varphi}_n}^2 + \vars2\ip{\data}{\vec{\varphi}_n}\right) \approx \frac{\vars2_0}{\vars2} \op{E}^\dag(\vars2\vec{\varphi}_n).
	\end{equation}
	Converting this into a spatial refinement criteria, recall 
	\begin{equation}
		\frac{\vars2_0}{\vars2} \approx \frac{\Norm{\A^*\vec{\varphi}_n}_*}{\Norm{\tens{\Pi}_n\A^*\vec{\varphi}_n}_*} = \frac{\max_{\domain\in\F M^n} \norm{\A^*\vec{\varphi}_n}_{L^\infty(\domain)}}{\max_{\domain\in\F M^n} |\tens{\Pi}_n\A^*\vec{\varphi}_n(\domain)|} \approx \max_{\domain\in\F M^n}\frac{ \norm{\A^*\vec{\varphi}_n}_{L^\infty(\domain)}}{|\tens{\Pi}_n\A^*\vec{\varphi}_n(\domain)|}
	\end{equation}
	is the maximum ratio of second vs. zeroth order Taylor approximations of $\A^*\vec{\varphi}_n$ on pixel $\domain$. This was found to be an efficient method of selecting pixels for refinement using quantities which had already been computed. Note briefly that this greedy strategy directly targets uncertainty, refinements also happen outside of the support of $u_n$ to guarantee that this is representative of $\var0^*$. Such refinement is necessary to avoid discrete minimisers of $\op{E}$ which are not global minimisers.
	
	\begin{figure}\centering
		\includegraphics[width=.9\textwidth]{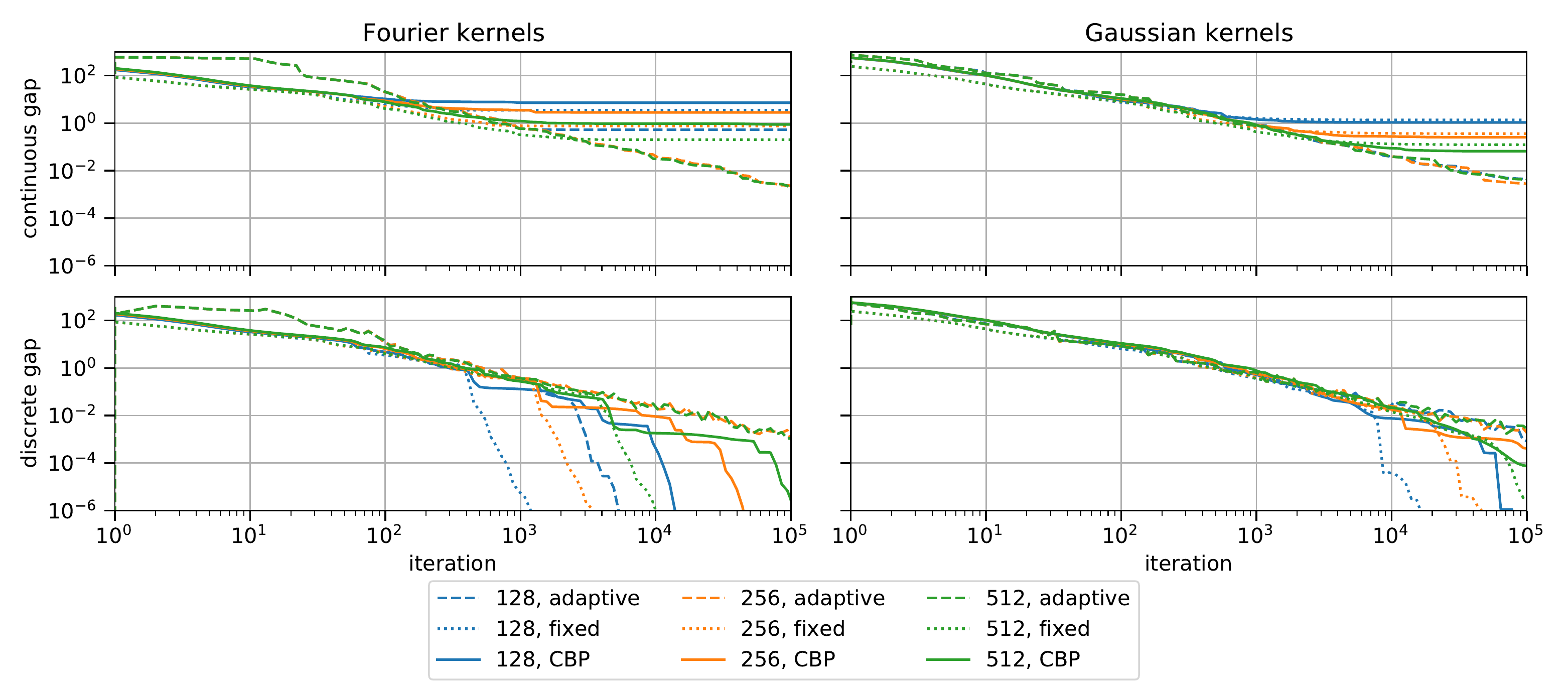}
		\caption{Rates of continuous/discrete gap convergence for different LASSO algorithms with 128, 256, or 512 pixels. The ``adaptive'' method uses the proposed algorithm. Both ``fixed'' and ``CBP'' use standard FISTA with a uniform discretisation.}\label{fig: convergence with ndofs}
		
		\vspace*{\floatsep}
		
		\includegraphics[width=.9\textwidth]{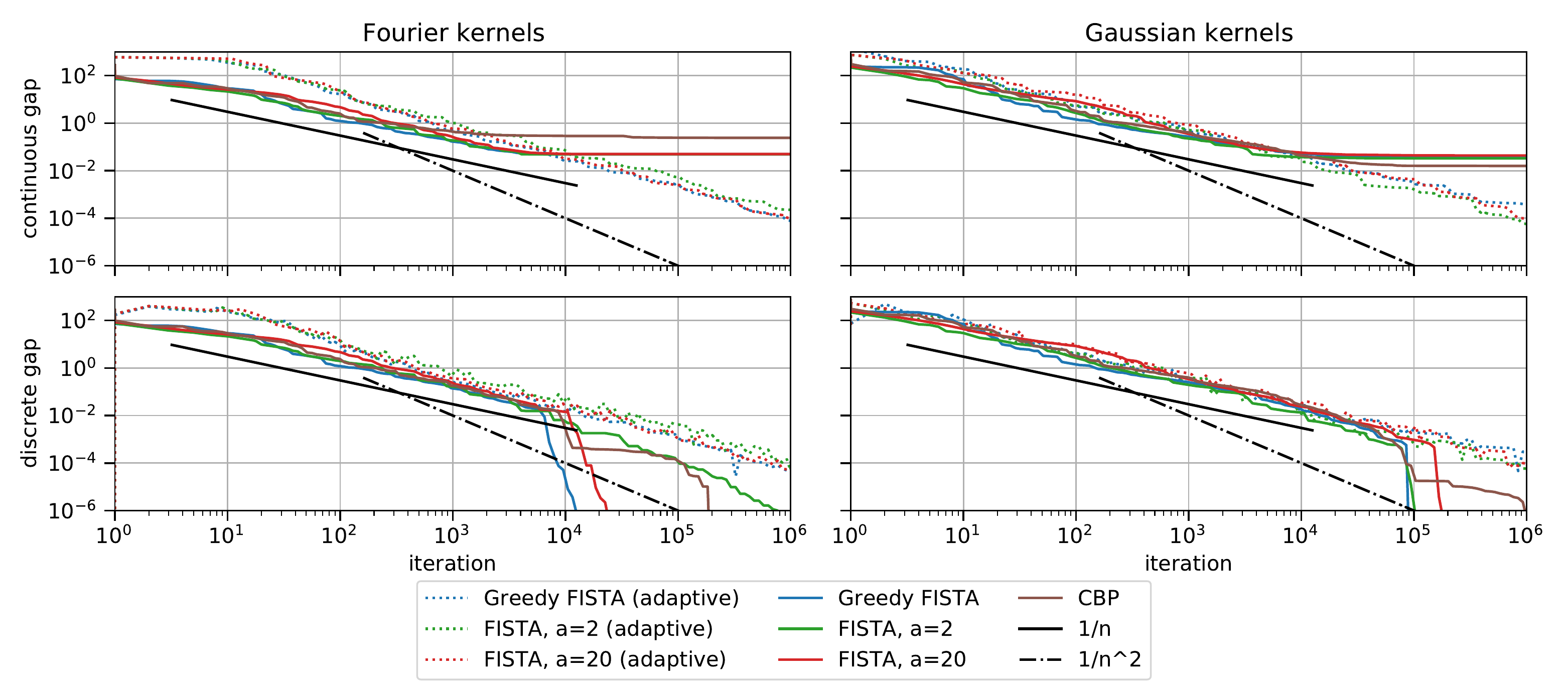}
		\caption{Convergence plots for solving 1D problems with different algorithms. ``Adaptive'' methods use Algorithm~\ref{alg: refining FISTA} with fewer than 1024 pixels and the remaining methods use a uniform discretisation of 1024 pixels.}\label{fig: convergence with method}
		
		\vspace*{\floatsep}
		
		\includegraphics[width=.9\textwidth]{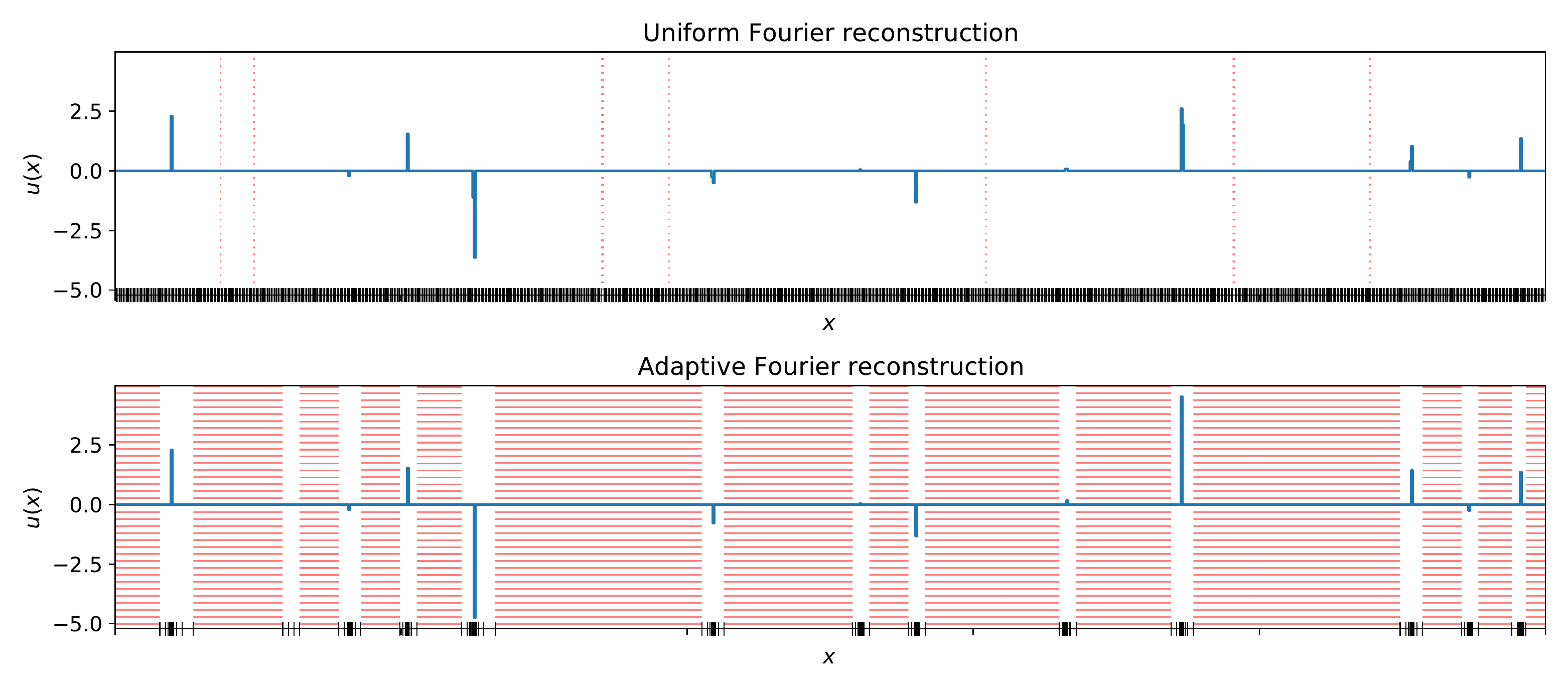}
		\caption{Example reconstruction from the algorithms considered in Fig.~\ref{fig: convergence with method}. Pixel boundaries are indicated on the $x$-axis and the filtering method of Section~\ref{sec: support detection} allows us to exclude the red shaded regions from $\op{supp}(u^*)$. Values on the $y$-axis are normalised to units of mass, i.e. a Dirac mass would have height 1.}\label{fig: example recon}
	\end{figure}
	
	\paragraph{Comparison of discretisation methods}
	In Fig.~\ref{fig: convergence with ndofs} we compare the three core approaches: fixed uniform
	discretisation, adaptive discretisation, and CBP. In particular, we wish to observe their convergence properties as the number of pixels is allowed to grow. In each case we use a FISTA stepsize of $t_n = \frac{n+19}{20}$. The adaptive discretisation is started with one pixel and limited to 128, 256, or 512 pixels while the fixed and CBP discretisations have uniform discretisations with the maximum number of pixels. The main observations are:
	\begin{itemize}
		\item The adaptive scheme is much more efficient, in both examples the adaptive scheme with 128 pixels is at least as good as both fixed discretisations with 512 pixels. In fact, only a maximum of 214 pixels were needed by the adaptive method in either example.
		\item With Fourier kernels the uniform piecewise constant discretisation is more efficient than CBP but in the Gaussian case this is reversed. This suggests that the performance of CBP depends on the smoothness of $\A$.
		\item The discrete gaps for non-adaptive optimisation behave as is common for FISTA, initial convergence is polynomial until a locally linear regime activates \cite{Tao2016}. CBP is always slower to converge than the piecewise constant discretisation.
		\item The adaptive refinement criterion succeeds in keeping the continuous/discrete gaps close for all $n$, i.e. \eqref{eq: less than double gap bound}.
	\end{itemize}
	It is not completely fair to judge CBP with the continuous gap because, although it generates a continuous representation, this continuous representation is not necessarily consistent with the discrete gap being optimised, unlike when discretised with finite element methods. On the other hand, this is still the intended interpretation of the algorithm and we have no more appropriate metric for success in this case.

	\paragraph{Comparison of FISTA variants}
	Fig.~\ref{fig: convergence with method} compares many methods with either fixed or adaptive discretisations. Each adaptive scheme is allowed up to 1024 pixels and each uniform discretisation uses exactly 1024. An example of each reconstruction method is shown in Fig.~\ref{fig: example recon}. The adaptive method better identifies the support of $u^*$ and clearly localises pixels on that support. The reconstruction with uniform grid fails to provably identify the support of $\var0^*$, despite having found a qualitatively accurate discrete minimiser. The ``Greedy FISTA'' implementation was proposed by in \cite{Liang2018} and we include the adaptive variant despite a lack of convergence proof. The remaining FISTA algorithms use a FISTA time step of $t_n = \frac{n+a-1}{a}$ for the given value of $a$, as proposed in \cite{Chambolle2015}. In this example CBP used the greedy FISTA implementation which gave faster observed convergence. Fig.~\ref{fig: convergence with method} compares the discrete gaps because it is the accurate metric for fixed discretisations, and for the adaptive discretisation it should also be an accurate predictor of the continuous gap. 
	The main observations are:
	\begin{itemize}
		\item Each algorithm displays very similar convergence properties. The main difference is that the reconstructions with fixed discretisations accelerate after $10^4$-$10^5$ iterations.
		\item During the initial ``slow'' phase, adaptive and fixed discretisations appear to achieve very similar (discrete) convergence rates. The coarse-to-fine adaptivity is not slower than fixed discretisations in this regime.
		\item Lemma~\ref{thm: practical refinement criteria} accurately predicts the $\frac1n$ rate of the adaptive methods, mirrored in the fixed discretisations. This suggests that high-resolution discretisations are also initially limited by this $\frac1n$ rate before entering the asymptotic regime, consistent with \eqref{eq: slow exact FISTA}.
		\item The fastest FISTA stepsize choice is consistently the greedy variant, although $a=20$ is very comparable.
		\item While each adaptive algorithm is allowed to use up to 1024 pixels, in Fig.~\ref{fig: convergence with method} the most used was 235.
	\end{itemize}
	
	\paragraph{Comparison of fixed and adaptive discretisation}
	Motivated by the findings in Fig.~\ref{fig: convergence with method}, we now look more closely at the performance of the $a=20$ and the greedy FISTA schemes. We have convergence results for the former, but the latter typically performs the best for non-adaptive optimisation and is never worse than $a=20$ in the adaptive setting. The question is whether it is faster/more efficient to use the proposed adaptive scheme, or to use a classical scheme at sufficiently high uniform resolution. The fixed discretisations use 1024 pixels (i.e. constant pixel size of $2^{-10}$ in Fig.~\ref{fig: comparison with iteration}) and the adaptive discretisation starts with two pixels with an upper limit of 1024. As expected, the fixed discretisation starts with a smaller continuous gap before plateauing to a sub-optimal gap around $\op{E}=\op{E}(\var0^*) + 0.1$.
	
	Fig.~\ref{fig: comparison with iteration} shows convergence of pixel size and continuous gap with respect to number of iterations. Fig.~\ref{fig: comparison with time} shows the more practical attributes of continuous gap and number of pixels against execution time. We see that the adaptive discretisation is consistently capable of computing lower energies with fewer pixels and in less time than the uniform discretisation. The convergence behaviour is very consistent with respect to number of iterations.
	
	Suppose that the numerical aim is to find a function $\var0_n$ with $\op{E}(\var0_n)-\op{E}(\var0^*)\leq 0.1$, all methods would converge after $O(10^3)$ iterations, demonstrating some equivalence between the two FISTA algorithms. For $n\in[10^3,10^4]$, in both problems, the adaptive schemes coincide with the fixed schemes in both energy and minimum pixel size. On the other hand, we also see that the adaptive scheme achieves this energy in almost an order of magnitude less time and fewer pixels.
	
	\begin{figure}\centering
		\includegraphics[width=.86\textwidth]{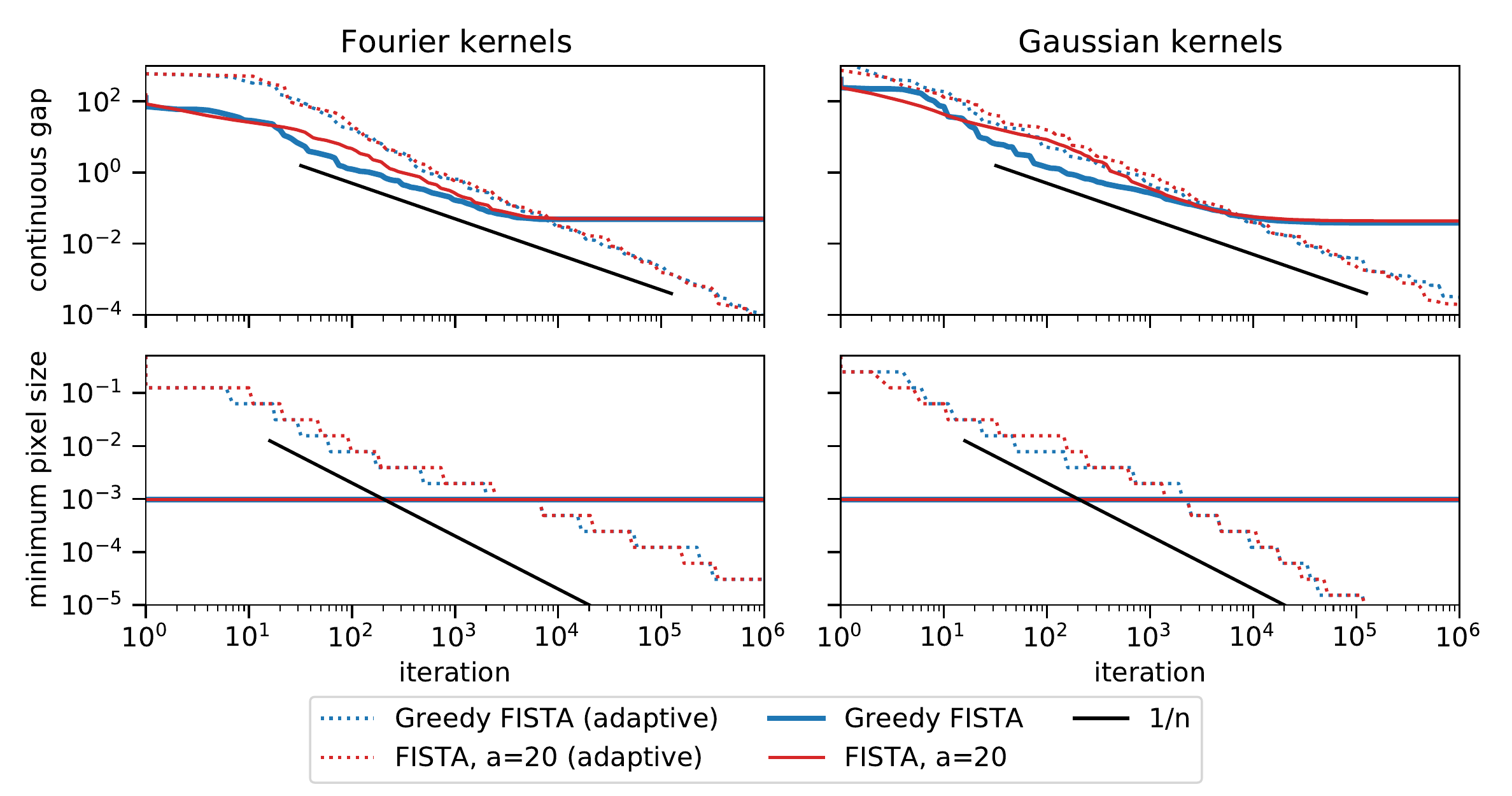}
		\caption{Continuous convergence of adaptive (coarse-to-fine pixel size) compared with uniform discretisation (constant pixel size) with respect to number of iterations. }\label{fig: comparison with iteration}
		
		\vspace*{\floatsep}
		
		\includegraphics[width=.86\textwidth]{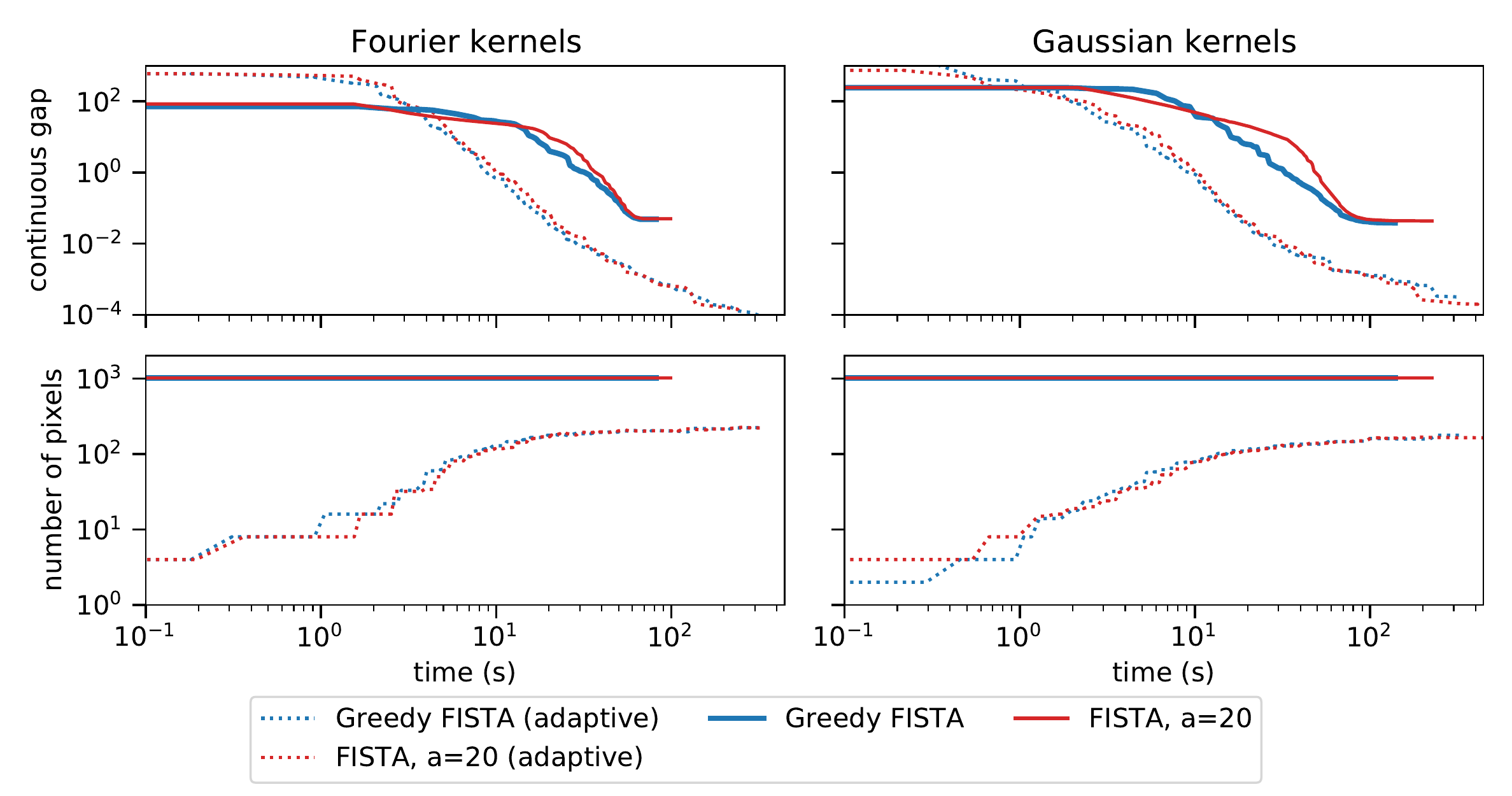}
		\caption{Continuous convergence of adaptive compared with uniform discretisation with respect to wall-clock time and total number of pixels (memory requirement).}\label{fig: comparison with time}
		
		\vspace*{\floatsep}
		
		\begin{center}\begin{tikzpicture}[grow'=right]
				\tikzstyle{level 1}=[level distance=5em, sibling distance=0em];
				\tikzstyle{level 2}=[level distance=5em, sibling distance=6em];
				\tikzstyle{level 3}=[level distance=5em, sibling distance=3em];
				\tikzstyle{bag} = [text width=3em, text centered];
				
				\node[text width=7cm] at (-6cm,0) {
					$\begin{aligned}
						J_n &= \{(0,0), (0,1), (0,2), (1,2), (1,1)\}
						\\\op{leaf}(J_n) &= \{(0,2), (1,2), (1,1)\}
						\\\F M^n &= \left\{[0,\tfrac14),[\tfrac14,\tfrac12),[\tfrac12,1)\right\}
					\end{aligned}$ };
				
				\node[bag] at (0,0) {$ $}
				child {node[bag] {$(0,0)$ $[0,1]$}
					child {node[bag] {$(0,1)$ $[0,\frac12]$}
						child {node[bag] {$(0,2)$ $[0,\frac14]$}
						} child {node[bag] {$(1,2)$ $[\frac14,\frac12]$}
						}
					} child {node[bag] {$(1,1)$ $[\frac12,1]$}
				}};
		\end{tikzpicture}\end{center}
		\caption{Example tree representation of 1D wavelets. Left: nodes, leaves, and mesh of discretisation. Right: arrangement into a tree with index $(j,k)$ and corresponding support of wavelet $w_{j,k}$ underneath.}\label{fig: wavelet tree}
	\end{figure}

	\subsection{2D robust sparse wavelet reconstruction}\label{sec: wavelet examples}
	In this example we consider $\A$ to be a 2D Radon transform. In particular, the rows of $\A$ correspond to integrals over the sets $\F{X}^I_i$ where 
	\begin{equation}
		\F{X}_i^I = \left\{\vec{x}\in[-\tfrac12,\tfrac12]^2\st \ip{\vec{x}}{\begin{pmatrix}\cos\theta_I\\\sin\theta_I\end{pmatrix}}\in \left[-\tfrac12+\tfrac{i-1}{100},-\tfrac12+\tfrac{i}{100}\right)\right\}, \quad \theta_I=\frac{180^\circ}{51}I
	\end{equation}
	for $i=\in[100]$, $I\in[50]$.
	This is not exactly in the form analysed by Theorem~\ref{thm: norm bound examples}, only the sets $\{\F{X}^I_i\st i\in[100]\}$ for each $I$ are disjoint, therefore we apply Theorem~\ref{thm: norm bound examples} block-wise to estimate
	\begin{equation}
		\norm{\A}_{L^2\to\ell^2} \leq \sqrt{\sum_{I\in[50]}\max_{i\in[100]} |\F{X}^I_i|} = \sqrt{\sum_{I\in[50]}\max_{i\in[100]} \int_{\F{X}^I_i}1\diff\vec{x}}= \sqrt{\sum_{I\in[50]}\max_{i\in[100]}\ (\A\1)_{i,I}}\;.
	\end{equation}
	$\A$ is not smooth, therefore we can't bound $|\A^*|_{C^k}$ for $k>0$, and so we must look to minimise over $\ell^1$ rather than $L^1$. The natural choice is to promote sparsity in a wavelet basis which can be rearranged into the  form of \eqref{eq: Lasso energy}:
	\begin{equation}
		\min_{\var0\in \F{U}} \op{f}(\A\var0-\data) + \mu \norm{\tens{W}^{-1}\var0}_{\ell^1} = \min_{\hat{\var0}\in\ell^1(\R)} \op{f}(\A\tens{W}\hat{\var0}-\data) + \mu \norm{\hat{\var0}}_{\ell^1}.
	\end{equation}
	The minimisers are related by $\var0^* = \tens{W}\hat{\var0}^*$ and, for wavelet bases, $\tens{W}$ is orthonormal so $\norm{\A\tens{W}}_{\ell^2\to\ell^2} = \norm{\A}_{L^2\to\ell^2}$. 
	In this example we consider the smoothed robust fidelity \cite{Rosset2007}
	\begin{equation}
		\op{f}(\vec{\varphi}) = \sum_{i=1}^m \splitln{10^{-4}|\varphi_i|}{|\varphi_i|\geq10^{-4}}{\tfrac12|\varphi_i|^2+\tfrac12 10^{-8}}{\text{else}} \approx 10^{-4}\norm{\vec{\varphi}}_{\ell^1}.
	\end{equation}
	From Section~\ref{sec: Lasso gap and gradient} we know that to track convergence and perform adaptive refinement, it is sufficient to accurately bound $|[\tens{W}^\top \A^*\vec{\varphi}_n]_j|$ for all $j\notin J_n$. If $\tens{W}$ is a wavelet transformation then its columns, $w_j\in L^2$, are simply the wavelets themselves and we can use the bound 
	\begin{equation}
		|\IP{w_j}{\A^*\vec{\varphi}_n}| = \left|\IP{w_j}{\1_{\supp(w_j)}\A^*\vec{\varphi}_n}\right| \leq \norm{\1_{\supp(w_j)} \A^*\vec{\varphi}_n}_{L^2}\leq \norm{\1_{\F{X}} \A^*\vec{\varphi}_n}_{L^2}
	\end{equation}
	for all $\F{X}\supset \supp(w_j)$.
	In the case of the Radon transform, we can compute the left-hand side explicitly for the finitely many $j\in J_n$, but we wish to use the right-hand side in a structured way to avoid computing the infinitely many $j\notin J_n$. To do this, we will take a geometrical perspective on the construction of wavelets to view them in a tree format. 
	
	\paragraph{Tree structure of wavelets}
	Finite elements are constructed with a mesh which provided a useful tool for adaptive refinement in Section~\ref{sec: bound continuous}. For wavelets, we will associate a tree with every discretisation and the leaves of the tree correspond to a mesh. This perspective comes from the multi-resolution interpretation of wavelets. An example is seen in Fig.~\ref{fig: wavelet tree} for 1D Haar wavelets, $w_{j,k}(x) = \sqrt{2}^k\psi(2^{k}x-j)$ where $\psi = \1_{[0,1)} - \1_{[-1,0)}$. 
	
	In higher dimensions, the only two things which change are the number of children ($2^d$ for non-leaves) and at each node you store the coefficients of $2^d-1$ wavelets. The support on each node is still a disjoint partition of unity consisting of regular cubes of side length $2^{-k}$ at level $k$. The only change in our own implementation is to translate the support to $[-\tfrac12,\tfrac12]^2$. We briefly remark that the tree structuring of wavelets is not novel and appears more frequently in the Bayesian inverse problems literature \cite{Castillo2019,Kekkonen2021}.
	
	\paragraph{Continuous gradient estimate}
	In Section~\ref{sec: 1D Lasso examples} we used the continuous gap as a measure for convergence, for wavelets we will use the continuous subdifferential. With the tree structure we can easily adapt the results of Section~\ref{sec: Lasso gap and gradient} to estimate subdifferentials (or function gaps). In particular,
	\begin{align}
		\Norm{\partial \op{E}(\var0_n)}_* &=\max\left(\Norm{\partial_n\op{E}(\var0_n)}_*, \max_{j\notin J_n} |\IP{w_j}{\A^*\vec{\varphi}_n}| -\mu \right) \\&\leq\max\left(\Norm{\partial_n\op{E}(\var0_n)}_*, \max_{j\in \op{leaf}(J_n)} \norm{\1_{\supp(w_j)}\A^*\vec{\varphi}_n}_{L^2} -\mu \right).\label{eq: wavelet error metric}
	\end{align}
	
	\paragraph{Numerical results}
	We consider two phantoms where the ground-truth is either a binary disc or the Shepp-Logan phantom. Both examples are corrupted with \SI{2}{\percent} Laplace distributed noise. This is visualised in Fig.~\ref{fig: haar data}. All optimisations shown are spatially adaptive using Haar wavelets and initialised with $\F{U}^0= \{x\mapsto c\st c\in\F R\}$. The gradient metric shown throughout is the $\ell^\infty$ norm. Motivated by \eqref{eq: wavelet error metric}, the spatial adaptivity is chosen to refine nodes $j\in\op{leaf}(J_n)$ to ensure that 
	$$ \norm{\1_{\supp(w_j)}\A^*\vec{\varphi}_n}_{L^2} -\mu  \leq 10\Norm{\partial_n\op{E}(\var0_n)}_*$$
	for all $j$ and $n$ (i.e. so that the continuous gradient is less than 10 times the discrete gradient).
	We do not expect wavelet regularisation to have state-of-the-art performance in the examples of Fig.~\ref{fig: haar data}. What they demonstrate is the preference Haar wavelets have to align large discontinuities with a coarse grid, even when the discretisation is allowed to be as fine as necessary. There is an average of $2\cdot10^6$ wavelet coefficients in each discretised reconstruction, although the higher frequencies have much smaller intensities. In limited data scenarios, wavelet regularisation automatically selects a local ``resolution'' which reflects the quality of data. Particularly in the Shepp-Logan reconstruction, we see that the outer ring is detected with a finer precision than the dark interior ellipses.
	
	The first numerical results shown in Fig.~\ref{fig: haar convergence} compare the same adaptive FISTA variants as shown in Fig.~\ref{fig: convergence with method}. In these examples we see that the greedy FISTA and the $a=20$ algorithms achieve almost linear convergence while $a=2$ is significantly slower. Interestingly, in both examples the $a=20$ variant uses half as many wavelets as the Greedy variant, and therefore converges slightly faster in time. 
	
	\begin{figure}\centering
		\includegraphics[width=.85\textwidth]{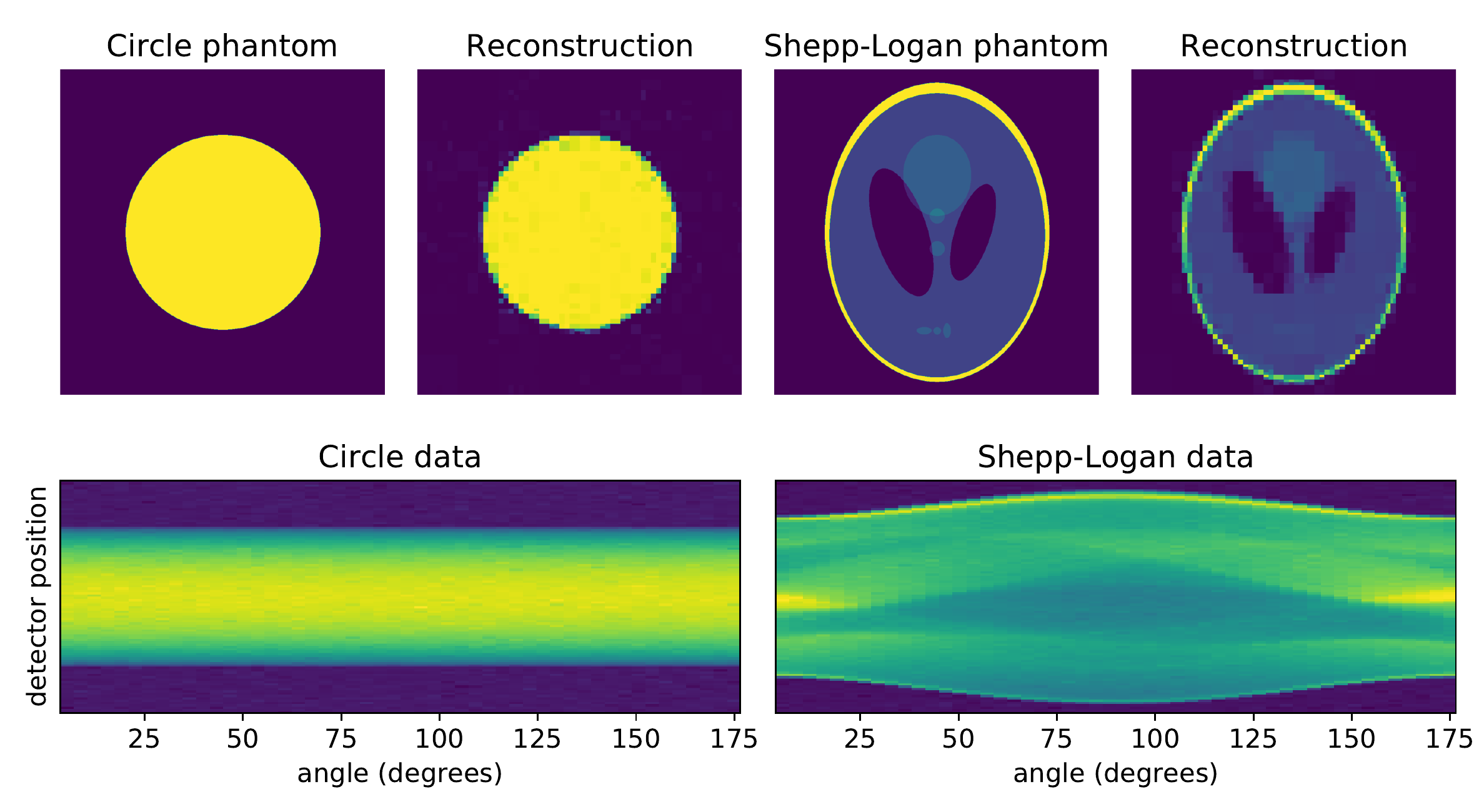}
		\caption{Phantoms, data and reconstructions for wavelet-sparse tomography optimisation. Both examples are corrupted with \SI{2}{\percent}\ Laplace distributed noise. }\label{fig: haar data}
		
		\vspace*{\floatsep}
		
		\includegraphics[width=.85\textwidth]{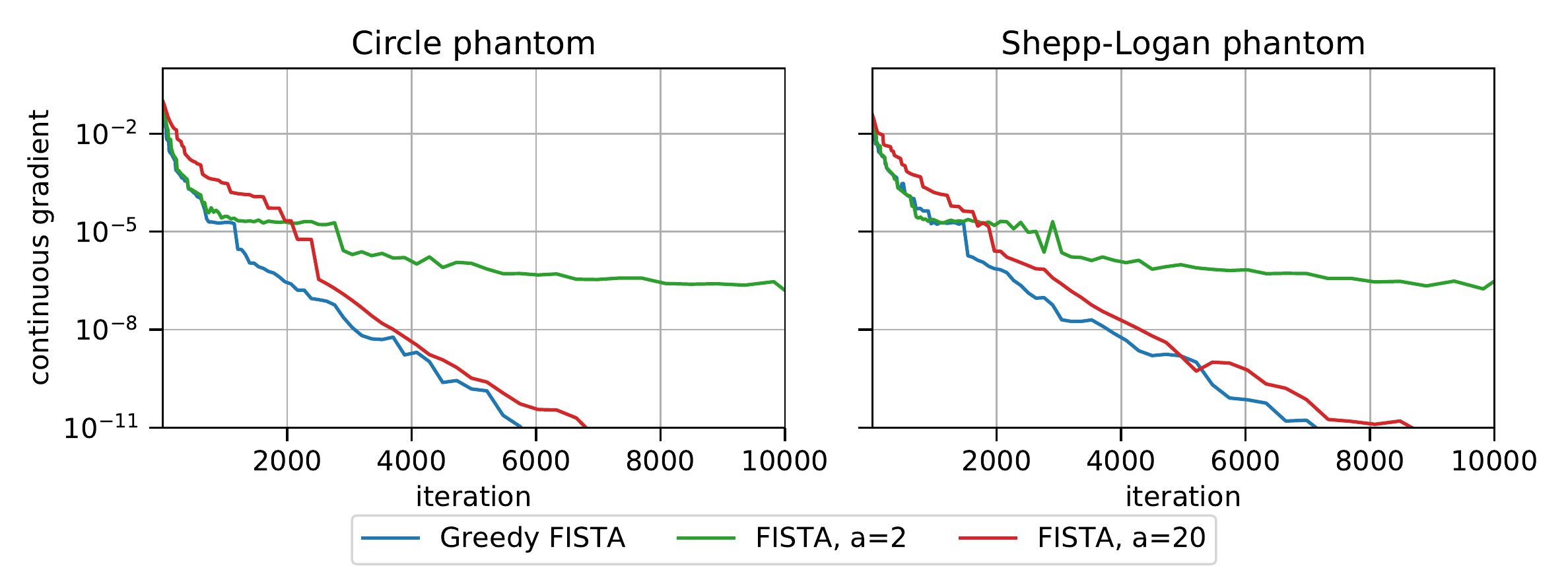}
		\caption{Convergence of different implementations of Algorithm~\ref{alg: refining FISTA} with an unlimited number of pixels for sparse wavelet optimisation.}\label{fig: haar convergence}
		
		\vspace*{\floatsep}
		
		\includegraphics[width=.85\textwidth]{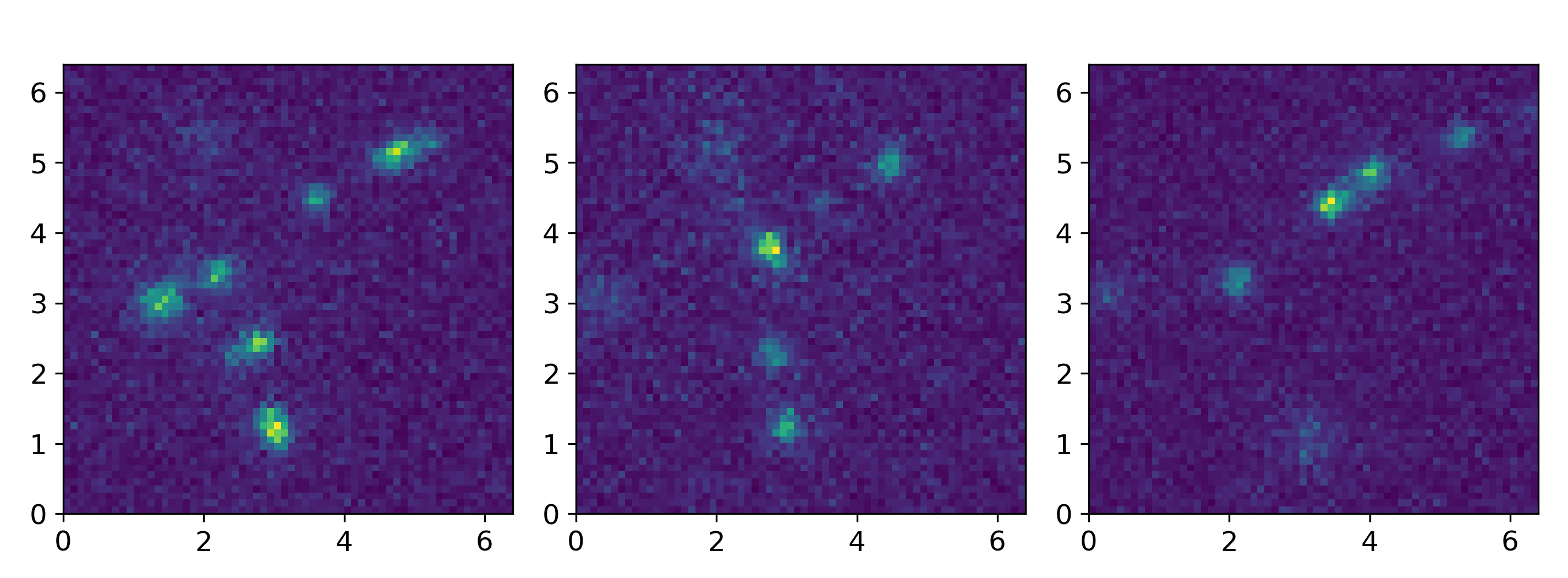}
		\caption{Example images from STORM dataset.}\label{fig: STORM data}
	\end{figure}
	
	\subsection{2D continuous LASSO}
	Our final application is a super-resolution/de-blurring inverse problem from biological microscopy. In mathematical terms, the observed data is a large number of sparse images which are corrupted by blurring and a large amount of noise, examples are seen in Fig.~\ref{fig: STORM data}. The task is to compute the centres of the spikes of signal in each image and then re-combine into a single super-resolved image, as in Fig.~\ref{fig: STORM results}. This technique is referred to as \emph{Single Molecule Localisation Microscopy} (SMLM), of which we consider the specific example of \emph{Stochastic Optical Reconstruction Microscopy} (STORM). Readers are directed to the references \cite{Sage2015,Sage2019,Schermelleh2019} for further details. The LASSO formulation ($\op{f}(\cdot)=\frac12\norm{\cdot}_{\ell^2}^2$) has previously been shown to be effective in the context of STORM \cite{Huang2017,Denoyelle2019}.
	
	Here we use a simulated dataset provided as part of the 2016 SMLM challenge\footnote{\url{http://bigwww.epfl.ch/smlm/challenge2016/datasets/MT4.N2.HD/Data/data.html}} for benchmarking software in this application. The corresponding LASSO formulation is 
	\begin{equation}
		(\A\var0)_{i} = (2\pi\sigma^2)^{-1}\int_{[0,6.4]^2} \exp\left(-\frac{1}{2\sigma^2}\left|\vec{x}- \Delta\begin{pmatrix}i_1+\tfrac12 & i_2+\tfrac12\end{pmatrix}^\top \right|^2\right)\var0(\vec{x})\diff\vec{x}, \qquad \sigma=0.2,\ \Delta=0.1
	\end{equation}
	for $i_1,i_2 = 1,2,\ldots,64$, $\F{U}=\C M([0,6.4]^2)$  and $\F H=L^2([0,6.4]^2)$ with lengths in $\SI{}{\micro\meter}$. 3020 frames are provided, examples of which are shown in Fig.~\ref{fig: STORM data}. To process this dataset, image intensities were normalised to $[0,1]$ then a constant was subtracted to approximate 0-mean noise. The greedy FISTA algorithm was used for optimisation with $\mu=0.15$, $10^3$ iterations, and a maximum of $10^5$ pixels per image. 
	
	Finally, all the reconstructions were summed and the result shown in Fig.~\ref{fig: STORM results}. The adaptive scheme used fewer than $10^4$ pixels per frame, a fixed discretisation with equivalent resolution of \SI{1.3}{\nano\meter} would have required more than $3\cdot10^6$ per frame. LASSO is compared with ThunderSTORM \cite{Ovesny2014}, a popular ImageJ plugin \cite{Schindelin2012} which finds the location of signal using Fourier filtering. The performance of ThunderSTORM was rated very highly in the initial SMLM challenge \cite{Sage2015}. Both methods compared here demonstrate the key structures of the reconstruction, however, both are sensitive to tuning parameters. In this examples, LASSO has possibly recovered too little signal and ThunderSTORM contains spurious signal. 
	
	Fig.~\ref{fig: STORM convergence} shows various convergence metrics for the adaptive reconstructions. The magenta line in the first panel shows that the continuous gap converges slightly faster than the $n^{-2/3}$ predicted by \eqref{eq: Lasso resolution rate} in dimension $d=2$. In this example we also implement the suggestion of Section~\ref{sec: support detection} to remove pixels outside of the support of $\var0^*$. From \eqref{eq: support equation}, any pixel $\domain\in\F M^n$ satisfying 
	\begin{equation}
		\vars2_0 \norm{ \tens{\Pi}_n\A^*\vec{\varphi}_n }_{L^\infty(\domain)} \leq (1-\op{threshold}_n)\mu
	\end{equation}
	guarantees that $\domain\cap\op{supp}(\var0^*) = \emptyset$. This threshold is plotted in red in the first panel of Fig.~\ref{fig: STORM convergence}. Once the value becomes less than 1, we can start reducing the number of pixels instead of continual refinement. We see that the resolution decreases steadily (second panel), but the total number of pixels (final panel) stops increasing after around 30 iterations.
	
	\begin{figure}[!b]\centering
		\includegraphics[width=.82\textwidth]{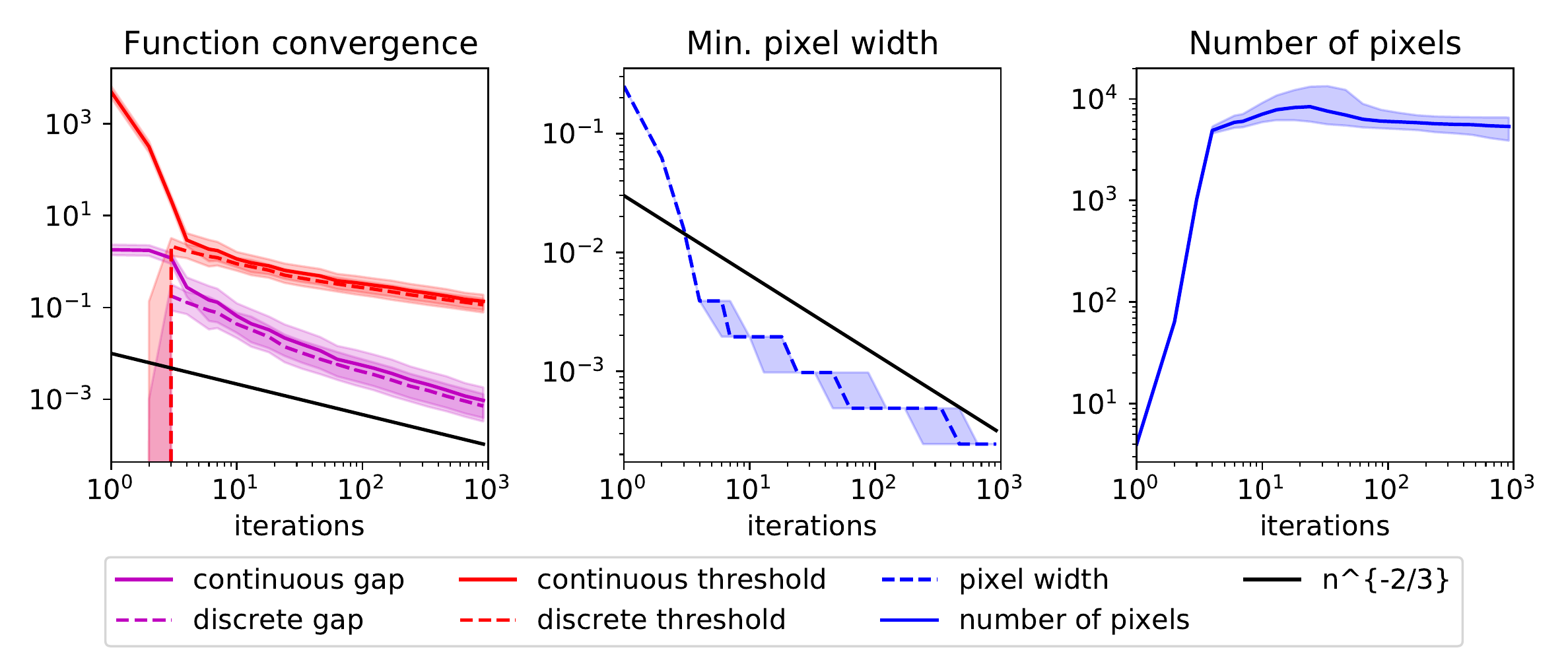}
		\caption{Convergence of adaptive FISTA for STORM dataset. Lines indicate the median value over 3020 STORM frames. Shaded regions indicate the \SIrange{25}{75}{\percent} interquartile range. Pixel width is scaled $[0,1]$ rather than $[0,\SI{6.4}{\micro\meter}]$.}\label{fig: STORM convergence}
		
		\vspace*{\floatsep}
		
		\includegraphics[width=.73\textwidth]{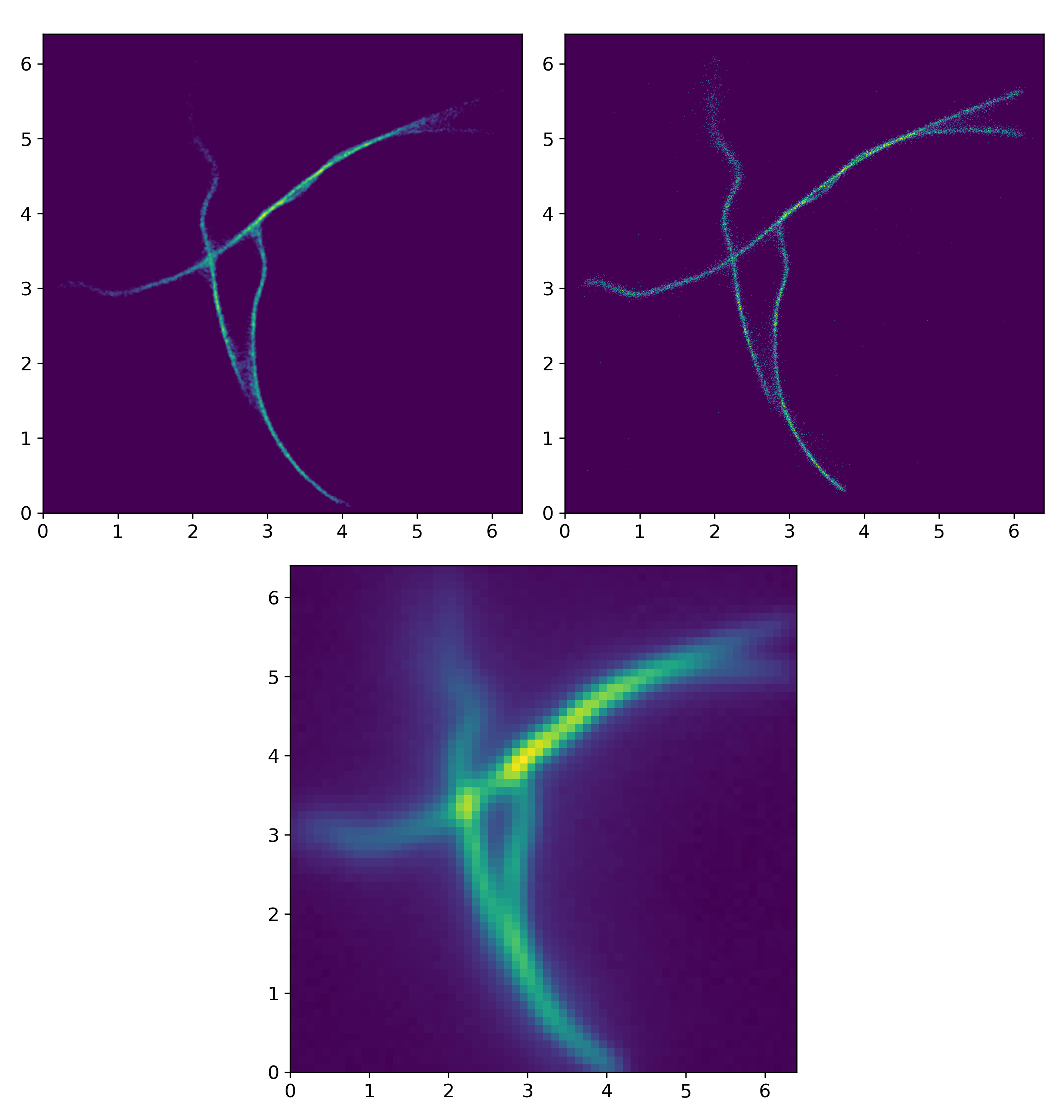}
		\caption{Processed results of the STORM dataset. Top left: LASSO optimisation with Algorithm~\ref{alg: refining FISTA}. Top right: Comparison with ThunderSTORM plugin. Bottom: Average data, no super-resolution or de-blurring.}\label{fig: STORM results}
	\end{figure}
	
	\section{Conclusions and outlook}
	In this work we have proposed a new adaptive variant of FISTA and provided convergence analysis. This algorithm allows FISTA to be applied outside of the classical Hilbert space setting, still with a guaranteed rate of convergence. We have presented several numerical examples where convergence with the refining discretisation is at least as fast as a uniform discretisation, although more efficient with regards to both memory and computation time. 
	
	In 1D we see good agreement with the theoretical rate. This rate also seems to be a good predictor for all variants of FISTA tested, although this is yet to be proven. Even the classical methods with a fixed discretisation are initially limited to the slower adaptive rate for small $n$.
	
	The results in 2D are similar, all tested FISTA methods converge at least at the guaranteed rate. The wavelet example was most impressive, achieving nearly linear convergence in energy. This is similar to the behaviour for classical FISTA although it is also yet to be formally proven.
	
	An interesting observation over all of the adaptive LASSO examples is that the standard oscillatory behaviour of FISTA has not occurred. With the monotone gaps plotted, oscillatory convergence should correspond to a piecewise constant descending gap. Either this behaviour only emerges for larger $n$, or the adaptivity provides a dampening effect for this oscillation.
	
	Moving forward, it would be interesting to see how far the analysis extends to other optimisation algorithms. Other variants of FISTA, such as the ``greedy'' implementation used here or the traditional Forward-Backward algorithm, should also be receptive to the analysis performed here. Furthermore, it would also interesting to attempt to replicate this refinement argument to extend the primal-dual algorithm \cite{Chambolle2011} or the Douglas-Rachford algorithm \cite{Douglas1956}.
	
	\begin{acknowledgements}
		R.T. acknowledges funding from EPSRC grant EP/L016516/1 for the Cambridge Centre for Analysis, and the ANR CIPRESSI project grant ANR-19-CE48-0017-01 of the French Agence Nationale de la Recherche. Most of this work was done while A.C. was still in CMAP, CNRS and Ecole Polytechnique, Institut Polytechnique de Paris, Palaiseau, France. Both authors would like to thank the anonymous reviewers who put in so much effort to improving this work.
	\end{acknowledgements}
	
	\begin{dataavailability}
		The synthetic STORM dataset was provided as part of the 2016 SMLM challenge, \url{http://bigwww.epfl.ch/smlm/challenge2016/datasets/MT4.N2.HD/Data/data.html}. The remaining examples used in this work can be generated with the supplementary code, \url{https://github.com/robtovey/2020SpatiallyAdaptiveFISTA}.
	\end{dataavailability}

	\begin{conflict}
		The authors have no conflicts of interest to declare which are relevant to the content of this article.
	\end{conflict}
	
	\bibliography{references}

\begin{thebibliography}{10}
\providecommand{\url}[1]{{#1}}
\providecommand{\urlprefix}{URL }
\expandafter\ifx\csname urlstyle\endcsname\relax
  \providecommand{\doi}[1]{DOI~\discretionary{}{}{}#1}\else
  \providecommand{\doi}{DOI~\discretionary{}{}{}\begingroup
  \urlstyle{rm}\Url}\fi

\bibitem{Alamo2019}
Alamo, T., Limon, D., Krupa, P.: Restart fista with global linear convergence.
\newblock In: 2019 18th European Control Conference (ECC), pp. 1969--1974. IEEE
  (2019)

\bibitem{Aujol2015}
Aujol, J.F., Dossal, C.: Stability of over-relaxations for the forward-backward
  algorithm, application to fista.
\newblock SIAM Journal on Optimization \textbf{25}(4), 2408--2433 (2015)

\bibitem{Beck2009}
Beck, A., Teboulle, M.: A fast iterative shrinkage-thresholding algorithm for
  linear inverse problems.
\newblock SIAM Journal on Imaging Sciences \textbf{2}(1), 183--202 (2009)

\bibitem{Bonnefoy2015}
Bonnefoy, A., Emiya, V., Ralaivola, L., Gribonval, R.: Dynamic screening:
  Accelerating first-order algorithms for the lasso and group-lasso.
\newblock IEEE Transactions on Signal Processing \textbf{63}(19), 5121--5132
  (2015)

\bibitem{Boyd2017}
Boyd, N., Schiebinger, G., Recht, B.: The alternating descent conditional
  gradient method for sparse inverse problems.
\newblock SIAM Journal on Optimization \textbf{27}(2), 616--639 (2017)

\bibitem{Boyer2019}
Boyer, C., Chambolle, A., Castro, Y.D., Duval, V., De~Gournay, F., Weiss, P.:
  On representer theorems and convex regularization.
\newblock SIAM Journal on Optimization \textbf{29}(2), 1260--1281 (2019)

\bibitem{Bredies2013}
Bredies, K., Pikkarainen, H.K.: Inverse problems in spaces of measures.
\newblock ESAIM: Control, Optimisation and Calculus of Variations
  \textbf{19}(1), 190--218 (2013)

\bibitem{Castillo2019}
Castillo, I., Rockova, V.: Multiscale analysis of bayesian cart.
\newblock University of Chicago, Becker Friedman Institute for Economics
  Working Paper (2019-127) (2019)

\bibitem{Catala2019}
Catala, P., Duval, V., Peyr{\'e}, G.: A low-rank approach to off-the-grid
  sparse superresolution.
\newblock SIAM Journal on Imaging Sciences \textbf{12}(3), 1464--1500 (2019)

\bibitem{Chambolle2015}
Chambolle, A., Dossal, C.: On the convergence of the iterates of the ``fast
  iterative shrinkage/thresholding algorithm''.
\newblock Journal of Optimization Theory and Applications \textbf{166}(3),
  968--982 (2015)

\bibitem{Chambolle2011}
Chambolle, A., Pock, T.: {A First-Order Primal-Dual Algorithm for Convex
  Problems with Applications to Imaging}.
\newblock Journal of Mathematical Imaging and Vision \textbf{40}(1), 120--145
  (2011).
\newblock \doi{10.1007/s10851-010-0251-1}

\bibitem{Castro2016}
De~Castro, Y., Gamboa, F., Henrion, D., Lasserre, J.B.: Exact solutions to
  super resolution on semi-algebraic domains in higher dimensions.
\newblock IEEE Transactions on Information Theory \textbf{63}(1), 621--630
  (2016)

\bibitem{Denoyelle2019}
Denoyelle, Q., Duval, V., Peyr{\'e}, G., Soubies, E.: The sliding frank--wolfe
  algorithm and its application to super-resolution microscopy.
\newblock Inverse Problems \textbf{36}(1), 014001 (2019)

\bibitem{Douglas1956}
Douglas, J., Rachford, H.H.: On the numerical solution of heat conduction
  problems in two and three space variables.
\newblock Transactions of the American Mathematical Society \textbf{82}(2),
  421--439 (1956)

\bibitem{Duval2017a}
Duval, V., Peyr{\'e}, G.: Sparse spikes super-resolution on thin grids i: the
  lasso.
\newblock Inverse Problems \textbf{33}(5), 055008 (2017)

\bibitem{Duval2017b}
Duval, V., Peyr{\'e}, G.: Sparse spikes super-resolution on thin grids ii: the
  continuous basis pursuit.
\newblock Inverse Problems \textbf{33}(9), 095008 (2017)

\bibitem{Ekanadham2011}
Ekanadham, C., Tranchina, D., Simoncelli, E.P.: Recovery of sparse
  translation-invariant signals with continuous basis pursuit.
\newblock IEEE transactions on signal processing \textbf{59}(10), 4735--4744
  (2011)

\bibitem{ElGhaoui2010}
El~Ghaoui, L., Viallon, V., Rabbani, T.: Safe feature elimination in sparse
  supervised learning.
\newblock Tech. Rep. UCB/EECS-2010--126, EECS Department, University of
  California, Berkeley (2010)

\bibitem{Hiriart2013}
Hiriart-Urruty, J.B., Lemar{\'e}chal, C.: Convex Analysis and Minimization
  Algorithms II: Advanced Theory and Bundle Methods, vol. 305.
\newblock Springer-Verlag, Berlin, Heidelberg (1993)

\bibitem{Huang2017}
Huang, J., Sun, M., Ma, J., Chi, Y.: Super-resolution image reconstruction for
  high-density three-dimensional single-molecule microscopy.
\newblock IEEE Transactions on Computational Imaging \textbf{3}(4), 763--773
  (2017)

\bibitem{Jiang2012}
Jiang, K., Sun, D., Toh, K.C.: An inexact accelerated proximal gradient method
  for large scale linearly constrained convex sdp.
\newblock SIAM Journal on Optimization \textbf{22}(3), 1042--1064 (2012)

\bibitem{Kekkonen2021}
Kekkonen, H., Lassas, M., Saksman, E., Siltanen, S.: Random tree besov
  priors--towards fractal imaging.
\newblock arXiv preprint arXiv:2103.00574  (2021)

\bibitem{Liang2017}
Liang, J., Fadili, J., Peyr{\'e}, G.: Activity identification and local linear
  convergence of forward--backward-type methods.
\newblock SIAM Journal on Optimization \textbf{27}(1), 408--437 (2017)

\bibitem{Liang2018}
Liang, J., Sch{\"o}nlieb, C.B.: Improving fista: Faster, smarter and greedier.
\newblock arXiv preprint arXiv:1811.01430  (2018)

\bibitem{Ndiaye2017}
Ndiaye, E., Fercoq, O., Gramfort, A., Salmon, J.: Gap safe screening rules for
  sparsity enforcing penalties.
\newblock The Journal of Machine Learning Research \textbf{18}(1), 4671--4703
  (2017)

\bibitem{Nesterov2004}
Nesterov, Y.: Introductory Lectures on Convex Optimization: A Basic Course.
\newblock Kluwer Academic Publishers Boston, Dordrecht, London (2004)

\bibitem{Ovesny2014}
Ovesn{\`y}, M., K{\v{r}}{\'\i}{\v{z}}ek, P., Borkovec, J., {\v{S}}vindrych, Z.,
  Hagen, G.M.: Thunderstorm: a comprehensive imagej plug-in for palm and storm
  data analysis and super-resolution imaging.
\newblock Bioinformatics \textbf{30}(16), 2389--2390 (2014)

\bibitem{Parpas2017}
Parpas, P.: A multilevel proximal gradient algorithm for a class of composite
  optimization problems.
\newblock SIAM Journal on Scientific Computing \textbf{39}(5), S681--S701
  (2017)

\bibitem{Poon2018}
Poon, C., Keriven, N., Peyr{\'e}, G.: The geometry of off-the-grid compressed
  sensing.
\newblock arXiv preprint arXiv:1802.08464  (2018)

\bibitem{Rosset2007}
Rosset, S., Zhu, J.: Piecewise linear regularized solution paths.
\newblock The Annals of Statistics pp. 1012--1030 (2007)

\bibitem{Sage2015}
Sage, D., Kirshner, H., Pengo, T., Stuurman, N., Min, J., Manley, S., Unser,
  M.: Quantitative evaluation of software packages for single-molecule
  localization microscopy.
\newblock Nature Methods \textbf{12}(8), 717--724 (2015)

\bibitem{Sage2019}
Sage, D., Pham, T.A., Babcock, H., Lukes, T., Pengo, T., Chao, J., Velmurugan,
  R., Herbert, A., Agrawal, A., Colabrese, S., et~al.: Super-resolution fight
  club: assessment of 2d and 3d single-molecule localization microscopy
  software.
\newblock Nature Methods \textbf{16}(5), 387--395 (2019)

\bibitem{Schermelleh2019}
Schermelleh, L., Ferrand, A., Huser, T., Eggeling, C., Sauer, M., Biehlmaier,
  O., Drummen, G.P.: Super-resolution microscopy demystified.
\newblock Nature cell biology \textbf{21}(1), 72--84 (2019)

\bibitem{Schindelin2012}
Schindelin, J., Arganda-Carreras, I., Frise, E., Kaynig, V., Longair, M.,
  Pietzsch, T., Preibisch, S., Rueden, C., Saalfeld, S., Schmid, B., et~al.:
  Fiji: an open-source platform for biological-image analysis.
\newblock Nature Methods \textbf{9}(7), 676--682 (2012)

\bibitem{Schmidt2011}
Schmidt, M., Roux, N.L., Bach, F.R.: Convergence rates of inexact
  proximal-gradient methods for convex optimization.
\newblock In: Advances in Neural Information Processing Systems, pp. 1458--1466
  (2011)

\bibitem{Strang1972}
Strang, G.: Approximation in the finite element method.
\newblock Numerische Mathematik \textbf{19}(1), 81--98 (1972)

\bibitem{Tao2016}
Tao, S., Boley, D., Zhang, S.: Local linear convergence of ista and fista on
  the lasso problem.
\newblock SIAM Journal on Optimization \textbf{26}(1), 313--336 (2016)

\bibitem{Unser2016}
Unser, M., Fageot, J., Gupta, H.: Representer theorems for sparsity-promoting
  $\ell^1$ regularization.
\newblock IEEE Transactions on Information Theory \textbf{62}(9), 5167--5180
  (2016)

\bibitem{Villa2013}
Villa, S., Salzo, S., Baldassarre, L., Verri, A.: Accelerated and inexact
  forward-backward algorithms.
\newblock SIAM Journal on Optimization \textbf{23}(3), 1607--1633 (2013)

\bibitem{Yu2021}
Yu, J., Lai, R., Li, W., Osher, S.: A fast proximal gradient method and
  convergence analysis for dynamic mean field planning.
\newblock arXiv preprint arXiv:2102.13260  (2021)

\end{thebibliography}
	\appendix
	\section{Proofs for FISTA convergence}\label{app: FISTA convergence}
	This section contains all of the statements and proofs of the results contained in Section~\ref{sec: FISTA convergence}. Recall that the subsets $\F{U}^n\subset\F H$ satisfy \eqref{eq: refining subspace definition}.
	
	\subsection{Proofs for Step 3}
	
	\begin{theorem}[Lemma~\ref{thm: one step FISTA}]\label{app:thm: one step FISTA}
		\Paste{thm: one step FISTA}
		\begin{equation}
			t_{n}^2(\op{E}(\var0_{n}) - \op{E}(\var2_n)) - (t_{n}^2-t_{n})(\op{E}(\var0_{n-1})-\op{E}(\var2_n)) \leq \tfrac1{2}\left[\norm{\var1_{n-1}}^2-\norm{\var1_{n}}^2\right] + \IP{\var1_{n}-\var1_{n-1}}{\var2_n}.
			\label{app:eq: one-step FISTA}
		\end{equation}
	\end{theorem}
	\begin{proof}
		Modifying \cite[Thm 3.2]{Chambolle2015}, for $n\geq1$ we apply Lemma~\ref{thm: descent lemma} with $\bar{\var0} = \bar{\var0}_{n-1}$ and $\var2 = (1-\frac{1}{t_{n}})\var0_{n-1} + \frac{1}{t_{n}}\var2_n$. By \eqref{eq: refining subspace definition}, $\var0_{n-1}\in\F{U}^n$ is convex so $\var2\in \F{U}^n$. This gives
		\begin{equation}
			\op{E}(\var0_{n}) + \tfrac12\norm{\tfrac{1}{t_{n}}\var1_{n}-\tfrac{1}{t_{n}}\var2_n}^2 \leq \op{E}\left((1-\tfrac{1}{t_{n}})\var0_{n-1}+\tfrac{1}{t_{n}}\var2_n\right) + \tfrac12\norm{\tfrac{1}{t_{n}}\var1_{n-1}-\tfrac{1}{t_{n}}\var2_n}^2.
		\end{equation}
		By the convexity of $\op{E}$, this reduces to
		\begin{equation}
			\op{E}(\var0_{n}) - \op{E}(\var2_n) - (1-\tfrac{1}{t_{n}})[\op{E}(\var0_{n-1})-\op{E}(\var2_n)] \leq \tfrac1{2t_{n}^2}\norm{\var1_{n-1}-\var2_n}^2 - \tfrac1{2t_{n}^2}\norm{\var1_{n}-\var2_n}^2 = \tfrac1{2t_{n}^2}\left[\norm{\var1_{n-1}}^2-\norm{\var1_{n}}^2\right] + \tfrac1{t_{n}^2}\IP{\var1_{n}-\var1_{n-1}}{\var2_n}.
		\end{equation}
		Multiplying through by $t_n^2$ gives the desired inequality.
		\qed\end{proof}

	\begin{theorem}[Theorem~\ref{thm: mini FISTA convergence}]\label{app:thm: mini FISTA convergence}
		\Paste{thm: mini FISTA convergence}
		\begin{equation}\label{app:eq: FISTA inequality}
			\Paste{thm:eq: mini FISTA convergence}
		\end{equation}
	\end{theorem}
	\begin{proof}
		Theorem~\ref{app:thm: mini FISTA convergence} is just a summation of \eqref{app:eq: one-step FISTA} over all $n=1,\ldots,N$. To see this: first add and subtract $\Emin$ to each term on the left-hand side to convert $\op{E}$ to $\op{E}_0$, then move $\op{E}_0(\var2_n)$ to the right-hand side. Now \eqref{app:eq: one-step FISTA} becomes
		\begin{equation}
			t_n^2\op{E}_0(\var0_n) - (t_n^2-t_n)\op{E}_0(\var0_{n-1}) \leq t_n\op{E}_0(\var2_n) + \tfrac1{2}\left[\norm{\var1_{n-1}}^2-\norm{\var1_n}^2\right] + \IP{\var1_n-\var1_{n-1}}{\var2_n}. 
		\end{equation}
		Summing this inequality from $n=1$ to $n=N$ gives
		\begin{equation}
			t_N^2\op{E}_0(\var0_N) + \sum_{n=1}^{N-1}(\underbrace{t_n^2-t_{n+1}^2+t_{n+1}}_{=\rho_n})\op{E}_0(\var0_n) \leq \frac{\norm{\var1_0}^2-\norm{\var1_N}^2}{2} + \sum_{n=1}^N t_n\op{E}_0(\var2_n)+\IP{\var1_n-\var1_{n-1}}{\var2_n}.
		\end{equation}
		The final step is to flip the roles of $\var1_n$/$\var2_n$ in the final inner product term. Re-writing the right-hand side gives
		\begin{equation}
			\sum_{n=1}^N\IP{\var1_n-\var1_{n-1}}{\var2_n} = \IP{\var1_N}{\var2_N} - \IP{\var1_0}{\var2_0}+ \sum_{n=1}^N\IP{\var1_{n-1}}{\var2_{n-1}-\var2_n}.
		\end{equation}
		Noting that $\var1_0=\var0_0$, the previous two equations combine to prove the statement of Theorem~\ref{app:thm: mini FISTA convergence}.
		\qed\end{proof}
	
	The following lemma is used to produce a sharper estimate on sequences $t_n$.
	\begin{lemma}\label{app: tn upper bound}
		If $\rho_n=t_{n}^2-t_{n+1}^2+t_{n+1}\geq0$, $t_n\geq 1$ for all $n\in\F N$ then $t_n\leq n-1 +t_1$.
	\end{lemma}
	\begin{proof}
		This is trivially true for $n=1$. Suppose true for $n-1$, the condition on $\rho_{n-1}$ gives
		\begin{equation}
			t_n^2 -t_n \leq t_{n-1}^2 \leq (n-2+t_1)^2 = (n-1+t_1)^2 -2(n-1+t_1) + 1.
		\end{equation}
		Assuming the contradiction, if $t_n> n-1+t_1$ then the above equation simplifies to $n-1+t_1 < 1$. However, $t_1\geq 1$ implying that $n<1$ which completes the contradiction.
		\qed\end{proof}

	\begin{lemma}[Lemma~\ref{thm: mini exponential FISTA convergence}]\label{app:thm: mini exponential FISTA convergence}
		\Paste{thm: mini exponential FISTA convergence}
		\begin{equation}
			\Paste{thm:eq: mini exponential FISTA convergence}
		\end{equation}
		\Paste{thm:end: mini exponential FISTA convergence}
	\end{lemma}
	\begin{proof}
		This is just a telescoping of the right-hand side of \eqref{app:eq: FISTA inequality} with the introduction of $n_k$ and simplification $\var2_n = \tilde{\var2}_k$,
		\begin{equation}
			\tfrac12\norm{\var2_N}^2 + \sum^N_{n=1} t_n\op{E}_0(\var2_n) + \IP{\var1_{n-1}}{\var2_{n-1}-\var2_n} = \tfrac12\norm{\tilde{\var2}_K}^2 + \sum_{n=n_K}^Nt_n\op{E}_0(\tilde{\var2}_K) 
			+ \sum_{k=1}^K\sum_{n=n_{k-1}}^{n_k-1}t_n\op{E}_0(\tilde{\var2}_{k-1}) 
			+ \IP{\var1_{n_k-1}}{\tilde{\var2}_{k-1}-\tilde{\var2}_k}.
		\end{equation}
		By Lemma~\ref{app: tn upper bound}, $t_n\leq n$ so we can further simplify
		$$\sum_{n=\vars3}^{\vars4-1}t_n \leq \sum_{n=\vars3}^{\vars4-1}n = (\vars4-\vars3)\frac{\vars4-1+\vars3}{2} \leq \frac{\vars4^2-\vars3^2}{2}$$
		to get the required bound.
		\qed\end{proof}
	
	\subsection{Proof for Step 4}
	\begin{lemma}[Lemma~\ref{thm: sufficiently fast}]\label{app:thm: sufficiently fast}
		\Paste{thm: sufficiently fast}
	\end{lemma}
	\begin{proof}
		Starting from Lemma~\ref{app:thm: mini exponential FISTA convergence} we have
		\begin{align}
			t_N^2\op{E}_0(\var0_N) + \tfrac12\norm{\var1_N-\tilde{\var2}_K}^2 &\leq C + \frac{\norm{\tilde{\var2}_K}^2}{2} + \frac{(N+1)^2-n_K^2}{2}\op{E}_0(\tilde{\var2}_K) \notag
			\\&\hspace{70pt}+ \sum_{k=1}^K \frac{n_k^2-n_{k-1}^2}{2}\op{E}_0(\tilde{\var2}_{k-1}) + \IP{\var1_{n_k-1}}{\tilde{\var2}_k-\tilde{\var2}_{k+1}}
			\\& \leq C + \frac{\norm{\tilde{\var2}_K}^2}{2} + \frac{n_{K+1}^2}{2}\op{E}_0(\tilde{\var2}_K) \notag
			\\&\hspace{40pt}+\sum_{k=1}^K \frac{n_k^2}{2}\op{E}_0(\tilde{\var2}_{k-1}) + \IP{\var1_{n_k-1}-\tilde{\var2}_{k-1}+\tilde{\var2}_{k-1}}{{\tilde{\var2}_{k-1}-\tilde{\var2}_k}}.
		\end{align}
		The inductive step now depends on the value of $\aU$.
		\begin{description}
			\item[Case $\aU>1$:] We simplify the inequality
			\begin{align}
				t_N^2\op{E}_0(\var0_N) + \tfrac12\norm{\var1_N-\tilde{\var2}_K}^2 &\lesssim \aU^{2K} + n_{K+1}^2\aE^{-K} + \sum_{k=1}^{K}n_k^2\aE^{-k} + \aU^{k}\norm{\var1_{n_k-1}-\tilde{\var2}_{k-1}} + \aU^{2k}
				\\&\leq C_1\left[\aU^{2K+2} + \sum_{k=1}^{K}\aU^{2k} + \aU^{k}\norm{\var1_{n_k-1}-\tilde{\var2}_{k-1}}\right]
			\end{align}
			for some $C_1>C$. Choose $C_2\geq \norm{\var1_{n_1-1}-\tilde{\var2}_{1-1}}\aU^{-1}$ such that 
			\begin{equation}
				\frac12C_2^2 \geq \frac{C_1}{\aU^2-1}(C_2+\aU^2).
			\end{equation}
			Assume $\norm{\var1_{n_k-1}-\tilde{\var2}_{k-1}}\leq C_2\aU^k$ for $1\leq k\leq K$ (trivially true for $K=1$), then for $N=n_{K+1}-1$ we have 
			\begin{align}
				\tfrac12\norm{\var1_{n_{K+1}-1}-\tilde{\var2}_K}^2 &\leq C_1\left[\aU^{2K+2} + \sum_{k=1}^{K}\aU^{2k} + \aU^{k}\norm{\var1_{n_k-1}-\tilde{\var2}_{k-1}}\right]
				\\&\leq C_1\left[\aU^{2K+2} + (1+C_2)\frac{\aU^{2K+2}}{\aU^2-1}\right]
				\\&\leq \frac{C_1\aU^{2K+2}}{\aU^2-1}\left(\aU^2+C_2\right) \leq \tfrac12(C_2\aU^{K+1})^2.
			\end{align}
			
			\item[Case $\aU=1$:] Denote $\vars4_k = \norm{\tilde{\var2}_k-\tilde{\var2}_{k+1}}$ and note that $\norm{\tilde{\var2}_{k-1}} \leq \norm{\tilde{\var2}_{0}} + \sum_0^\infty \vars4_k\lesssim 1$. We therefore bound
			\begin{align}
				t_N^2\op{E}_0(\var0_N) + \tfrac12\norm{\var1_N-\tilde{\var2}_K}^2 &\lesssim 1 + n_{K+1}^2\aE^{-K} + \sum_{k=1}^{K}n_k^2\aE^{-k} + (\norm{\var1_{n_k-1}-\tilde{\var2}_{k-1}} + 1)b_k
				\\&\leq C_1\left[1 + \sum_{k=1}^{K}\norm{\var1_{n_k-1}-\tilde{\var2}_{k-1}}\vars4_{k-1}\right]
			\end{align}
			for some $C_1>0$. Choose $C_2\geq \frac{\norm{\var1_{n_1-1}-\tilde{\var2}_{1-1}}}{\sum_0^\infty \vars4_k}$ such that 
			\begin{equation}
				\frac12C_2^2 \geq C_1\left(1+C_2\sum_0^\infty\vars4_k\right).
			\end{equation}
			Assume $\norm{\var1_{n_k-1}-\tilde{\var2}_{k-1}}\leq C_2$ for $1\leq k\leq K$ (trivially true for $K=1$), then for $N=n_{K+1}-1$ we have 
			\begin{equation}
				\tfrac12\norm{\var1_{n_{K+1}-1}-\tilde{\var2}_K}^2 \leq C_1\left[1 + \sum_{k=1}^{K}\norm{\var1_{n_k-1}-\tilde{\var2}_{k-1}}\vars4_{k-1}\right]
				\leq C_1\left(1+C_2\sum_0^\infty\vars4_k\right) \leq \frac{C_2^2}{2}
			\end{equation}
		\end{description}
		In both cases, the induction on $\norm{\var1_{n_{K+1}-1}-\tilde{\var2}_K}$ holds for all $K$, and we have $t_N^2\op{E}_0(\var0_N) \leq \frac12C_2^2\aU^{2K}$ for all $N<n_K-1$.
		\qed\end{proof}
	
	\subsection{Proof for Step 5}
	\begin{lemma}[Lemma~\ref{thm: sufficiently slow}]\label{app:thm: sufficiently slow}
		\Paste{thm: sufficiently slow}
	\end{lemma}
	\begin{proof}
		The proof is direct computation, note that
		\begin{equation}\label{eq: aEaU^kappa}
			(\aE\aU^2)^\kappa = \exp\left(\kappa\log(\aE\aU^2)\right) = \exp(\log\aU^2) = \aU^2,
		\end{equation}
		therefore 
		\begin{equation}
			\aU^{2K} = \left((\aE\aU^2)^K\right)^\kappa \lesssim n_K^{2\kappa} \leq N^{2\kappa},
		\end{equation}
		so $\op{E}_0(\var0_N)\lesssim N^{-2(1-\kappa)}$
		as required.
		\qed\end{proof}

	\subsection{Proofs for Step 6}
	\begin{theorem}[Theorem~\ref{thm: stronger exponential FISTA convergence}]\label{app:thm: stronger exponential FISTA convergence}
		\Paste{thm: stronger exponential FISTA convergence}
	\end{theorem}
	\begin{proof}
		Let $C>0$ satisfy $n_k^2\leq C\aE^k\aU^{2k}$ for each $k\in\F N$. Fix $N>C$ and choose $k$ such that $C\aE^{k-1}\aU^{2k-2}\leq N<C\aE^k\aU^{2k}$. By construction, and using the equality from \eqref{eq: aEaU^kappa}, we have 
		\begin{equation}
			\min_{n\leq N}\op{E}_0(\var0_N)\leq \op{E}_0(\tilde{\var2}_{k-1}) \lesssim \aE^{-k} = (\aE\aU^2)^{-k(1-\kappa)}<C^{\kappa-1}N^{-2(1-\kappa)}
		\end{equation}
		as required.
		\qed\end{proof}

	\begin{lemma}[Lemma~\ref{thm: practical refinement criteria}]\label{app:thm: practical refinement criteria}
		\Paste{thm: practical refinement criteria}
	\end{lemma}
	\begin{proof}
		The conditions for $\aU$ in Definition~\ref{def: exp subspaces} are already met, it remains to be shown that $\op{E}_0(\tilde{\var2}_{k}) \leq C\aE^{-k}$ for some fixed $C>0$. For cases (3) and (4), fix $R>0$ such that both $\{\tilde{\var2}_k\}_{k\in\F N}$ and the sublevel set $ \{\var0\in\F{U}\st\op{E}_0(\var0)\leq 1+\vars1\} $ are contained in the ball of radius $R$. Any minimising sequences of $\op{E}$ in $\F{U}$ or $\tilde{\F{U}}^k$ are contained in this ball. We can therefore compute $C$ in each case:
		\begin{itemize}
			\item[(1)] $\op{E}_0(\tilde{\var2}_k)\leq \vars1\aE^{-k}$, so $C=\vars1$ suffices.
			\item[(2)] $\op{E}_0(\tilde{\var2}_k)\leq \op{E}_0(\tilde{\F{U}}^k) + \vars1\aE^{-k} \leq (\aE +1)\beta\aE^{-k}$, so $C=(\aE +1)\beta$ suffices.
			\item[(3)] $\op{E}_0(\tilde{\var2}_k)-\op{E}_0(\var0)\leq \inf_{\var1\in\partial\op{E}(\tilde{\var2}_k)}\IP{\var1}{\tilde{\var2}_k-\var0}\leq 2R\vars1\aE^{-k}$ for any $\var0\in\F{U}$ with $\Norm{\var0}\leq R$. Maximising over $\var0$ gives $C=2R\vars1$
			\item[(4)] $\op{E}_0(\tilde{\var2}_k)-\op{E}_0(\var0)\leq \inf_{\var1\in\partial\op{E}(\tilde{\var2}_k)}\IP{\var1}{\tilde{\var2}_k-\var0}\leq 2R\vars1\aE^{-k}$ for any $\var0\in\tilde{\F{U}}^k$ with $\Norm{\var0}\leq R$, so $\op{E}_0(\tilde{\var2}_k) \leq \op{E}_0(\tilde{\F{U}}^k) + 2R\vars1\aE^{-k}$ and $C=(1+2R)\vars1$.
		\end{itemize}
		This completes the requirements of Definition~\ref{def: exp subspaces}.
		\qed\end{proof}

	\section{Proof of \texorpdfstring{Theorem~\ref{thm: generic a_U and a_E}}{rates for general finite element spaces}}\label{app: generic a_U and a_E}
	First we recall the setting of Definition~\ref{def: finite element space}, fix: $p\geq0$, $q\in[1,\infty]$, $\meshsize\in(0,1)$, $N\in\F N$, connected and bounded domain $\Domain\subset\R^d$, and $\var0^*\in\argmin_{\var0\in\F U}\op{E}(\var0)$. We assume that $\F H=L^2(\Domain)$, $\norm{\cdot}_q\lesssim\Norm\cdot$, and there exist spaces $(\tilde{\F U}^k)_{k\in\F N}$ with $\tilde{\F{U}}^k\subset \F{U}$ containing a sequence $(\tilde{\var2}_k\in \tilde{\F U}^k)_{k\in\F N}$ such that $\Norm{\tilde{\var2}_k - \var0^*} \lesssim \meshsize^{kp}$, c.f. \eqref{eq: projection of u*}. Furthermore, there exists constant $c_\alpha>0$ and meshes $\F M^k$ such that:
	\begin{gather}
		\exists\domain_0\subset\Domain\quad\text{such that}\quad \forall\domain\in\F{M}^k \quad \exists (\vars0_\domain,\vvars1_\domain)\in\R^{d\times d}\times\R^d \quad\text{such that}\quad \vec{x}\in\domain_0\iff \vars0_\domain\vec{x}+\vvars1_\domain\in\domain, \qquad\text{ and}
		\\\forall (\tilde{\var0},\domain)\in\tilde{\F{U}}^k\times\F M^k,\quad \exists \var0\in\tilde{\F U}^0 \quad \text{such that}\quad \op{det}(\alpha_\domain)\geq c_\alpha\meshsize^{kd} \quad\text{and}\quad \forall\vec{x}\in\domain_0,\ \var0(\vec x) = \tilde{\var0}(\alpha_\domain \vec{x}+\vec\beta_\domain). \label{eq: decomposition property}
	\end{gather}
	In this section, these assumptions will be summarised simply by saying that $\F H$ and $(\tilde{\F U}^k)_{k\in\F N}$ satisfy Definition~\ref{def: finite element space}.
	We prove Theorem~\ref{thm: generic a_U and a_E} as a consequence of Lemma~\ref{thm: pq to p'q'}, namely we compute exponents $p',q'$ with $\aU=\meshsize^{-q'}$ and $\aE=\meshsize^{-p'}$. These values are computed as the result of the following three lemmas. The first, Lemma~\ref{app:thm: generic a_U bound}, is a quantification of the equivalence between $L^q$ and $L^2$ norms on general sub-spaces. Lemma~\ref{app:thm: specific a_U bound} applies this result to finite-element spaces to compute the value of $q'$. Finally, Lemma~\ref{app:thm: generic a_E bound} then performs the computations for $p'$ depending on the smoothness properties of $\op{E}$.
	
	\begin{lemma}[Equivalence of norms for fixed $k$]\label{app:thm: generic a_U bound}
		Suppose $\F H= L^2(\Domain)$ for some connected, bounded domain $\Domain\subset\R^d$ and $\norm\cdot_q\leq C \Norm\cdot$ for some $q\in[1,\infty]$, $C>0$. 
		For any linear subspace $\tilde{\F{U}}\subset \F{U}$ and $\tilde{\var2}\in\tilde{\F U}$,
		\begin{equation}
				\norm{\tilde{\var2}}\leq \sup_{\var0,\tilde{\var0}\in\tilde{\F{U}}}\frac{\IP{\var0}{\tilde{\var0}}}{\norm{\var0}\, \Norm{\tilde{\var0}}}\Norm{\tilde{\var2}} \quad\text{where}\quad \sup_{\var0,\tilde{\var0}\in\tilde{\F{U}}}\frac{\IP{\var0}{\tilde{\var0}}}{\norm{\var0}\, \Norm{\tilde{\var0}}} \leq C^{-1}\splitln{|\Domain|^{\frac12-\frac1q}}{\text{ if }q\geq 2\text{, otherwise}}{|\Domain|^{1-\frac1q}\sup_{\var0\in\tilde{\F{U}}}\norm{\var0}_\infty/\norm{\var0}\qquad}{\text{ if }q\in[1,2).}
		\end{equation}
		
	\end{lemma}
	\begin{proof}	
		The first statement of the result is by definition, for each $\tilde{\var2}\in\tilde{\F{U}}\subset L^\infty(\Domain)\subset \F H$ we have
		$$ \norm{\tilde{\var2}} = \frac{\IP{\tilde{\var2}}{\tilde{\var2}}}{\norm{\tilde{\var2}}} \leq \sup_{\var0\in\tilde{\F{U}}} \frac{\IP{\var0}{\tilde{\var2}}}{\norm{\var0}\,\Norm{\tilde{\var2}}}\Norm{\tilde{\var2}} \leq \sup_{\var0,\tilde{\var0}\in\tilde{\F{U}}} \frac{\IP{\var0}{\tilde{\var0}}}{\norm{\var0}\,\Norm{\tilde{\var0}}}\Norm{\tilde{\var2}}. $$
		Recall $\Norm\cdot\geq C^{-1}\norm\cdot_q$. To go further we use H\"older's inequality. If $\frac1q+\frac{1}{q^*}=1$, then for any $\var0,\tilde{\var0}\in\tilde{\F U}$
		\begin{equation}
			\frac{\IP{\var0}{\tilde{\var0}}}{\norm{\var0}\,\Norm{\tilde{\var0}}} \leq C^{-1}\frac{\IP{\var0}{\tilde{\var0}}}{\norm{\var0}\,\norm{\tilde{\var0}}_q} 
			\leq C^{-1}\frac{\norm{\var0}_{q^*}}{\norm{\var0}}.
		\end{equation}
		If $q\geq2$ we use H\"older's inequality a second time:
		\begin{equation}
			\int_\Domain |\var0(\vec x)|^{q^*}\diff\vec{x} \leq \left(\int_\Domain 1\diff\vec{x}\right)^{1-q^*/2}\left(\int_\Domain |\var0(\vec x)|^2\diff\vec{x}\right)^{q^*/2} = \left(|\Domain|^{\frac{1}{2} -\frac{1}{q}}\norm{\var0}\right)^{q^*}.
		\end{equation}
		This confirms the inequality when $q\geq2$. If $q<2$, we can simply upper bound $\norm\cdot_{q^*}\leq |\Domain|^{\frac1{q^*}}\norm\cdot_\infty$ as required.
		\qed\end{proof}
	
	\begin{lemma}\label{app:thm: specific a_U bound}
		Suppose $\F H$ and $(\tilde{\F U}^k)_{k\in\F N}$ satisfy Definition~\ref{def: finite element space}, then
		\begin{enumerate}
			\item If $q\geq 2$, then $\norm{\tilde{\var2}_k} \lesssim 1$ (i.e. $q'=0$).
			\item If $q<2$ and $\sup_{\var0\in\tilde{\F{U}}^0}\frac{\norm{\var0}_{L^\infty(\domain_0)}}{\norm{\var0}_{L^2(\domain_0)}}<\infty$, then $\norm{\tilde{\var2}_k} \lesssim \meshsize^{-\frac{kd}{2}}$ (i.e. $q'=-\frac d2$).
		\end{enumerate}
	\end{lemma}
	\begin{proof}
		Most of the conditions of Lemma~\ref{app:thm: generic a_U bound} are already satisfied. Furthermore observe that $\Norm{\tilde{\var2}_k}\lesssim \Norm{\var0^*}+\meshsize^{kp}\lesssim 1$. For the $q\geq2$ case, this is already sufficient to conclude $\norm{\tilde{\var2}_k}\lesssim 1$ from Lemma~\ref{app:thm: generic a_U bound}, as required.

		For the case $q<2$, from Lemma~\ref{app:thm: generic a_U bound} recall that we are required to bound
		\begin{equation}
			\sup_{\tilde{\var0}\in\tilde{\F{U}}^k}\frac{\norm{\tilde{\var0}}_\infty}{\norm{\tilde{\var0}}} = \sup_{\tilde{\var0}\in\tilde{\F{U}}^k}\sup_{\domain\in\F M^k} \frac{\norm{\tilde{\var0}}_{L^\infty(\domain)}}{\norm{\tilde{\var0}}_{L^2(\Domain)}} \leq \sup_{\tilde{\var0}\in\tilde{\F{U}}^k}\sup_{\domain\in\F M^k} \frac{\norm{\tilde{\var0}}_{L^\infty(\domain)}}{\norm{\tilde{\var0}}_{L^2(\domain)}}.
		\end{equation}
		However, due to the decomposition property \eqref{eq: decomposition property}, for each $\domain\in\F M^k$ and $\tilde{\var0}\in\tilde{\F{U}}^k$ there exists $\var0\in\tilde{\F{U}}^0$ such that
		\begin{equation}
			\norm{\var0}_{L^\infty(\domain_0)} = \norm{\tilde{\var0}}_{L^\infty(\domain)},\qquad  \norm{\var0}_{L^2(\domain_0)}^2 = \int_{\domain_0} |\var0(\vec x)|^2\diff\vec{x} = \int_{\domain_0} |\tilde{\var0}(\alpha\vec x+\vec\beta)|^2\diff\vec{x} = \op{det}(\alpha)^{-1}\norm{\tilde{\var0}}_{L^2(\domain)}^2.
		\end{equation}
		Combining these two equations with the assumed bound on $\frac{\norm{\var0}_{L^\infty(\domain_0)}}{\norm{\var0}_{L^2(\domain_0)}}$ confirms $\norm{\tilde{\var2}_k} \lesssim \sqrt{\op{det}(\alpha)^{-1}} \leq c_\alpha^{-\frac12}\meshsize^{-\frac{kd}{2}}$ as required.
		\qed\end{proof}
	
	\begin{lemma}\label{app:thm: generic a_E bound}
		Suppose $\F H$ and $(\tilde{\F U}^k)_{k\in\F N}$ satisfy Definition~\ref{def: finite element space} and $\var0^*$ is the minimiser of $E$ such that $\Norm{\tilde{\var2}_k-\var0^*}\lesssim \meshsize^{kp}$.
		\begin{enumerate}
			\item If $\op{E}$ is $\Norm\cdot$-Lipschitz at $\var0^*$, then
			$\op{E}(\tilde{\var2}_k)-\op{E}(\var0^*) \lesssim \meshsize^{kp}$  (i.e. $p'=p$).
			\item If $\nabla\op{E}$ is $\Norm\cdot$-Lipschitz at $\var0^*$, then 
			$\op{E}(\tilde{\var2}_k)-\op{E}(\var0^*) \lesssim \meshsize^{2kp}$  (i.e. $p'=2p$).
		\end{enumerate}
	\end{lemma}
	\begin{proof}
		Both statements are direct by definition, observe
		\begin{gather}
			\op{E}(\tilde{\var2}_k)-\op{E}(\var0^*) \leq \op{Lip}(\op{E})\Norm{\tilde{\var2}_k-\var0^*},
			\\\op{E}(\tilde{\var2}_k)-\op{E}(\var0^*) \leq \IP{\nabla\op{E}(\var2)}{\tilde{\var2}_k-\var0^*} = \IP{\nabla \op{E}(\tilde{\var2}_k)-\nabla \op{E}(\var0^*)}{ \tilde{\var2}_k-\var0^*} \leq \op{Lip}(\nabla\op{E})\Norm{\tilde{\var2}_k-\var0^*}^2.
		\end{gather}
		The proof is concluded by using the approximation bounds of $\tilde{\var2}_k$ in Definition~\ref{def: finite element space}.
		\qed\end{proof}

	\section{Operator norms for numerical examples}
	\begin{theorem}\label{thm: norm bound examples}
		Suppose $\A\colon \F H \to \R^m$ has kernels $\psi_j\in L^\infty([0,1]^d)$ for $j\in[m]$.
		\begin{enumerate}
			\item[Case 1:] If $\psi_j(\vec{x}) = \splitln{1}{\vec{x}\in \F{X}_j}{0}{\text{ else}}$ for some collection $\F{X}_j\subset\Domain$ such that $\F{X}_i\cap \F{X}_j = \emptyset$ for all $i\neq j$, then 
			$\norm{\A}_{L^2\to\ell^2} = \max_{j\in[m]} \sqrt{|\F{X}_j|}.$
			
			\item[Case 2:] If $\psi_j(\vec{x}) = \cos(\ip{\vvars3_j}{\vec{x}})$ for some frequencies $\vvars3_j\in\R^d$ with $|\vvars3_j|\leq A$, then
			$$\norm{\A}_{L^2\to\ell^2} \leq \sqrt{m}, \qquad |\A^*\vec{r}|_{C^k}\leq m^{1-\frac1q}A^k\norm{\vec{r}}_q, \quad\text{and}\quad |\A^*|_{\ell^2\to C^k}\leq \sqrt{m}A^k$$
			for all $\vec{r}\in\R^m$ and $q\in[1,\infty]$.
			
			\item[Case 3:] Suppose $\psi_j(\vec{x}) = (2\pi\sigma^2)^{-\frac{d}{2}}\exp\left(-\frac{|\vec{x}-\vec{x}_j|^2}{2\sigma^2}\right)$ for some regular mesh $\vec{x}_j\in[0,1]^d$ and separation $\Delta$. i.e. 
			$$\{\vec{x}_j\st j\in[m]\} = \{\vec{x}_0 + (j_1\Delta,\ldots,j_d\Delta)\st j_i\in[\hat m]\}$$
			for some $\vec{x}_0\in\R^d$, $\hat m\coloneqq\sqrt[d]{m}$. For all $\frac1q + \frac{1}{q^*} = 1$, $q\in(1,\infty]$, we have
			\begin{align}
				\norm{\A}_{L^2\to\ell^2} &\leq \bigg((4\pi\sigma^2)^{-\frac 12}\sum_{j=-2\hat m,\ldots,2\hat m}\exp(-\tfrac{\Delta^2}{4\sigma^2}j^2)\bigg)^d,
				\\ |\A^*\vec{r}|_{C^0} &\leq (2\pi\sigma^2)^{-\frac{d}{2}}\bigg(\sum_{\vec j\in J}\exp\left(-\tfrac{q^*\Delta^2}{2\sigma^2}\max(0,|\vec j|-\delta)^2\right)\bigg)^{\frac1{q^*}}\norm{\vec{r}}_q,
				\\ |\A^*\vec{r}|_{C^1} &\leq \frac{(2\pi\sigma^2)^{-\frac{d}{2}}}{\sigma}\frac{\Delta}{\sigma}\bigg(\sum_{\vec j\in J}(|\vec j|+\delta)^{q^*}\exp\left(-\tfrac{q^*\Delta^2}{2\sigma^2}\max(0,|\vec j|-\delta)^2\right)\bigg)^{\frac1{q^*}}\norm{\vec{r}}_q,
				\\ |\A^*\vec{r}|_{C^2} &\leq \frac{(2\pi\sigma^2)^{-\frac{d}{2}}}{\sigma^2} \bigg(\sum_{\vec j\in J} \left(1+\tfrac{\Delta^2}{\sigma^2}(|\vec j|+\delta)^2\right)^{q^*}\exp\left(-\tfrac{q^*\Delta^2}{2\sigma^2}\max(0,|\vec j|-\delta)^2\right)\bigg)^{\frac1{q^*}} \norm{\vec{r}}_q,
			\end{align}
			where $\delta= \frac{\sqrt d}{2}$ and $J=\{\vec j\in\F Z^d \st \norm{\vec j}_{\ell^\infty}\leq 2\hat m\}$. The case for $q=1$ can be inferred from the standard limit of $\norm\cdot_{{q^*}}\to \norm\cdot_{\infty}$ for $q^*\to\infty$.
		\end{enumerate}
	\end{theorem}
	\begin{proof}[Case 1.]
		From Lemma~\ref{thm: norm bound L2} we have 
		\begin{equation}
			(\A\A^*)_{i,j} = \IP{\1_{\F{X}_i}}{\1_{\F{X}_j}} = |\F{X}_i\cap \F{X}_j| = \splitln{|\F{X}_i|}{i=j}{0}{i\neq j}.
		\end{equation}
		Therefore, $\A\A^*$ is a diagonal matrix and $\norm{\A\A^*}_{\ell^2\to\ell^2} = \max_{j\in[m]} |\F{X}_j|$ completes the result.
		\qed\end{proof}
	
	\begin{proof}[Case 2.]
		$\psi_j$ are not necessarily orthogonal however $|\IP{\psi_i}{\psi_j}|\leq 1$ therefore we can estimate
		\begin{equation}
			\norm{\A\A^*}_{\ell^2\to\ell^2} \leq \norm{\A\A^*}_{\ell^\infty\to\ell^\infty} \leq m.
		\end{equation}
		Now looking to apply Lemma~\ref{thm: norm bound smoothness}, note $\norm{\nabla^k\psi_j}_\infty \leq A^k$, therefore
		\begin{equation}
			|\A^*\vec{r}|_{C^k}\leq A^k m^{\frac1{q^*}}\norm{\vec{r}}_q = A^k m^{1-\frac1{q}}\norm{\vec{r}}_q
			\quad\text{and}\quad |\A^*|_{\ell^2\to C^k} \leq A^k \min_{q\in[1,\infty]}m^{1-\frac1q}\sqrt{m}^{\max(0,2-q)} = \sqrt{m}A^k.
		\end{equation}
		\qed\end{proof}
	
	\begin{proof}[Case 3.]
		In the Gaussian case, we build our approximations around the idea that sums of Gaussians should converge very quickly. The first example can be used to approximate the operator norm. Computing the inner products gives
		\begin{equation}
			\IP{\psi_i}{\psi_j} = (2\pi\sigma^2)^{-d}\int_{[0,1]^d} \exp\left(-\tfrac{|\vec{x}-\vec{x}_i|^2}{2\sigma^2}-\tfrac{|\vec{x}-\vec{x}_j|^2}{2\sigma^2}\right)\diff\vec{x} \leq (2\pi\sigma^2)^{-d}(\pi\sigma^2)^{\frac{d}{2}}\exp\left(-\tfrac{|\vec{x}_i-\vec{x}_j|^2}{4\sigma^2}\right).
		\end{equation}
		Estimating the operator norm,
		\begin{align}
			\norm{\A\A^*}_{\ell^2\to\ell^2}&\leq \norm{\A\A^*}_{\ell^\infty\to\ell^\infty} = \max_{i\in[m]} \sum_{j=1}^m |\IP{\psi_i}{\psi_j}| 
			\\&= \max_{i\in[m]}(4\pi\sigma^2)^{-\frac d2}\sum_{j_1,\ldots,j_d\in[\hat m]} \exp\left(-\frac{(j_1\Delta-i_1\Delta)^2+\ldots+(j_d\Delta-i_d\Delta)^2}{4\sigma^2}\right)
			\\&\leq (4\pi\sigma^2)^{-\frac d2}\sum_{\vec j\in\F Z^d\cap[-\hat m,\hat m]^d} \exp\left(-\frac{(j_1\Delta)^2+\ldots+(j_d\Delta)^2}{4\sigma^2}\right)
			= \left[(4\pi\sigma^2)^{-\frac 12}\sum_{j=-\hat m}^{\hat m} \exp\left(-\frac{\Delta^2j^2}{4\sigma^2}\right)\right]^d.
		\end{align}
		This is a nice approximation because it factorises simply over dimensions.
		Applying the results from Lemma~\ref{thm: norm bound smoothness}, note
		$$\begin{array}{rll}
			\displaystyle |\psi_j(\vec{x})| &\displaystyle= \left|\psi_j(\vec{x})\right| &\displaystyle = (2\pi\sigma^2)^{-\frac d2}\exp\left(-\frac{|\vec{x}-\vec{x}_j|^2}{2\sigma^2}\right),
			\\\displaystyle |\nabla\psi_j(\vec{x})| &\displaystyle= \left|\frac{\vec{x}-\vec{x}_j}{\sigma^2}\psi_j(\vec{x})\right| &\displaystyle = \frac{(2\pi\sigma^2)^{-\frac d2}}{\sigma}\frac{|\vec{x}-\vec{x}_j|}{\sigma}\exp\left(-\frac{|\vec{x}-\vec{x}_j|^2}{2\sigma^2}\right),
			\\|\nabla^2\psi_j(\vec{x})| &\displaystyle =\left|\frac1{\sigma^2} + \frac{(\vec{x}-\vec{x}_j)(\vec{x}-\vec{x}_j)^\top }{\sigma^4}\right|\psi_j(\vec{x}) &\displaystyle = \frac{(2\pi\sigma^2)^{-\frac d2}}{\sigma^2}\left(1+\frac{|\vec{x}-\vec{x}_j|^2}{\sigma^2}\right)\exp\left(-\frac{|\vec{x}-\vec{x}_j|^2}{2\sigma^2}\right).
		\end{array}$$
		We now wish to sum over $j=1,\ldots,m$ and produce an upper bound on these, independent of $t$. To do so we will use the following lemma.
		
		\begin{lemma}\label{app: exp sum bound}
			Suppose $q>0$. If the polynomial $p(|\vec{x}|) = \sum p_k|\vec{x}|^k$ has non-negative coefficients and $\vec{x}\in[-m,m]^d$, then
			$$\sum_{\norm{\vec{j}}_{\ell^\infty}\leq m} p(|\vec{j}-\vec{x}|)\exp\left(-\tfrac{q|\vec{j}-\vec{x}|^2}{2}\right)\leq \sum_{\norm{\vec{j}}_{\ell^\infty}\leq2m}p(|\vec{j}|+\delta)\exp\left(-\frac{q\max(0,|\vec{j}|-\delta)^2}{2}\right)$$
			where $\delta\coloneqq \frac{\sqrt{d}}{2}$ and $\vec{j}\in\F Z^d$.
		\end{lemma}
		\begin{proof}
			There exists $\hat{\vec{x}}\in[-\tfrac12,\tfrac12]^d$ such that $\vec{x} + \hat{\vec{x}}\in\F Z^d$, therefore
			\begin{align*}
				\sum_{\norm{\vec{j}}_{\ell^\infty}\leq m} p(|\vec{j}-\vec{x}|)\exp\left(-\tfrac{q|\vec{j}-\vec{x}|^2}{2}\right) &= \sum_{\norm{\vec{j}}_{\ell^\infty}\leq m} p(|\vec{j}-(\vec{x}+\hat{\vec{x}})+\hat{\vec{x}}|)\exp\left(-\tfrac{q|\vec{j}-(\vec{x}+\hat{\vec{x}})+\hat{\vec{x}}|^2}{2}\right)
				\\&\leq \sum_{\norm{\vec{j}}_{\ell^\infty}\leq2m} p(|\vec{j}+\hat{\vec{x}}|)\exp\left(-\tfrac{q|\vec{j}+\hat{\vec{x}}|^2}{2}\right)
				\\&\leq \sum_{\substack{\vec{j}\in\F Z^d\\\norm{\vec{j}}_{\ell^\infty}\leq2m}} p(|\vec{j}|+\delta)\exp\left(-\tfrac{q\max(0,|\vec{j}|-\delta)^2}{2}\right)
			\end{align*}
			as $|\hat{\vec{x}}|\leq \delta$ and $p$ has non-negative coefficients. 
			\qed\end{proof}
		
		Now, continuing the proof of Theorem~\ref{thm: norm bound examples}, for $\hat m=\sqrt[d]{m}$, $\delta=\frac{\sqrt{d}}{2}$ and $J=\{\vec{j}\in\F Z^d \st \norm{\vec{j}}_{\ell^\infty}\leq 2\hat m\}$, Lemma~\ref{app: exp sum bound} bounds 
		\begin{align*}
			\sum_{j=1}^m |\psi_j(\vec{x})|^{q^*} &\leq (2\pi\sigma^2)^{-\frac{dq^*}{2}} \left[\sum_{\vec{j}\in J} \exp\left(-\frac{q^*\Delta^2}{2\sigma^2}\max(0,|\vec{j}|-\delta)^2\right)\right]
			\\\sum_{j=1}^m |\nabla\psi_j(\vec{x})|^{q^*} &\leq \frac{(2\pi\sigma^2)^{-\frac{dq^*}{2}}}{\sigma^{q^*}}\frac{\Delta^{q^*}}{\sigma^{q^*}} \left[\sum_{\vec{j}\in J} (|\vec{j}|+\delta)^{q^*}\exp\left(-\frac{q^*\Delta^2}{2\sigma^2}\max(0,|\vec{j}|-\delta)^2\right)\right]
			\\\sum_{j=1}^m |\nabla^2\psi_j(\vec{x})|^{q^*} &\leq \frac{(2\pi\sigma^2)^{-\frac{dq^*}{2}}}{\sigma^{2q^*}} \left[\sum_{\vec{j}\in J} \left(1+\frac{\Delta^2}{\sigma^2}(|\vec{j}|+\delta)^2\right)^{q^*}\exp\left(-\frac{q^*\Delta^2}{2\sigma^2}\max(0,|\vec{j}|-\delta)^2\right)\right]
		\end{align*}
		for all $\vec{x}\in\Domain$. In a worst case, this is $O(2^dm)$ time complexity however the summands all decay faster than exponentially and so should converge very quickly.	
		\qed\end{proof}

\end{document}